\pdfoutput=1
\documentclass[10pt]{aart}

\usepackage{arxiv-setting}
\usepackage{booktabs}
\usepackage{placeins}
\usepackage{tabularx}
\usepackage{needspace}

\cslet{blx@noerroretextools}\empty
\usepackage{autonum}

\newcommand{\EnumParameter}{\omega}
\newcommand{\GraphMoment}{\mathfrak M}
\newcommand{\WordClass}[1]{[#1]}
\newcommand{\ReversalClass}[1]{[#1]_{\mathrm{rev}}}
\newcommand{\CoreCode}{%
  \bigl(\gU,\zeta,r,(k_f)_{f\in\gE(\gU)}\bigr)}
\newcommand{\SquarePluralClasses}[1]{\sQ^{#1,{\rm p}}}
\newcommand{\SquareSingletonClasses}[1]{\sQ^{#1,{\rm s}}}
\newcommand{\RectangularPluralClasses}[2]{%
  \widetilde{\sQ}_{#1}^{#2,{\rm p}}}
\newcommand{\RectangularSingletonClasses}[2]{%
  \widetilde{\sQ}_{#1}^{#2,{\rm s}}}
\newcommand{\SquarePluralContribution}{T^{\rm p}}
\newcommand{\SquareSingletonContribution}{T^{\rm s}}
\newcommand{\RectangularPluralContribution}{\widetilde T^{\rm p}}
\newcommand{\RectangularSingletonContribution}{\widetilde T^{\rm s}}
\def\ratio{{\gamma}} % proportional ratio n/m >= 1
\renewcommand{\rowvarmin}{\underline{v}_{\rm r}}
\renewcommand{\rowvarmax}{\overline{v}_{\rm r}}
\renewcommand{\colvarmin}{\underline{v}_{\rm c}}
\renewcommand{\colvarmax}{\overline{v}_{\rm c}}
\renewcommand{\hermvarmin}{\underline{v}}
\renewcommand{\hermvarmax}{\overline{v}}

\title{Concentration of Regularized Sparse Random Matrices:\\
Spectral Edge Bounds via Nonbacktracking Operators}

\author{
Hai-Xiao Wang\orcidlink{0000-0003-2730-1439}%
\thanks{Department of Applied Mathematics, University of Washington, Seattle, WA 98195, USA; \texttt{haixwang@uw.edu}.}
}

\begin{document}

\maketitle

\begin{abstract}

  In sparse random matrices, spectral outliers (eigenvalues and singular values located away from the bulk) emerge due to degree fluctuations: high degrees inflate the operator norm, while low column degrees reduce the least singular value. As proved by Feige and Ofek \cite{feige2005spectral} and Le, Levina, and Vershynin \cite{le2017concentration}, degree regularization enforces concentration at the expected norm scale. However, precise bounds incorporating the cutoffs remain unexplored and challenging since regularization introduces dependencies among entries.

  \smallskip

For the first time in the literature, we provide variance and cutoff dependent bounds for extreme singular values and eigenvalues of regularized inhomogeneous random matrices. In the absence of regularization, our lower bound for the least singular value matches the same leading constant obtained by Brailovskaya and van Handel \cite{brailovskaya2022universality}. Moreover, our error term vanishes under the milder condition $d/\log N\to\infty$, as opposed to their stronger requirement $d/(\log N)^4\to\infty$.
A key ingredient is to extend spectral radius bounds for nonbacktracking matrices to the dependent setting. We build on approaches for independent cases established by Benaych-Georges, Bordenave, and Knowles \cite{benaych2020spectral}, as well as Dumitriu and Zhu \cite{dumitriu2024extreme}, and carefully handle edges traversed only once.
Our proof framework separates deterministic spectral comparisons from probabilistic estimates: once Loewner inequalities and columnwise variance controls are established, the remaining probabilistic analysis boils down to verifying the graph moment conditions formulated in this paper. We hope this framework can be extended to handle general random matrices with more complex dependencies.

\end{abstract}


\newpage
\tableofcontents

%%%%%%%%%%%%%%%%%%%%%%%%%%%%%%%%%%%%%%%%%%%%%%
\newpage
%%%%%%%%%%%%%%%%%%%%%%%%%%%%%%%%%%%%%%%%%%%%%%

\section{Introduction}
\label{sec:introduction}
%%%%%%%%%%%%%%%%%%%%%%%%%%%%%%%%%%%%%%%%%%%%%%

Sparse random matrices arise naturally in many applications, including community detection \cite{abbe2018community}, numerical linear algebra \cite{martinsson2020randomized}, and modern efficient attention mechanisms \cite{liu2025deepseek}. Over the past several decades, understanding the spectral behavior of sparse random matrices, especially the precise locations of the extreme eigenvalues, has attracted significant attention. As a motivating example, consider the Erd\H{o}s--R\'enyi random graph $\gG\sim\mathds{G}(N, d/N)$ on $N$ vertices, where edges are sampled independently with probability $d/N$. Let $\mA$ denote its adjacency matrix without self-loops, and assume $d=o(N)$. When $d/\log(N)\to\infty$, the extreme eigenvalues of the centered adjacency matrix approach the semicircle edges; in particular, $\|\mA-\E\mA\|=(2+o_\P(1))\sqrt d$ \cite{furedi1981eigenvalues, vu2007spectral, benaych2020spectral}. However, when $d/\log(N) \to 0$, high-degree fluctuations cause $\|\mA-\E\mA\|/\sqrt d\to\infty$ in probability, so the centered matrix does not concentrate at the $\sqrt d$ scale \cite{krivelevich2003largest,le2017concentration,benaych2019largest}.
In contrast, the removal of high degree vertices, a.k.a., regularization enforces the concentration of $\|(\mA - \E\mA)_{I}\| = O(\sqrt{d})$ as proved in \cite{feige2005spectral, lei2015consistency, le2017concentration}, where $I$ represents the set of remaining indices supporting the new adjacency matrix. It is worth noting that even though the regularized matrix $\mA_{I}$ concentrates with the optimal rate $\sqrt{d}$, the precise constant factor multiplying $\sqrt{d}$ has never been explicitly characterized.

A related and equally important area in random matrix theory is the study of the least singular values, which plays a crucial role in understanding matrix invertibility \cite{tikhomirov2020singularity} and conditioning \cite{litvak2020small}, with broad implications in numerical linear algebra \cite{martinsson2020randomized} and algorithms analysis \cite{spielman2001smoothed}. In the context of rectangular sparse random matrices, we consider the bipartite Erd\H{o}s--R\'enyi random graph $\widetilde{\gG}\sim\mathds{G}(\gV_{1}, \gV_{2}, d/\sqrt{mn})$, where edges, containing exactly one vertex in $\gV_{1} =[n]$ and another vertex in $\gV_2 =[m]$, are sampled independently with probability $d/\sqrt{mn}$. Let $\widetilde{\mA}\in \{0, 1\}^{n\times m}$ denote the \emph{bi-adjacency} matrix of $\widetilde{\gG}$ with $n\geq m$. We assume $n\geq m$ and focus on the first $m$ nontrivial singular values. As shown in \cite{dumitriu2025singular}, below a certain sparsity, the least of these fails to converge to the left edge of the limiting spectral distribution. This failure is due to the presence of \emph{low degree} vertices, and in particular, the least singular value may actually become zero when isolated vertices emerge, i.e., zero columns\footnote{In the tall matrix considered here, zero rows will not lead to essential zero singular values} in $\widetilde{\mA}$. Consequently, it is natural to consider the effect of regularization on the smallest singular values. This raises the question of whether the removal of vertices with extreme degrees can likewise recover concentration at the left and right edge of the spectrum. To the best of our knowledge, the effect of regularization on singular values has not yet been studied.

In this paper, we study spectral edges of mean-field type sparse random matrices after regularization with pre-determined thresholds. For the first time in the literature, we establish the explicit relationship between cutoff thresholds and extreme eigenvalues/singular values. Our contributions can be summarized as follows:

\begin{enumerate}
    \item \emph{Least singular value.}
    When the lowest expected degree is at least $\log(N)$ and the lower column cutoff is properly controlled, ensuring a non-vanishing fraction of columns remaining, we establish lower bounds for least singular values (Theorem~\ref{thm:rectangular-smallest-singular-value}), relying on both lower cutoff threshold and minimum column variance.    
    When no thresholding is applied, the leading contant of our lower bound matches the one previously established by Brailovskaya and van Handel \cite{brailovskaya2022universality} for the regime $d/(\log(N))^4\to\infty$. In contrast, our error term vanishes when $d/\log(N)\to\infty$, extending the regime where this tight lower bound remains valid. See Remark~\ref{rem:rectangular-smallest-singular-value-comparison}.

    \item \emph{Largest singular values and eigenvalues.}
    We provide upper bounds for the largest singular values (Theorem~\ref{thm:rectangular-largest-singular-value}) and eigenvalues (Theorem~\ref{thm:hermitian-edge}), which depend on both the upper cutoffs and the maximal row and column variance profiles. Whereas estimating the left edge necessitates the lowest expected degree to be at least $\log(N)$, upper bounds for the right edge remain valid as long as the highest expected degree exceeds some sufficiently large constant.

    \item \emph{Spectral radii of nonbacktracking matrices.} For nonbacktracking matrices constructed from matrices with dependent entries, we provide upper bounds for their spectral radii (Proposition~\ref{prop:nonbacktracking-radius-bounds}) via the moment method (Lemma~\ref{lem:general-nonbacktracking-trace-bounds}), extending previous analysis in \cite{benaych2020spectral} and \cite{dumitriu2024extreme} where independent entries are required. Our analysis show that traces for powers of nonbacktracking matrices are upper bounded, as long as the entries of the base matrix satisfy certain graph moment conditions in the Subsection~\ref{subsec:moment-assumptions}, where independence is not mandatory. See Section~\ref{sec:trace-estimates} for details.
    
    \item \emph{Loewner inequalities and edge bounds.}
    We derive deterministic Loewner inequalities from the Ihara--Bass formulas, combining nonbacktracking radius bounds with entrywise and separate row and column variance bounds to control extreme eigenvalues and singular values. In particular, Proposition~\ref{prop:rectangular-lower-quadratic-form} provides a weighted Gram matrix comparison that sharpens the least singular value bound, and generalizes the previous established result in \cite{dumitriu2025singular} for homogeneous entries.

    \item \emph{Vertex retention guarantees.}
    We quantify deletion cascades ( Remark~\ref{remark:deletion-cascade}) and retained dimensions, and
    establish near-full retention in the stated upper-cutoff and strong-band regimes (Section~\ref{sec:regularization-thresholds}). These guarantees ensure that the spectral bounds apply to submatrices of nearly the original dimensions, with explicit control on the proportion of vertices removed by regularization.

    \item \emph{A general framework for spectral edges.}
    We extend the framework that separates deterministic spectral comparisons from probabilistic input, in the vein of \cite{benaych2020spectral} and \cite{dumitriu2024extreme}. With the Loewner inequalities and the required entrywise and row/column variance bounds in place, the remaining probabilistic task (Section~\ref{sec:model-moment-verification}) is to verify the graph moment assumptions listed in Subsection~\ref{subsec:moment-assumptions}. We hope this framework can be adapted to a broader class of random matrix models, including those with dependent entries.
\end{enumerate}
Finally, we would like to make a remark on the effectiveness of the nonbacktracking operator method. As will be demonstrated at the beginning of Section~\ref{sec:trace-estimates}, this approach is particularly well-suited for handling lower sparsity regimes, and sometimes it has been used in ways that sacraficed the accuracy to achieve near critical sparsity 
\cite{dumitriu2024extreme}, perhaps 
unachievable by other methods. Here, we show that, when applied in full strength, the nonbacktracking operator method is not only capable of matching the accuracy of alternative approaches, but it can also be used to reach optimal sparsity thresholds.

\subsection{Sparse random matrices and degree regularization}
\label{subsec:bounded-row-and-column-sums}
We first introduce the sparse random matrix model that will be used throughout the paper, and an algorithm to obtain the regularized matrix through the row-column projection.
\begin{definition}[Sparse random matrix]
\label{def:sparse-random-matrix}
We say that a random matrix $\mX \in \C^{n\times m}$ has \emph{sparsity scale} $d >0$ if each of its entries $\emX_{ij}$, $i\in[n], j\in[m]$, satisfies
\begin{align}
 |\emX_{ij}|&\leq C d^{-1/2}\quad\text{almost surely},
 &\E|\emX_{ij}|&\leq\frac{C\sqrt d}{N},
 &\E|\emX_{ij}|^k&\leq\frac{C}{Nd^{(k-2)/2}}
       \quad(k\geq2),
 \label{eqn:sparse-variance-budgets}
\end{align}
where $C>0$ is some fixed constant that does not rely on $n$, $m$, $d$, and $k$.
\end{definition}
Note that the model above permits dependencies between entries and does not assume that the entries are centered or identically distributed. In addition, $\sqrt{d}|\emX_{ij}|$ are bounded univerally. We further denote the extremal second-moment sums of $\mX$ by
\begin{subequations}
\begin{align}
 \rowvarmin&\coloneqq\min_{i\in[n]}
       \sum_{j=1}^m\E|\emX_{ij}|^2,
 &\colvarmin&\coloneqq\min_{j\in[m]}
       \sum_{i=1}^n\E|\emX_{ij}|^2,
 \label{eqn:row-column-var-min}\\
 \rowvarmax&\coloneqq\max_{i\in[n]}
       \sum_{j=1}^m\E|\emX_{ij}|^2,
 &\colvarmax&\coloneqq\max_{j\in[m]}
       \sum_{i=1}^n\E|\emX_{ij}|^2.
 \label{eqn:row-column-var-max}
\end{align}
\end{subequations}
We use $\sigma_{\min}(\mX)$ and $\sigma_{\max}(\mX)$ to denote the smallest and largest singular values of $\mX$, respectively. For any $\mX \in \C^{n\times m}$, we define its Hermitian linearization as
\begin{align}
 \widetilde{\mH}\coloneqq
 \begin{bmatrix}
  \bzero_{n\times n}&\mX\\
  \mX^*&\bzero_{m\times m}
 \end{bmatrix}\in\C^{(n+m)\times (n+m)},
 \label{eqn:block-linearization}
\end{align}
which satisfies $\|\widetilde{\mH}\|=\|\mX\|$. Throughout the paper, we work under the proportional dimension regime.
\begin{assumption}
\label{assumption:proportional-dimension-regime}
There is some constant $\Gamma > 0$ such that
      \begin{align}
     1 \leq  \ratio \coloneqq n/m \leq \Gamma,
      \label{eqn:proportional-block-regime}
     \end{align}
\end{assumption}

\paragraph{Projection regularization.}
We now introduce an algorithm to obtain the regularized matrix through the row-column projection. Let $\mX_{i\bullet}$ and $\mX_{\bullet j}$ denote the $i$-th row and $j$-th column of $\mX$, respectively. Recall that the $\ell_0$ norm of a vector $\vx \in \C^n$ is defined as $\|\vx\|_0=|\{j:\evx_j\ne0\}|$, counting the number of non-zero entries. For fixed integer bands $0\leq L_{\rm r}\leq U_{\rm r}\leq m$ and $0\leq L_{\rm c}\leq U_{\rm c}\leq n$, define the support events by
\begin{align}
 \sR_i(\mX)
 &=\{L_{\rm r}\leq\|\mX_{i\bullet}\|_0\leq U_{\rm r}\},
 &\sC_j(\mX)
 &=\{L_{\rm c}\leq\|\mX_{\bullet j}\|_0\leq U_{\rm c}\}.
 \label{eqn:row-column-support-events}
\end{align}
Algorithm~\ref{alg:row-column-projection} iteratively eliminates rows and columns that do not meet the criteria, continuing this process until the sets of active rows and columns remain static.

\begin{algorithm}[h]
 \caption{Row--column projection regularization of $\mX$}
 \label{alg:row-column-projection}
 \begin{algorithmic}[1]
  \Require $\mX\in\C^{n\times m}$ and the bands in
     \eqref{eqn:row-column-support-events}
  \Ensure Active sets $I_\star,J_\star$, projections
     $\mP_{\rm r} \in \C^{n\times n}, \mP_{\rm c} \in \C^{m\times m}$, and active submatrix $\mX_{\rm reg} \in \C^{|I_\star|\times|J_\star|}$
  \State $I_0\gets[n]$, $J_0\gets[m]$, and $t\gets0$
  \Repeat
   \State $\mP_{\rm r}^{(t)}\gets
    \diag((\indi{i\in I_t})_{i\in[n]})$,\quad
    $\mP_{\rm c}^{(t)}\gets
    \diag((\indi{j\in J_t})_{j\in[m]})$
   \State $\mX^{(t)}\gets
    \mP_{\rm r}^{(t)}\mX\mP_{\rm c}^{(t)}$
   \State $I_{t+1}\gets\{i\in I_t:\sR_i(\mX^{(t)})\}$,\quad
    $J_{t+1}\gets\{j\in J_t:\sC_j(\mX^{(t)})\}$
   \State $t\gets t+1$
  \Until{$I_t=I_{t-1}$ and $J_t=J_{t-1}$}
  \State $I_\star\gets I_t$, $J_\star\gets J_t$
  \State $\mP_{\rm r}\gets
   \diag((\indi{i\in I_\star})_{i\in[n]})$,\quad
   $\mP_{\rm c}\gets
   \diag((\indi{j\in J_\star})_{j\in[m]})$
  \State $\mX_{\rm reg}\gets\mX_{I_\star,J_\star}$
 \end{algorithmic}
\end{algorithm}

\begin{remark}\label{remark:deletion-cascade}
Positive lower cutoffs can cause a deletion cascade, as each removal reduces the remaining support sizes and may trigger further deletions, potentially leaving $I_\star=J_\star=\varnothing$.
\end{remark}

\subsection{Extreme singular values bounds for regularized rectangular matrices}
\label{sec:bounds-on-largest-and-smallest-singular-values-of-sparse-rectangular-random-matrices}

Let $\widetilde{\gG}\sim\mathds{G}(n,m,\{p_{ij}\})$ be an
inhomogeneous bipartite Erd\H{o}s--R\'enyi graph with $N=n+m$ and vertex classes $\gV_1=[n]$, $\gV_2=[N]\setminus[n]$, which, for convenience, identified as $[m]$. Let $\widetilde{\mA}$ denote the \emph{biadjacency} matrix of $\widetilde{\gG}$, i.e., $\widetilde{\mA}\in\{0,1\}^{n\times m}$ with independent entries $\widetilde{\emA}_{ij}\sim\Ber(p_{ij})$.
For each row $i\in[n]$ and column $j\in[m]$, we denote their expected degrees by $\mu_i^{\rm r}=\sum_jp_{ij}$ and $\mu_j^{\rm c}=\sum_ip_{ij}$. The minimum and maximum expected degrees for each row and column are given by
\begin{align}
  \underline{\drow}\coloneqq\min_i\mu_i^{\rm r},\qquad
  \underline{\dcol}\coloneqq\min_j\mu_j^{\rm c}, \qquad
 \overline{\drow}\coloneqq\max_i\mu_i^{\rm r},\qquad
 \overline{\dcol}\coloneqq\max_j\mu_j^{\rm c}.
\end{align}
We define the sparsity scale as the geometric mean of the minimum row and column expected degrees, formally
\begin{align}
 d\coloneqq
   \bigl(\underline{\drow}\,\underline{\dcol}\bigr)^{1/2}>0.
 \label{eqn:bipartite-normalization}
\end{align}
Throughout the subsection, we work under the propositional dimension regime \eqref{eqn:proportional-block-regime}. We further introduce the uniform sparsity assumption.

\begin{assumption}[Uniform rectangular sparsity]
\label{assumption:rectangular-uniform-sparsity}
For a fixed constant $C_{\rm sp}$,
\begin{align}
 p_{\star} \coloneqq \max_{i,j}p_{ij}\leq C_{\rm sp}d/\sqrt{nm}.
 \label{eqn:rectangular-uniform-sparsity-bound}
\end{align}
\end{assumption}
\paragraph{Projection and centering.}
We apply Algorithm~\ref{alg:row-column-projection} to $\widetilde{\mA}$ with the prescribed conditions in \eqref{eqn:row-column-support-events}, obtaining $I_\star,J_\star$ and $\mP_{\rm r},\mP_{\rm c}$. Define the centered matrix and the active submatrices by
\begin{align}
 \mX&\coloneqq(\widetilde{\mA}-\E\widetilde{\mA})/\sqrt d,
 &\mX^{(\tau)}&\coloneqq\mX_{I_\star,J_\star},
 &\widetilde{\mA}^{(\tau)}&\coloneqq\widetilde{\mA}_{I_\star,J_\star}.
 \label{eqn:compressed-rectangular-matrix}
\end{align}
Both active matrices have size $|I_\star|\times|J_\star|$.
For the centered matrix $\mX$, using notations in \eqref{eqn:row-column-var-min}--\eqref{eqn:row-column-var-max}, the second-moment summations are
\begin{subequations}
\begin{align}
 \rowvarmin&=d^{-1}\min_i\sum_jp_{ij}(1-p_{ij}),
 &\colvarmin&=d^{-1}\min_j\sum_ip_{ij}(1-p_{ij}),
 \label{eqn:rectangular-variance-lower-bounds}\\
 \rowvarmax&=d^{-1}\max_i\sum_jp_{ij}(1-p_{ij}),
 &\colvarmax&=d^{-1}\max_j\sum_ip_{ij}(1-p_{ij}).
 \label{eqn:rectangular-variance-upper-bounds}
\end{align}
\end{subequations}

Below, we will provide bounds for the largest and smallest singular values of the active submatrix $\mX^{(\tau)}$, separately. We first define the cutoff errors by
\begin{align}
 \varepsilon_N&\coloneqq\min\big\{
 ([U_{\rm r}-\overline{\drow}]_+ +[U_{\rm c}-\overline{\dcol}]_+)/{d},\quad
 \sqrt{\log(N)/{d}}\,\big\}.
 \label{eqn:rectangular-cutoff-error-scale}
\end{align}
We then introduce the normalized cutoff-dependent variance parameters by
\begin{align}
 \overline{\rho_{\rm r}}^{2}
 &\coloneqq\min\left\{ U_{\rm r}/{d},\,\, \rowvarmax\right\},
 &\overline{\rho_{\rm c}}^{2}
 &\coloneqq\min\left\{ U_{\rm c}/{d},\,\,\colvarmax\right\},
 &\underline{\rho_{\rm c}}^{2}
 &\coloneqq\max\left\{ (1-p_\star)^2L_{\rm c}/{d},\,\,
                          \colvarmin\right\}.
\label{eqn:cutoff-block-budgets}
\end{align}
We further define the right and left singular value thresholds by
\begin{align}
      \sigma_{\rm L}&\coloneqq\sqrt{\colvarmin}-\sqrt{\rowvarmax}, 
      &
 \sigma_{\rm R}&\coloneqq\inf_{t\geq(\rowvarmax\colvarmax)^{1/4}}
 \sqrt{t^2+\overline{\rho_{\rm r}}^{2}+\overline{\rho_{\rm c}}^{2}
       +\overline{\rho_{\rm r}}^{2}\overline{\rho_{\rm c}}^{2}/t^2},
 \label{eqn:rectangular-explicit-upper-threshold}
\end{align}

\begin{assumption}[Right edge sparsity]\label{ass:right-edge-sparsity}
Fix $0<\delta_{\rm NB}<1$. Let $C_{\rm NB}>1$ be some sufficiently large constant, which depends only on the fixed profile and aspect-ratio constant $\Gamma$ from \eqref{eqn:proportional-block-regime}, such that
\begin{align}
 d\geq C_{\rm NB},\qquad
 d^{\,6}(\log(N))^3\leq N^{1-\delta_{\rm NB}},
\label{eqn:rectangular-edge-sparsity-range}
\end{align}
\end{assumption}
For regularization conditions, we allow any integer cutoffs satisfying
\begin{align}
 L_{\rm r}=0,\qquad 0\leq U_{\rm r}\leq m,\qquad
 0\leq L_{\rm c}\leq U_{\rm c}\leq n.
 \label{eqn:rectangular-edge-cutoff-range}
\end{align}
\begin{theorem}[Largest singular value]
\label{thm:rectangular-largest-singular-value}
Suppose the sparsity Assumptions~\ref{assumption:rectangular-uniform-sparsity}, ~\ref{ass:right-edge-sparsity}, and the cutoff conditions \eqref{eqn:rectangular-edge-cutoff-range}. Recall $\sigma_{\rm R}$ in \eqref{eqn:rectangular-explicit-upper-threshold} and $\varepsilon_N$ in \eqref{eqn:rectangular-cutoff-error-scale}. 
There exists a constant $K>0$ such that, for all sufficiently large $N$,
with probability at least $1-N^{-1}$,
\begin{align}
 \|\mX^{(\tau)}\|
 &\leq\sigma_{\rm R}+K(\varepsilon_N+d^{-1/2}).
\label{eqn:largest-singular-edge}
\end{align}
The operator norm of an empty restriction is taken to be zero.
The constant $K$ depends only on the fixed hypothesis constants, uniformly over the admissible cutoffs.
\end{theorem}

We now introduce our lower bound for least singular values.
\begin{assumption}[Left edge sparsity]\label{ass:left-edge-sparsity}
      In addition to \eqref{eqn:rectangular-edge-sparsity-range}, we assume that
\begin{align}
 d\geq c_{\rm deg}\log(N),
 \label{eqn:rectangular-lower-edge-sparsity}
\end{align}
where $c_{\rm deg}>0$ is some sufficiently large constant.
\end{assumption}
For degree regularization, the integer cutoffs in \eqref{eqn:rectangular-edge-cutoff-range} must satisfy both the upper separation conditions
\begin{align}
 5\sqrt{C_0d\log d}
 &\leq U_{\rm r}-\overline{\drow},
 &5\sqrt{C_0d\log d}
 &\leq U_{\rm c}-\overline{\dcol},
 \label{eqn:rectangular-edge-upper-cutoffs}
\end{align}
and, for a fixed $C_{\rm cut}>0$, the lower column conditions
\begin{align}
 (1-p_\star)^2L_{\rm c}/d
 &\geq\colvarmin-C_{\rm cut}\sqrt{\log(d)/d},
 &L_{\rm c}&\leq\underline{\dcol}.
\label{eqn:rectangular-edge-lower-cutoff}
\end{align}
The lower bound on $L_{\rm c}$ controls the minimum column support,
while $L_{\rm c}\leq\underline{\dcol}$ ensures that a positive
fraction of columns survives, as required in
Remark~\ref{remark:deletion-cascade}.

\begin{theorem}[Least singular value]
\label{thm:rectangular-smallest-singular-value}
Suppose Assumptions~\ref{assumption:rectangular-uniform-sparsity} and~\ref{ass:left-edge-sparsity} hold. For integer cutoffs satisfying \eqref{eqn:rectangular-edge-cutoff-range}, \eqref{eqn:rectangular-edge-upper-cutoffs}, and~\eqref{eqn:rectangular-edge-lower-cutoff}, there are constants $K,C,c>0$ such that the following holds with probability at least $1-Cd^{-3}-e^{-cm}-N^{-1}$ for all sufficiently large $N$: both active sets are nonempty, $|J_\star|\geq cm$, and
\begin{align}
 \sigma_{\min}(\mX^{(\tau)})
 &\geq\left[\sqrt{\colvarmin}-\sqrt{\rowvarmax}-K(\varepsilon_N+d^{-1/2})\right]_+.
 \label{eqn:smallest-singular-edge}
\end{align}
Separately, under the same sparsity assumptions alone, the unthresholded case $I_\star=[n]$, $J_\star=[m]$ satisfies \eqref{eqn:smallest-singular-edge} with $\mX^{(\tau)}=\mX$ and $\varepsilon_N=\sqrt{\log(N)/d}$, with probability at least $1-N^{-1}$ for all sufficiently large $N$.
The constant $K$ depends only on the fixed hypothesis constants, uniformly over the admissible cutoffs in the regularized case.
\end{theorem}

\begin{remark}[Comparison with previous results]\label{rem:rectangular-smallest-singular-value-comparison} When no thresholding, $\varepsilon_N = \sqrt{\log(N)/d}$ in \eqref{eqn:smallest-singular-edge}, which vanishes when $d\gg\log(N)$. The leading term matches \cite[Corollary~3.18]{brailovskaya2022universality}, for which the normalized error $O(d^{-1/6}(\log(N))^{2/3})$ becomes negligible when $d\gg(\log(N))^4$. Compared with \cite[Theorem~2.4]{dumitriu2024extreme}, our result improve their leading term
\(
\big((1-\sqrt{\rowvarmax/\colvarmax})             (\colvarmin-\sqrt{\rowvarmax\,\colvarmax})\big)^{1/2}
\).
The improvement stems from the use of column rescaling as will be introduced in \eqref{eqn:rescaled-column-matrix}, the spectral radius bound for the nonbacktracking matrix of the rescaled matrix in Proposition~\ref{prop:nonbacktracking-radius-bounds}~\ref{item:nonbacktracking-rescaled}, and a refined deterministic comparison for singular values provided in Lemma~\ref{lem:rectangular-singular-value-comparisons}\textup{(ii)}, which is based on the Ihara--Bass formula. Together, these tools address the dependencies created by regularization and yield tighter quantitative error bounds.
\end{remark}

Regarding the sparsity assumption, both right and left edge bounds must satisfy the restriction $d^6(\log(N))^3\leq N^{1-\delta_{\rm NB}}$ in \eqref{eqn:rectangular-edge-sparsity-range}, which arises from the requirement in nonbacktracking trace estimates. For the right edge, it suffices to have constant sparsity $d\geq C_{\rm NB}$. In contrast, the left edge requires a logarithmic lower bound on $d$ to ensure that very low degree vertices are sufficiently rare. This prevents excessive vertex removal and avoids the degree-deletion cascade that could occur if the lower cutoff were set too high.

\subsection{Extreme eigenvalues of regularized Hermitian matrices}
\label{sec:bounds-on-eigenvalues-of-sparse-hermitian-random-matrices}
Let $\gG\sim\mathds{G}(N,\{p_{uv}\})$ be an inhomogeneous Erd\H{o}s--R\'enyi graph with $p_{uu}=0$ and $p_{uv}=p_{vu}$. Its adjacency matrix $\mA\in\{0,1\}^{N\times N}$ is symmetric, with zero diagonal and independent entries $\emA_{uv}\sim\Ber(p_{uv})$ for $u<v$.
For each vertex $i\in[N]$, we let $\mu_i=\sum_{j} p_{ij}$ denote its expected degree. Let the maximum expected degree denote the sparsity scale:
\begin{align}
 d\coloneqq \overline{d} &\coloneqq\max_i\mu_i >0.
 \label{eqn:hermitian-sparsity-scale}
\end{align}

\begin{assumption}[Uniform Hermitian sparsity]
\label{ass:uniform-Hermitian-sparsity}
For a fixed constant $C_{\rm sp}>0$,
\begin{align}
 p_\star\coloneqq\max_{u,v}p_{uv}\leq C_{\rm sp}d/N.
 \label{eqn:uniform-sparsity-bound}
\end{align}
\end{assumption}

\paragraph{Projection and centering.}
Apply Algorithm~\ref{alg:row-column-projection} to $\mA$ with
the same integer band $[L,U]$, where $0\leq L\leq U\leq N-1$,
for rows and columns.  Symmetry gives $I_\star=J_\star$ and
$\mP_{\rm r}=\mP_{\rm c}=\mP$.  Similarly, define
\begin{align}
 \mH&\coloneqq(\mA-\E\mA)/\sqrt d,
 &\mH^{(\tau)}&\coloneqq\mH_{I_\star,I_\star},
 &\mA^{(\tau)}&\coloneqq\mA_{I_\star,I_\star}.
 \label{eqn:compressed-hermitian-matrix}
\end{align}
Both active matrices have size $|I_\star|\times|I_\star|$, and
the active adjacency matrix has all degrees in $[L,U]$.
The extremal second-moment sums of $\mH$ are
\begin{align}
 \hermvarmin&\coloneqq d^{-1}\min_u\sum_vp_{uv}(1-p_{uv}),
 &\hermvarmax&\coloneqq d^{-1}\max_u\sum_vp_{uv}(1-p_{uv}).
 \label{eqn:hermitian-variance-extrema}
\end{align}

We control both eigenvalue edges through the spectral norm of
$\mH^{(\tau)}$.  Define the variance error by
\begin{align}
 \varepsilon_N&\coloneqq
       \min\big\{[U-d]_+/d,\,\sqrt{\log(N)/d}\big\},
 \label{eqn:hermitian-cutoff-error-scale}
\end{align}
and the normalized cutoff-dependent variance parameter and
eigenvalue threshold by
\begin{align}
 \overline\rho^{\,2}&\coloneqq\min\{U/d,\,\hermvarmax\},
 &\lambda_{\rm H}&\coloneqq
       \inf_{t\geq1}\left(t+\overline\rho^{\,2}/t\right)
       =1+\overline\rho^{\,2}.
 \label{eqn:hermitian-norm-budget}
\end{align}

\begin{theorem}[Spectral norm]
\label{thm:hermitian-edge}
Suppose Assumptions~\ref{ass:right-edge-sparsity} and~\ref{ass:uniform-Hermitian-sparsity} hold, and let $0\leq L\leq U\leq N-1$ be integer cutoffs.
Recall $\lambda_{\rm H}$ in \eqref{eqn:hermitian-norm-budget} and $\varepsilon_N$ in \eqref{eqn:hermitian-cutoff-error-scale}. There is a constant $K>0$ such that, for all sufficiently large $N$, with probability at least $1-N^{-1}$,
\begin{align}
 \|\mH^{(\tau)}\|
 &\leq\lambda_{\rm H}+K(\varepsilon_N+d^{-1/2}).
 \label{eqn:hermitian-edge-probability}
\end{align}
The constant $K$ depends only on the fixed hypothesis constants, uniformly over the admissible cutoffs.
\end{theorem}
\begin{remark}
While \cite{le2017concentration} established the optimal concentration rate $\|(\mA - \E\mA)_{I_\star}\| = O(\sqrt{d})$ up to a constant factor, our result specifies the exact leading constant $\lambda_{\rm H}$ in the spectral norm bound.
\end{remark}

\subsection{Spectral radius of nonbacktracking matrices}
\label{subsec:nonbacktracking-radii}
An essential probabilistic step in proving Theorems~\ref{thm:rectangular-largest-singular-value}, \ref{thm:rectangular-smallest-singular-value}, and~\ref{thm:hermitian-edge} is to derive bounds on the spectral radius of the nonbacktracking matrices. Below, we first introduce the nonbacktracking matrices built from square and rectangular matrices, respectively. We then bound their radii for the bipartite degree-projected matrix, the rescaled matrix, and the original Hermitian matrix. Let $\gK_N$ be the complete host including loops and
$\gK_{n,m}$ the complete bipartite host on $\gV_1\sqcup\gV_2$.
Define their ordered edge sets by
\begin{align}
 \vec{\gE}(\gK_N)
 &\coloneqq[N]^2,
 &
 \vec{\gE}(\gK_{n,m})
 &\coloneqq
 (\gV_1\times\gV_2)\sqcup(\gV_2\times\gV_1).
 \label{eqn:ordered-host-edges}
\end{align}
\begin{definition}[Nonbacktracking matrices]
\label{def:nonbacktracking-matrix}
For $\mH\in\C^{N\times N}$, define
$\mB=\mB(\mH)\in\C^{N^2\times N^2}$ by
\begin{align}
 \emB_{ef}
 &\coloneqq\emH_{kl}\indi{j=k}\indi{i\neq l},
 &e=(i,j),\quad f=(k,l)&\in\vec{\gE}(\gK_N).
 \label{eqn:nonbacktracking-definition-hermitian}
\end{align}
For $\mX\in\C^{n\times m}$ with block linearization
$\widetilde{\mH}$ from \eqref{eqn:block-linearization}, define
$\widetilde{\mB}=\widetilde{\mB}(\mX)\in\C^{2nm\times2nm}$ by
\begin{align}
 \widetilde{\emB}_{ef}
 &\coloneqq\widetilde{\emH}_{kl}\indi{j=k}\indi{i\neq l},
 &e=(i,j),\quad f=(k,l)&\in\vec{\gE}(\gK_{n,m}).
 \label{eqn:nonbacktracking-definition-bipartite}
\end{align}
Each transition joins consecutive oriented edges without an immediate reversal, with weight given by the destination edge.  In the bipartite case, the two orientations use entries of $\mX$ and $\mX^*$.
\end{definition}
\begin{remark}
\label{rem:nonbacktracking-definition-bipartite-block-linearization}
Given the block linearization $\widetilde{\mH}$ of $\mX$, and using the construction in \eqref{def:nonbacktracking-matrix}, if we order the crossing host edges first, we obtain
\begin{align}
 \mB(\widetilde{\mH})
 =\begin{bmatrix}
  \widetilde{\mB}&\bzero\\
  \ast&\bzero
 \end{bmatrix} \in \C^{N^2\times N^2}.
\end{align}
Hence $\mB(\widetilde{\mH})$ and $\widetilde{\mB}$ share the same non-zero spectrum.
\end{remark}

After regularization, the dimensions of the active submatrices $\mX^{(\tau)}$ and $\mH^{(\tau)}$ are random. To construct nonbacktracking matrices, we embed the active matrices in their original dimensions, with zero rows and columns outside the active sets. Recall $\mP_{\rm r}$ and $\mP_{\rm c}$ from Algorithm~\ref{alg:row-column-projection}, we define the padded matrices by
\begin{align}
 \widehat{\mX}&\coloneqq\mP_{\rm r}\mX\mP_{\rm c}
       \in\R^{n\times m},
 &\widehat{\mH}&\coloneqq\mP\mH\mP
       \in\R^{N\times N}
 \label{eqn:padded-regularized-matrices}
\end{align}
Their nonbacktracking matrices act on the fixed host edge sets in \eqref{eqn:ordered-host-edges}.

We then introduce the \emph{rescaled} matrix $\mZ$, which retains rows that satisfy the original upper-degree cutoff and caps each column's expected variance at $\colvarmin$. Precisely, denote the row degree $\rD_i^{\rm r}=\sum_j\widetilde{\emA}_{ij}$ and the indicator function $U_i=\indi{\rD_i^{\rm r}\leq U_{\rm r}}$, and we define the rescaling factor by
\[
 Q_j=\sum_iU_i|\emX_{ij}|^2,\qquad
 f_j=\min \Big\{1,\sqrt{\colvarmin/Q_j}\Big\}, \text{ and } f_j=1 \text{ if } Q_j=0.
\]
Then we define the row projection $\mP_{\rm r}^{\rm u}$, the column rescaling $\mP_{\rm c}^{\rm a}$, and the rescaled matrix $\mZ$ by
\begin{align}
 \mZ&=\mP_{\rm r}^{\rm u}\mX\mP_{\rm c}^{\rm a},
 &\mP_{\rm r}^{\rm u}&=\diag((U_i)_{i\in[n]}),
 &\mP_{\rm c}^{\rm a}&=\diag((f_j)_{j\in[m]})
 \label{eqn:rescaled-column-matrix}
\end{align}
For each $j \in [m]$, the matrix $\mZ$ satisfies $\sum_i|\emZ_{ij}|^2 = \min\{Q_j,\,\colvarmin\}$.
As noted in Remark~\ref{rem:rectangular-smallest-singular-value-comparison}, our improved leading term in Theorem~\ref{thm:rectangular-smallest-singular-value} compared to \cite[Theorem~2.4]{dumitriu2024extreme} is due to the application of column rescaling. Importantly, reducing column variances to the minimum value $\colvarmin$ maintains the direction required to transfer lower bounds back to $\mX$. Further details on this point can be found in the proof of Theorem~\ref{thm:rectangular-smallest-singular-value}.

We now illustrate the nonbacktracking radius bounds under three different scenarios.

\begin{proposition}[Nonbacktracking radius bounds]
\label{prop:nonbacktracking-radius-bounds}
Suppose the sparsity conditions \eqref{eqn:rectangular-edge-sparsity-range} hold, with $N$ and $d$ defined for the relevant model. For every fixed $\nu>0$, under the respective assumptions below, each bound holds with probability at least $1-N^{-\nu}$ for all sufficiently large $N$. The constant $C_\nu>0$ depends only on $\nu$ and the fixed hypothesis constants.
\begin{enumerate}[label=\textup{(\roman*)},ref=\textup{(\roman*)},leftmargin=2.5em,itemsep=0.6em]
\item \emph{Bipartite model.}
\label{item:nonbacktracking-bipartite}
Suppose Assumptions~\ref{assumption:proportional-dimension-regime} and~\ref{assumption:rectangular-uniform-sparsity}, and the upper cutoff conditions \eqref{eqn:rectangular-edge-upper-cutoffs}, with lower cutoffs $L_{\rm r}=L_{\rm c}=0$. For the padded matrix $\widehat{\mX}$ in \eqref{eqn:padded-regularized-matrices}, the spectral radius of its nonbacktracking matrix is bounded by
\begin{align}
 \rho(\widetilde{\mB}(\widehat{\mX}))
 &\leq(\rowvarmax\colvarmax)^{1/4}+C_\nu d^{-1/2}.
 \label{eqn:bipartite-radius-bound}
\end{align}

\item \emph{Rescaled bipartite model.}
\label{item:nonbacktracking-rescaled}
Suppose Assumptions~\ref{assumption:proportional-dimension-regime} and~\ref{assumption:rectangular-uniform-sparsity}, the upper cutoff conditions \eqref{eqn:rectangular-edge-upper-cutoffs}, the lower sparsity bound \eqref{eqn:rectangular-lower-edge-sparsity}, and $\rowvarmax<\colvarmin$, with variance sums taken for the original matrix $\mX$. For the rescaled matrix $\mZ$ in \eqref{eqn:rescaled-column-matrix}, the spectral radius of its nonbacktracking matrix is bounded by
\begin{align}
 \rho(\widetilde{\mB}(\mZ))
 &\leq(\rowvarmax\colvarmin)^{1/4}
       +C_\nu\bigl(\sqrt{\log d/d}+d^{-1/2}\bigr).
 \label{eqn:rescaled-column-radius-bound}
\end{align}

\item \emph{Hermitian matrix.}
\label{item:nonbacktracking-hermitian}
Suppose Assumption~\ref{ass:uniform-Hermitian-sparsity}. For the original centered matrix $\mH$, the spectral radius of its nonbacktracking matrix is bounded by
\begin{align}
 \rho(\mB(\mH))\leq1+C_\nu d^{-1/2}.
 \label{eqn:hermitian-radius-bound}
\end{align}
\end{enumerate}
\end{proposition}

\begin{remark}[Active matrices and nonbacktracking radius bounds]
\label{rem:cutoff-active-padded-radius}
For every realization, zero padding preserves the nonzero
nonbacktracking spectrum and the following equalities hold:
\begin{align}
 \rho(\widetilde{\mB}(\widehat{\mX}))
 &=\rho(\widetilde{\mB}(\mX^{(\tau)})),
 &\rho(\mB(\widehat{\mH}))
 &=\rho(\mB(\mH^{(\tau)})).
 \label{eqn:cutoff-active-padded-radius}
\end{align}
Thus part~\ref{item:nonbacktracking-bipartite} applies to
$\mX^{(\tau)}$ without conditioning on the active sets.
Without thresholding, it recovers the independent bound for $\mX$ in
\cite[Theorem~4.1]{dumitriu2024extreme}.
The rescaling in part~\ref{item:nonbacktracking-rescaled} caps every column variance of $\mZ$ at $\colvarmin$. Under its additional hypotheses, this replaces $\colvarmax$ by $\colvarmin$ in the leading radius bound, from $(\rowvarmax\colvarmax)^{1/4}$ to $(\rowvarmax\colvarmin)^{1/4}$, at the cost of an additional
$O(\sqrt{\log d/d})$ error. This sharper profile dependence is the key input for the leading term $\sqrt{\colvarmin}-\sqrt{\rowvarmax}$ in the smallest-singular-value bound. The same column rescaling also gives the unthresholded lower bound,
since full upper caps satisfy the cutoff conditions in
part~\ref{item:nonbacktracking-rescaled}.
\end{remark}

\subsection{Related literature}

\paragraph{Edge universality for classical random matrices}
For Wigner matrices, the work of F\"uredi and Koml\'os \cite{furedi1981eigenvalues} and Bai and Yin \cite{bai1988necessary} established the necessary and sufficient conditions for the convergence of largest eigenvalue to the semicircle edge. The distribution function of the largest eigenvalue was first identified by Tracy and Widom for Gaussian ensembles \cite{tracy1994level,tracy1996orthogonal}, later extended to Wigner matrices \cite{soshnikov1999universality, ruzmaikina2006universalityot, tao2010random}. A necessary and sufficient condition for Tracy-Widom law of Wigner matrices was established by Lee and Yin \cite{lee2014necessary}.
For rectangular matrices, Bai and Yin \cite{yin1988limit,bai1993limit} identify the convergence of largest and smallest singular values to the right and left edges of the Marcenko-Pastur distribution under finite fourth moment assumptions. Later, Tikhomirov \cite{tikhomirov2015limit} reduced the requirement for the convergence of smallest singular value to finite variance. Chafa{\"i} and Tikhomirov \cite{chafai2018convergence} extended the Bai-Yin to random matrices with dependent entries. Tracy-Widom fluctuation for the largest singular value was proved for isotropic \cite{soshnikov2002note, peche2009universality, Johnstone2001OnTD, pillai2012edge, pillai2014universality} and anisotropic \cite{elkaroui2007tracywidom, onatski2008tracywidom, bao2015universality, knowles2017anisotropic, lee2016tracywidom, ding2018necessary, yang2019edge} sample covariance matrices, respectively. For the smallest singular value, Tracy-Widom fluctuation was identified in \cite{feldheim2010universality, tao2010smallest, pillai2012edge, pillai2014universality}. See the recent survey \cite{yau2026randomwigner} for more details.

\paragraph{Spectral edges for sparse random matrices}
For the Erd\H{o}s--R\'enyi graph $\gG(N,d/N)$, it was shown in \cite{furedi1981eigenvalues, vu2007spectral} that the smallest
and second-largest eigenvalues of the adjacency matrix converge to the edges of the support of the asymptotic eigenvalue distribution provided that $d/(\log(N))^{4}\to\infty$, and subsequently strengthened to $d/\log(N)\to\infty$ by Benaych-Georges, Bordenave, and Knowles \cite{benaych2020spectral} for the inhomogeneous case. When $d/\log(N) \to 0$, Krivelevich and Sudakov \cite{krivelevich2003largest} showed that the largest eigenvalue of the adjacency matrix is asymptotically equivalent to the maximum of the square root of the largest degree and the maximal mean degree $d$, later refined by Benaych-Georges, Bordenave, and Knowles \cite{benaych2019largest} for the $k$-th largest eigenvalue. At the critical scale $d= b\log(N)$ for some constant$b>0$, Alt, Ducatez, and Knowles \cite{alt2021extremal}, as well as Tikhomirov and Youssef \cite{tikhomirov2021outliers}, showed that spectral outliers beyond the semicircle edges, due to high-degree vertices, emerge if and only if $b<1/\ln(4/e)$. For the bipartite Erd\H{o}s--R\'enyi graphs $\gG(n, m, d/\sqrt{mn})$, when $d= b\log(n+m)$ for some constant$b>0$, Dumitriu, Wang, Wang, and Zhu \cite{dumitriu2025singular} established the sharp thresholds the emergence of left and right outliers outside the Marcenko-Pastur edges induced by high- and low-degree vertices, respectively, with additional parameter restrictions for the lower edge. For inhomogeneous case, Dumitriu and Zhu \cite{dumitriu2024extreme} obtained bounds for largest and smallest singular values when $d/\log(N)\to\infty$, while their lower bound is not as sharp as the current paper. Brailovskaya and van Handel \cite{brailovskaya2022universality} obtained sharp bounds for smallest singular value which is valid when $d/(\log^{4}N)\to \infty$. Che and Lopatto \cite{che2019universality} proved that the distribution of the least singular value is the same as the distribution of the least singular value of a Gaussian matrix ensemble. See the recent reviews \cite{yau2026randomwigner, bordenave2026sparse} for more details.

\paragraph{Concentration of regularized sparse random matrices and algorithm applications.}
Following the approach developed by Friedman, Kahn, and Szemer\'edi \cite{friedman1989second}, Feige and Ofek \cite{feige2005spectral} showed that removing high-degree vertices restores concentration. Later results extended regularized adjacency concentration to stochastic block models \cite{chin2015stochastic,lei2015consistency} and hypergraph variants \cite{dumitriu2025partial}. Through a different decomposition of directed inhomogeneous Erd\H{o}s--R\'enyi graphs, Le, Levina, and Vershynin \cite{le2017concentration} established concentration for adjacency matrices and normalized graph Laplacians after regularization. For community detection tasks, these concentration results provide theoretical guarantees for regularized spectral clustering algorithms \cite{chaudhuri2012spectral,qin2013regularized,joseph2016impact,gao2017achieving,guedon2016community,dumitriu2026optimal}.

\paragraph{Spectra of nonbacktracking matrices and algorithm applications}
The Ihara--Bass formula \cite{hashimoto1989zeta, bass1992ihara, kotani2000zeta, angel2015non} relates the spectrum of the nonbacktracking matrix to the spectrum of the base matrix. Bordenave \cite{bordenave2020new} reproved Friedman's second eigenvalue Theorem for random $d$-regular graph using the nonbacktracking matrix. Wang and Wood \cite{wang2017limiting} provided a precise description of the limiting empirical spectral distribution (ESD) for the non-backtracking matrices constructed from $\gG(N, d/N)$ when $d/\log (N)\to\infty$. Benaych-Georges, Bordenave, and Knowles \cite{benaych2020spectral} obtained spectral radius bounds for nonbacktracking matrices of inhomogeneous sparse square matrices, with Dumitriu and Zhu \cite{dumitriu2024extreme} extending these results to the rectangular case. More recently, McKenzie \cite{mckenzie2026alon} obtained a lower bound on the $k$-th largest modulus of an eigenvalue of the non-backtracking matrix for any fixed $k$. Spectrum of nonbacktracking matrices constructed from hypergraphs were also explored in \cite{dumitriu2021spectra,stephan2022sparse, stephan2024long, fernandez2026achieving}. For applications, algorithms based on nonbacktracking matrices achieve weak recovery in Stochastic Block Models \cite{massoulie2014community, mossel_reconstruction_2015,mossel2018proof, bordenave2018nonbacktracking, stephan2020non, coste2021eigenvalues, stephan2025bethe} and their hypergraph variants \cite{stephan2022sparse,fernandez2026achieving}. Nonbacktracking algorithms are also applied to matrix recovery \cite{bordenave2020detection, stephan2020non} and tensor unfolding \cite{stephan2024long}.

\paragraph{Non-asymptotic matrix concentration results}
Rudelson and Vershynin \cite{rudelson2009smallest} proved an optimal estimate of the smallest singular value of a random subgaussian matrix, combining the approach developed in \cite{litvak2005smallest} with certain Littlewood-Offord-type theorems \cite{rudelson2008littlewood}. Their results were extended to random matrices with weaker moment assumptions \cite{koltchinskii2015bounding, tikhomirov2016smallest}, inhomogeneous entries \cite{cook2018lower,livshyts2021smallest} and heavy-tailed entries \cite{tikhomirov2018sample,guedon2020random,livshyts2021smallest2}. Bandeira and van Handel \cite{bandeira2016sharp} established nonasymptotic bounds on the spectral norm of random matrices with independent entries. Dimension-free bounds were later obtained in \cite{van2017spectral, latala2018dimension}. Fluctuations at the Tracy-Widom scale were established in \cite{brailovskaya2024extremal}. For heteroskedastic Wishart matrices, non-asymptotic concentration bounds were established in \cite{zhivotovskiy2024dimension,cai2022non}.

\paragraph{Matrix concentration inequalities for sum of random matrices}
Tropp established matrix concentration inequalities for sums of independent random self-adjoint matrices, providing noncommutative analogues of classical results through the matrix Laplace-transform method \cite{tropp2012user,tropp2015introduction}. Oliveira \cite{oliveira2009concentration} and Tropp \cite{tropp2011freedman} established matrix Freedman inequalities for matrix-valued martingales. Moreover, Tropp \cite{tropp2018second} refined matrix Khintchine inequality by using \emph{matrix alignment parameter} to reduce the dimensional dependence. More recently, Tropp obtained one-sided Gaussian comparisons for extreme eigenvalues \cite{tropp2025comparison,tropp2026extreme} and an exchangeable-counterpart proof of Brailovskaya and van Handel's universality theorems \cite{tropp2026exchangeable}.
Bandeira, Boedihardjo, and van Handel \cite{bandeira2021matrix} established nonasymptotic upper bounds for the spectra of general Gaussian random matrices, relying on concepts such as asymptotic freeness \cite{Voi91} and strong convergence \cite{HT05,And13} to account for matrix noncommutativity. Corresponding lower bounds for the spectral edges were subsequently established in \cite{bandeira2024matrix}. For sum of random matrices with non-Gaussian entries, \cite{brailovskaya2022universality} proved concentration inequalities by comparing them to sums of independent Gaussian matrices having the same means and covariances. Exponential concentration inequalities for sums of dependent random matrices have been established using an alternative exchangeable-pair method \cite{MJCFT14}, while \cite{SvW23} demonstrates Gaussian universality for sums arising from $\psi$-mixing Markov chains. For noncommutative polynomial of random matrices, strong convergence, a.k.a, convergence of operator norm, were established in \cite{bordenave2019eigenvalues, bordenave2023norm, bordenave2024strong, chen2024new}. See the recent reviews \cite{bordenave2026sparse, tropp2026applied, van2026strong} for more details.

\subsection{Organization}
\label{subsec:organization}
The rest of the paper is organized as follows. Section~\ref{sec:proofs-main-results} presents the proofs of the main results. Section~\ref{sec:trace-estimates} develops the trace estimates used in proving Proposition~\ref{prop:nonbacktracking-radius-bounds}. Section~\ref{sec:future-directions} discusses possible extensions. Numerical experiments are presented in  Appendix~\ref{sec:numerical-experiments}. Estimates for degree-cutoff and vertex-retention are presented in Appendix~\ref{sec:regularization-thresholds}. The graph moment conditions for the models considered in this paper are presented in Appendix~\ref{sec:model-moment-verification}.

\subsection{Notation}
\label{subsec:notation}
For an integer $k\geq1$, write $[k]=\{1,\ldots,k\}$. For a finite set $I$, $|I|$ denotes its cardinality and $I^c$ its complement in the relevant ambient index set. All logarithms are natural, $\ii$ denotes the imaginary unit, and $[x]_+\coloneqq\max\{x,0\}$ for real $x$.

For a matrix $\mM$, $\mM^*$ denotes its conjugate transpose,
$\mM_{I,J}$ its restriction to rows in $I$ and columns in $J$, and
$\mM_I\coloneqq\mM_{I,I}$ a principal submatrix when $\mM$ is square.
The notation $\id_k$ denotes the $k\times k$ identity matrix, with
the dimension omitted when it is clear from context.
We write $\Tr$ for the trace and $\diag$ for a diagonal or block
diagonal matrix, as appropriate.
The norm $\|\mM\|\coloneqq\sup_{\|x\|_2=1}\|\mM x\|_2$ is the
Euclidean operator norm, and
$\|\mM\|_{1\to\infty}\coloneqq\max_{i,j}|\emM_{ij}|$ is the
entrywise maximum norm. The operator norm of an empty restriction
is taken to be zero.

For a nonempty matrix $\mM\in\C^{n\times m}$, its
$k=\min\{n,m\}$ singular values are ordered as
$\sigma_1(\mM)\geq\cdots\geq\sigma_k(\mM)\geq0$; we write
$\sigma_{\max}(\mM)=\sigma_1(\mM)$ and
$\sigma_{\min}(\mM)=\sigma_k(\mM)$.
For a Hermitian $k\times k$ matrix, eigenvalues are ordered as
$\lambda_1(\mM)\geq\cdots\geq\lambda_k(\mM)$, with
$\lambda_{\max}(\mM)=\lambda_1(\mM)$ and
$\lambda_{\min}(\mM)=\lambda_k(\mM)$.
For any square matrix, its spectral radius is
$\rho(\mM)\coloneqq\max\{|\lambda|:\lambda\text{ is an eigenvalue of }\mM\}$,
with $\rho(\mM)=0$ for an empty matrix.
For Hermitian matrices of the same size, $\mA\preceq\mB$ means that
$\mB-\mA$ is positive semidefinite, and $\mA\prec\mB$ means that
it is positive definite; $\succeq$ and $\succ$ denote the reversed orders.

Unless otherwise stated, limits are taken as $N\to\infty$.
For $b_N>0$, $a_N=O(b_N)$ means that $|a_N|\leq Cb_N$ eventually,
and $a_N=o(b_N)$, also written $a_N\ll b_N$, means that
$a_N/b_N\to0$; $a_N\gg b_N$ means $b_N\ll a_N$.
For positive $a_N,b_N$, the notation $a_N\asymp b_N$ means that
$cb_N\leq a_N\leq Cb_N$ eventually for fixed constants $c,C>0$.
Constants denoted by $c$ or $C$ may change from line to line;
their permitted dependencies are those specified in the relevant statement.

We write $\P$, $\E$, and $\Var$ for probability, expectation, and
variance, respectively, $\indi{E}$ for the indicator of an event $E$,
and $\Ber(p)$ for the Bernoulli distribution with success probability $p$.
Unless otherwise stated, convergence of random quantities is in probability.
For deterministic $b_N>0$, $Y_N=o_\P(b_N)$ means that
$\P(|Y_N|>\eta b_N)\to0$ for every fixed $\eta>0$.
An event holds with high probability if its probability tends to one.

\section{Proofs of Main Results}
\label{sec:proofs-main-results}

Our framework separates the proof of spectral-edge bounds into a deterministic step and a probabilistic step. 
\begin{itemize}
      \item \emph{Deterministic step}: we establish the Ihara--Bass formulas to derive Loewner inequalities that control extreme eigenvalues and singular values in terms of spectral radius of corresponding nonbacktracking matrices, the maximum entry magnitude, and the row and column squared $\ell^2$-norms.
      \item \emph{Probabilistic step}: once the entrywise and row/column variance estimates are available, the spectral radius bounds for nonbacktracking matrices follow from trace-moment estimates and Markov's inequality. In Section~\ref{sec:trace-estimates}, those trace estimates are reduced to verify the graph moment assumptions in Subsection~\ref{subsec:moment-assumptions}; we carry out this model-specific verification in Section~\ref{sec:model-moment-verification}.
\end{itemize}

This section is organized as follows: we first establish the deterministic comparisons in Subsection~\ref{subsec:ihara-bass-loewner-orders}. Proofs of main Theorems are subsequently established in Subsections~\ref{subsec:singular-value-proofs}, ~\ref{subsec:rectangular-smallest-singular-value-proof} and \ref{subsec:eigenvalue-proofs}, by applying these deterministic comparisons with Proposition~\ref{prop:nonbacktracking-radius-bounds}. Finally in Subsection~\ref{subsec:trace-to-radius}, we will show how the proof of Proposition~\ref{prop:nonbacktracking-radius-bounds} can be reduced to trace estimates.

\subsection{Ihara--Bass formulas and Loewner orders}
\label{subsec:ihara-bass-loewner-orders}

\subsubsection{Ihara--Bass identities.}
The Ihara--Bass formulas below provide the connection between the nonbacktracking spectral radius and the spectral parameters at which a matrix acting on vertices becomes singular. 

\begin{lemma}[Ihara--Bass formula, {\cite[Lemma~4.1]{benaych2020spectral}}]
\label{lem:general-ihara-bass}
Let $\mH\in\C^{N\times N}$ be Hermitian, and let
$\mB=\mB(\mH)$ be its nonbacktracking matrix from
Definition~\ref{def:nonbacktracking-matrix}. For $\lambda\in\C\setminus\{0\}$ satisfying $\lambda^2\ne|\emH_{uv}|^2$ for all $u,v\in[N]$, define $\mH(\lambda)\in\C^{N\times N}$ and the diagonal matrix $\mM(\lambda)=\diag(\emM_{uu}(\lambda))_{u\in[N]}$ by
\begin{align}
 \emH_{uv}(\lambda)
 &\coloneqq\frac{\lambda\emH_{uv}}
                    {\lambda^2-|\emH_{uv}|^2},
 &\emM_{uu}(\lambda)
 &\coloneqq 1+\sum_{v=1}^N\frac{|\emH_{uv}|^2}
                         {\lambda^2-|\emH_{uv}|^2}.
 \label{eqn:weighted-ihara-bass-matrix}
\end{align}
Then $\lambda$ is an eigenvalue of $\mB$ if and only if
$\det\bigl(\mM(\lambda)-\mH(\lambda)\bigr)=0$.
\end{lemma}

\begin{lemma}[Bipartite Ihara--Bass formula, {\cite[Lemma~3.2]{dumitriu2024extreme}}]
\label{lem:bipartite-ihara-bass}
Let $\mX\in\C^{n\times m}$, set $N=n+m$, and let
$\widetilde{\mB}=\widetilde{\mB}(\mX)$ be its nonbacktracking
matrix from Definition~\ref{def:nonbacktracking-matrix}.
For $\lambda\in\C\setminus\{0\}$ with
$\lambda^2\ne|\emX_{ij}|^2$ for all $i,j$, define
$\mX(\lambda)\in\C^{n\times m}$,
$\mX^\sharp(\lambda)\in\C^{m\times n}$, and the diagonal matrices
$\mM_{\rm r}(\lambda)$ and $\mM_{\rm c}(\lambda)$ by
\begin{align}
 \emX_{ij}(\lambda)
 &\coloneqq\frac{\lambda\emX_{ij}}
                    {\lambda^2-|\emX_{ij}|^2},
& \emX_{ji}^{\sharp}(\lambda)
 &\coloneqq\frac{\lambda\overline{\emX_{ij}}}
                    {\lambda^2-|\emX_{ij}|^2}, \label{eqn:rectangular-X(lambda)}\\
 \emM_{\rm{r}, ii}(\lambda)
 &\coloneqq 1+\sum_{j=1}^m\frac{|\emX_{ij}|^2}
                          {\lambda^2-|\emX_{ij}|^2},
 &\emM_{\rm{c}, jj}(\lambda)
 &\coloneqq 1+\sum_{i=1}^n\frac{|\emX_{ij}|^2}
                          {\lambda^2-|\emX_{ij}|^2}.
 \label{eqn:rectangular-M(lambda)}
\end{align}
\begin{enumerate}[label=\textup{(\roman*)},leftmargin=2.5em,itemsep=0.6em]
\item For the block linearization $\widetilde{\mH}$ in
\eqref{eqn:block-linearization}, the vertex matrix is
\begin{align}
 \mM(\lambda)-\widetilde{\mH}(\lambda)
 =\begin{bmatrix}
    \mM_{\rm r}(\lambda)&-\mX(\lambda)\\
    -\mX^{\sharp}(\lambda)&\mM_{\rm c}(\lambda)
   \end{bmatrix}.
 \label{eqn:block-ihara-entries}
\end{align}
The parameter $\lambda$ is an eigenvalue of $\widetilde{\mB}$ if
and only if this block matrix is singular.

\item If $\mM_{\rm r}(\lambda)$ is nonsingular, then $\lambda$ is an eigenvalue of $\widetilde{\mB}$ if and only if
\begin{align}
      \det\bigl(\mM_{\rm c}(\lambda)
      -\mX^{\sharp}(\lambda)[\mM_{\rm r}(\lambda)]^{-1}
        \mX(\lambda)\bigr) = 0.
 \label{eqn:weighted-ihara-bass-block-determinant}
\end{align}
\end{enumerate}
\end{lemma}

\begin{proof}[Proof of Lemma~\ref{lem:bipartite-ihara-bass}]
The nonzero entries of $\widetilde{\mH}$ lie in its two rectangular
off-diagonal blocks.  Substituting those entries in
\eqref{eqn:weighted-ihara-bass-matrix} gives
\eqref{eqn:block-ihara-entries}.  By
Remark~\ref{rem:nonbacktracking-definition-bipartite-block-linearization},
$\mB(\widetilde{\mH})$ and $\widetilde{\mB}$ have the same
nonzero spectrum, so Lemma~\ref{lem:general-ihara-bass} proves the
first assertion. When $\mM_{\rm r}(\lambda)$ is invertible,
\eqref{eqn:weighted-ihara-bass-block-determinant} follows from the
determinant identity in Lemma~\ref{lem:schur-complement}.
\end{proof}

The previous argument uses determinant tests to detect singularity, but positivity is needed for quadratic-form bounds. As the parameter increases, the Hermitian matrix approaches the identity; loss of positivity as it decreases signals a nonbacktracking eigenvalue via the Ihara--Bass identity, which cannot occur beyond the spectral radius. Recall the row and column variance bounds in \eqref{eqn:row-column-var-min}--\eqref{eqn:row-column-var-max}.

\begin{lemma}[Loewner orders]
\label{lem:weighted-ihara-bass-positivity}
With the notation above, the following orders hold.
\begin{enumerate}[label=\textup{(\roman*)},leftmargin=2.5em,itemsep=0.6em]
\item If $\mH$ is Hermitian and
$\lambda>\max\{\rho(\mB),\|\mH\|_{1\to\infty}\}$ is real, then
\begin{align}
 \mM(\lambda)-\mH(\lambda)\succ0.
 \label{eqn:weighted-ihara-bass-positivity}
\end{align}
\item For $\mX\in\C^{n\times m}$, write
$\rowvarmax=\max_i\sum_j|\emX_{ij}|^2$.
If $\lambda=\ii\beta$ with
$\beta>\max\{\rho(\widetilde{\mB}),\sqrt{\rowvarmax}\}$, then
\begin{subequations}
\begin{align}
 \mM_{\rm r}(\lambda)
 \succeq(1-\rowvarmax/\beta^2)\id_n &\succ0,
 \label{eqn:weighted-ihara-bass-imaginary-row-positivity}\\
 \mM_{\rm c}(\lambda)
   -\mX^{\sharp}(\lambda)\mM_{\rm r}(\lambda)^{-1}
       \mX(\lambda)&\succ0.
 \label{eqn:weighted-ihara-bass-imaginary-positivity}
\end{align}
\end{subequations}
\end{enumerate}
\end{lemma}

\begin{proof}[Proof of Lemma~\ref{lem:weighted-ihara-bass-positivity}]
Following the continuity argument in
\cite[Proposition~6.2]{dumitriu2025singular}, we show in each case
that the relevant Hermitian matrix tends to the identity and is
nonsingular throughout the stated interval.  Its eigenvalues are
then positive for large parameters and cannot cross zero.

For \textup{(i)}, the denominators in
\eqref{eqn:weighted-ihara-bass-matrix} are positive for
$t>\max\{\rho(\mB),\|\mH\|_{1\to\infty}\}$, so
$\mM(t)-\mH(t)$ is Hermitian and continuous on this interval.
The same formulas yield, as $t\to\infty$, in operator norm,
\[
 \mM(t)-\mH(t)=\id_N-t^{-1}\mH+O(t^{-2}).
\]
Moreover, Lemma~\ref{lem:general-ihara-bass}
shows that this matrix is nonsingular throughout the interval,
since $t>\rho(\mB)$.  The continuity argument therefore proves
\eqref{eqn:weighted-ihara-bass-positivity}.

For \textup{(ii)}, by taking $\beta>\sqrt{\rowvarmax}$,  \eqref{eqn:weighted-ihara-bass-imaginary-row-positivity} is proved by the diagonal formula \eqref{eqn:rectangular-M(lambda)}, since
\[
 \emM_{{\rm r},ii}(\ii\beta)
 =1-\sum_j|\emX_{ij}|^2(\beta^2+|\emX_{ij}|^2)^{-1}
 \geq1-\rowvarmax\beta^{-2}>0,
\]
ensuring the existence and continuity of $[\mM_{\rm r}(\ii\beta)]^{-1}$ on this interval. Since $\mX^\sharp(\ii\beta)=-\mX(\ii\beta)^*$,
\[
 \begin{aligned}
 \mS(\beta)
 &=\mM_{\rm c}(\ii\beta)
   +\mX(\ii\beta)^*[\mM_{\rm r}(\ii\beta)]^{-1}\mX(\ii\beta)
 \end{aligned}
\]
is Hermitian and continuous.  The diagonal and entry formulas yield, as $\beta\to\infty$, in operator norm,
\[
 \mM_{\rm c}(\ii\beta)=\id_m+O(\beta^{-2}),\qquad
 [\mM_{\rm r}(\ii\beta)]^{-1}=\id_n+O(\beta^{-2}),\qquad
 \mX(\ii\beta)=O(\beta^{-1}).
\]
Hence $\mS(\beta)=\id_m+O(\beta^{-2})$ as
$\beta\to\infty$.
Finally, for
$\beta>\max\{\rho(\widetilde{\mB}),\sqrt{\rowvarmax}\}$,
there are no poles at $\lambda=\ii\beta$, and
Lemma~\ref{lem:bipartite-ihara-bass}\textup{(ii)} implies that
$\mS(\beta)$ is nonsingular because
$|\ii\beta|>\rho(\widetilde{\mB})$.
The same continuity argument gives $\mS(\beta)\succ0$ throughout
this interval, proving
\eqref{eqn:weighted-ihara-bass-imaginary-positivity}.
\end{proof}

\subsubsection{Loewner orders for quadratic forms}
We now convert Lemma~\ref{lem:weighted-ihara-bass-positivity} into inequalities for the original block linearization $\widetilde{\mH}$ in \eqref{eqn:block-linearization}. Define the diagonal matrices $\mQ_{\rm r}\coloneqq\diag(\mQ_{{\rm r},ii})_{i\in[n]}$ and $\mQ_{\rm c}\coloneqq\diag(\mQ_{{\rm c},jj})_{j\in[m]}$ by
\begin{align}
      \mQ&\coloneqq
      \diag(\mQ_{\rm r},\mQ_{\rm c}),
      & \mQ_{\rm r, ii}
 &\coloneqq
       \sum_j|\emX_{ij}|^2,
 &\mQ_{\rm c, jj}
 &\coloneqq
       \sum_i|\emX_{ij}|^2. \label{eqn:rectangular-mS}
\end{align}
We denote $\overline{q_{\rm r}}\coloneqq\|\mQ_{\rm r}\|$,
$\overline{q_{\rm c}}\coloneqq\|\mQ_{\rm c}\|$, and
$\underline{q_{\rm c}}\coloneqq\min_j(\mQ_{\rm c})_{jj}$.
In particular, for any $\vf=(\vx,\vy)\in\C^n\oplus\C^m$,
\[
 \vf^*\mQ \vf
 =\sum_{i,j}|\emX_{ij}|^2\bigl(|x_i|^2+|y_j|^2\bigr).
\]
\begin{proposition}[Upper bound for Hermitian linearization]
\label{prop:rectangular-upper-quadratic-form}
Define the nonnegative diagonal error matrix
\begin{align}
 \mE_{\rm R}(t)
 \coloneqq\frac{\|\mX\|_{1\to\infty}}{t(t-\|\mX\|_{1\to\infty})}\mQ,
 \label{eqn:rectangular-upper-form-error}
\end{align}
where $t>\max\{\rho(\widetilde{\mB}),\|\mX\|_{1\to\infty}\}$. Then
the following inequality holds:
\begin{align}
 \widetilde{\mH}
 \preceq t\id_N+t^{-1}\mQ+\mE_{\rm R}(t).
 \label{eqn:block-upper-quadratic-form}
\end{align}
\end{proposition}

\begin{proof}[Proof of Proposition~\ref{prop:rectangular-upper-quadratic-form}]
Fix $t>\max\{\rho(\widetilde{\mB}),\|\mX\|_{1\to\infty}\}$.
Let $\mT=t\mX(t)$, so that
$\emT_{ij}=t^2\emX_{ij}(t^2-|\emX_{ij}|^2)^{-1}$, and denote its
Hermitian linearization by $\widetilde{\mT}$.  Write
$\mM(t)=\diag(\mM_{\rm r}(t),\mM_{\rm c}(t))$.
Observe that
\begin{align}
 \widetilde{\mH}-t\id_N-t^{-1}\mQ
 &=(\widetilde{\mH}-\widetilde{\mT})
   +(t\mM(t)-t\id_N-t^{-1}\mQ)
   +(\widetilde{\mT}-t\mM(t)).
 \label{eqn:rectangular-real-form-decomposition}
\end{align}
Below, we will bound the first two terms separately, then use the real-parameter Ihara--Bass order to show that the last term is nonpositive.
First, the entrywise deformation satisfies
\[
 |\emT_{ij}-\emX_{ij}|
 =|\emX_{ij}|^3(t^2-|\emX_{ij}|^2)^{-1}
 \leq\|\mX\|_{1\to\infty}|\emX_{ij}|^2
             (t^2-\|\mX\|_{1\to\infty}^2)^{-1}.
\]
For any $\vf=(\vx,\vy)\in\C^n\oplus\C^m$, the inequality
$2|x_i y_j|\leq|x_i|^2+|y_j|^2$ therefore gives
\begin{align}
 \bigl|\vf^*(\widetilde{\mH}-\widetilde{\mT})\vf\bigr|
 &\leq2\sum_{i,j}|\emX_{ij}-\emT_{ij}|\,|x_i|\,|y_j|
 \nonumber\\
 &\leq\|\mX\|_{1\to\infty}(t^2-\|\mX\|_{1\to\infty}^2)^{-1}
       \sum_{i,j}|\emX_{ij}|^2(|x_i|^2+|y_j|^2)
 \nonumber\\
 &=\|\mX\|_{1\to\infty}(t^2-\|\mX\|_{1\to\infty}^2)^{-1}
       \vf^*\mQ\vf.
 \label{eqn:rectangular-real-deformation-form-error}
\end{align}
Second, the diagonal formulas give
\[
 \begin{aligned}
 0\leq t\emM_{{\rm r},ii}(t)-t-t^{-1}\mQ_{{\rm r},ii}
 &=t^{-1}\sum_j|\emX_{ij}|^4(t^2-|\emX_{ij}|^2)^{-1}\\
 &\leq\|\mX\|_{1\to\infty}^2t^{-1}
       (t^2-\|\mX\|_{1\to\infty}^2)^{-1}\mQ_{{\rm r},ii}.
 \end{aligned}
\]
The same calculation holds for columns.  Since both blocks are
diagonal, these entrywise inequalities imply
\begin{align}
 0\preceq t\mM(t)-t\id_N-t^{-1}\mQ
 \preceq\|\mX\|_{1\to\infty}^2t^{-1}
       (t^2-\|\mX\|_{1\to\infty}^2)^{-1}\mQ.
 \label{eqn:rectangular-real-block-error}
\end{align}

Finally, since $t \geq \max\{ \rho(\widetilde{\mB}), \|\mX\|_{1\to\infty}\}$, Lemma~\ref{lem:weighted-ihara-bass-positivity}\textup{(i)}
and \eqref{eqn:block-ihara-entries}, multiplied by $t$, lead to $\widetilde{\mT}\preceq t\mM(t)$. With this, substituting \eqref{eqn:rectangular-real-deformation-form-error} and
\eqref{eqn:rectangular-real-block-error} into \eqref{eqn:rectangular-real-form-decomposition} yields
\[
 \begin{aligned}
 \widetilde{\mH}
 &\preceq t\id_N+t^{-1}\mQ
   +\bigl(\|\mX\|_{1\to\infty}
          +t^{-1}\|\mX\|_{1\to\infty}^2\bigr)
      (t^2-\|\mX\|_{1\to\infty}^2)^{-1}\mQ\\
 &=t\id_N+t^{-1}\mQ
   +\|\mX\|_{1\to\infty}
       \bigl[t(t-\|\mX\|_{1\to\infty})\bigr]^{-1}\mQ
 =t\id_N+t^{-1}\mQ+\mE_{\rm R}(t),
 \end{aligned}
\]
which completes the proof of \eqref{eqn:block-upper-quadratic-form}.
\end{proof}

To obtain the lower bound, we consider a purely imaginary parameter $\lambda = \ii\beta$. Using the identity $\mX^\sharp(\ii\beta) = -\mX(\ii\beta)^*$, the Schur-complement order in \eqref{eqn:weighted-ihara-bass-imaginary-positivity} yields a lower bound for a weighted Gram matrix. The next proposition translates this lower bound into a quadratic form in $\mX$, employing diagonal approximations for the row and column sums of second moments. Additionally, a positive row margin ensures the inverse weight remains uniformly bounded. While similar approaches can be found in \cite{brito2022spectral, dumitriu2024extreme, dumitriu2025singular}, our result is more general, accommodating a broader range of parameters and error terms.

\begin{proposition}[Lower bound for quadratic-form]
\label{prop:rectangular-lower-quadratic-form}
Fix $0<\eta<1$, $t\in[\eta,\eta^{-1}]$, and $u\in[0,1]$.
Let $\mW_{\rm r}\in\R^{n\times n}$ and
$\mW_{\rm c}\in\R^{m\times m}$ be nonnegative diagonal matrices. Assume that for some $\varepsilon_{\rm r},\varepsilon_{\rm c}\geq0$, the following hold:
\begin{subequations}
\label{eqn:rectangular-lower-form-hypotheses}
\begin{alignat}{3}
  \|\mQ_{\rm r}-\mW_{\rm r}\|&\leq\varepsilon_{\rm r},\qquad
  &\|\mQ_{\rm c}-\mW_{\rm c}\|&\leq\varepsilon_{\rm c},\qquad
  & \id_n-t^{-2}\mW_{\rm r}&\succeq\eta\id_n,
  \label{eqn:rectangular-lower-form-approximation}\\
  \varepsilon_{\rm r}&\leq\tfrac12\eta t^2, \qquad
 &\max\{\overline{q_{\rm r}}, \overline{q_{\rm c}}\}&\leq\eta^{-1},\qquad
 &\rho(\widetilde{\mB})&\leq t+u.
 \label{eqn:rectangular-lower-form-bounds}
\end{alignat}
\end{subequations}
Define the positive-definite matrix
\[
 \mK(t)\coloneqq(\id_n-t^{-2}\mW_{\rm r})^{-1}.
\]
Then the following inequality holds:
\begin{align}
  \mX^*\mK(t)\mX
  \succeq\mW_{\rm c}-t^2\id_m-\mE_{\rm L},
  \label{eqn:block-lower-quadratic-form}
 \end{align}
where $\mE_{\rm L}$ is the nonnegative diagonal error matrix defined by
\[
 \mE_{\rm L}\coloneqq
 C_\eta\bigl(\|\mX\|_{1\to\infty}+u
                   +\varepsilon_{\rm r}+\varepsilon_{\rm c}\bigr)\id_m
\]
and $C_\eta>0$ is a constant depending only on $\eta$.
\end{proposition}

\begin{proof}[Proof of Proposition~\ref{prop:rectangular-lower-quadratic-form}]
Let $\mT=\ii\beta\mX(\ii\beta)$, then $\emT_{ij}=\beta^2\emX_{ij}/(\beta^2+|\emX_{ij}|^2)$. Observe that
\begin{align}
  & \|\mX^*\mK(t)\mX-\mT^*[\mM_{\rm r}(\ii\beta)]^{-1}\mT\|\\
  \leq & \|\mX^*\mK(t)\mX - \mX^*[\mM_{\rm r}(\ii\beta)]^{-1}\mX\| + \|\mX^*[\mM_{\rm r}(\ii\beta)]^{-1}\mX - \mT^*[\mM_{\rm r}(\ii\beta)]^{-1}\mT\| \\
  \leq & \|\mX\|^2\|\mK(t)-[\mM_{\rm r}(\ii\beta)]^{-1}\|
  +\|[\mM_{\rm r}(\ii\beta)]^{-1}\| \big(\|\mX\|+\|\mT\|\big)\|\mX-\mT\|. \label{eqn:rectangular-weighted-gram-error}
  %\leq &C_\eta\bigl(\|\mX\|_{1\to\infty}+\varepsilon_{\rm r}+\beta-t\bigr).
 \end{align}
We bound the factors on the right-hand side of
\eqref{eqn:rectangular-weighted-gram-error} in turn, throughout taking
$t+u<\beta\leq t+u+1$.

First, \eqref{eqn:rectangular-lower-form-bounds} implies that
$\|\mX\|_{1\to\infty}^2\leq\overline{q_{\rm r}}\leq\eta^{-1}$ and
$\|\mQ\|\leq\eta^{-1}$. By applying
Proposition~\ref{prop:rectangular-upper-quadratic-form} at
$\|\mX\|_{1\to\infty}+(t+u+1)$, we obtain
\[
 \|\mX\|=\lambda_{\max}(\widetilde{\mH})
 \leq\|\mX\|_{1\to\infty}+(t+u+1)+(t+u+1)^{-1}\|\mQ\|
 \leq C_\eta.
\]
Since $|\emT_{ij}-\emX_{ij}|\leq\|\mX\|_{1\to\infty}\beta^{-2}|\emX_{ij}|^2$, the Schur test, Lemma~\ref{lem:schur_test}, implies
\begin{align}
 \|\mT-\mX\|\leq
 \|\mX\|_{1\to\infty}\sqrt{\overline{q_{\rm r}}\,\overline{q_{\rm c}}}/{\beta^2} \leq C_\eta\|\mX\|_{1\to\infty}.
 \label{eqn:block-imaginary-deformation-error}
\end{align}
where the last inequality follows from $\beta\geq\eta$ and \eqref{eqn:rectangular-lower-form-bounds}. Hence $\|\mT\|\leq\|\mX\|+\|\mT-\mX\|\leq C_\eta$.

Second, the row conditions in \eqref{eqn:rectangular-lower-form-approximation} and \eqref{eqn:rectangular-lower-form-bounds} imply
\[
 \mQ_{\rm r}\preceq\mW_{\rm r}+\varepsilon_{\rm r}\id_n
 \preceq(1-\eta/2)t^2\id_n.
\]
Since $\beta>t$, the diagonal formulas give
\begin{align}
 \mM_{\rm r}(\ii\beta)
 \succeq\id_n-\beta^{-2}\mQ_{\rm r}\succeq (\eta/2)\id_n,
 \label{eqn:block-imaginary-row-positivity}
\end{align}
hence $\|[\mM_{\rm r}(\ii\beta)]^{-1}\|\leq2\eta^{-1}$ and $\|\mK(t)\|\leq\eta^{-1}$.
For $q\in\{\mathrm r,\mathrm c\}$, define the diagonal remainder matrices by
\[
 \mR_q(\beta)\coloneqq\mM_q(\ii\beta)-\id+\beta^{-2}\mQ_q.
\]
Since $|\emX_{ij}|^2\leq\|\mX\|_{1\to\infty}^2$ and
$\beta^2(\beta^2+|\emX_{ij}|^2)\geq\beta^4$, we have
\[
 \begin{aligned}
 0\leq [\mR_{\rm r}(\beta)]_{ii}
 &=\sum_j\left(\frac{|\emX_{ij}|^2}{\beta^2}
       -\frac{|\emX_{ij}|^2}{\beta^2+|\emX_{ij}|^2}\right)
 =\sum_j\frac{|\emX_{ij}|^4}
       {\beta^2(\beta^2+|\emX_{ij}|^2)} \leq\|\mX\|_{1\to\infty}^2\beta^{-4} \mQ_{\rm r, ii}.
 \end{aligned}
\]
The same calculation holds for the column remainder. Then for $q\in\{\mathrm r,\mathrm c\}$, we have
\[
 0\preceq\mR_q(\beta)
 \preceq\|\mX\|_{1\to\infty}^2\beta^{-4}\mQ_q,
 \qquad
 \|\mR_q(\beta)\|
 \leq\|\mX\|_{1\to\infty}^2\beta^{-4}\|\mQ_q\|.
\]
Using \eqref{eqn:rectangular-lower-form-approximation},
\eqref{eqn:rectangular-lower-form-bounds}, and
$0\leq t^{-2}-\beta^{-2}\leq2\eta^{-3}(\beta-t)$, we therefore have
\begin{align}
 &\|\mM_q(\ii\beta)-(\id-t^{-2}\mW_q)\|
 =\|\mR_q(\beta)-t^{-2}(\mQ_q-\mW_q)
                    +(t^{-2}-\beta^{-2})\mQ_q\|
 \nonumber\\
 \leq & \|\mR_q(\beta)\|+t^{-2}\varepsilon_q
                    +(t^{-2}-\beta^{-2})\|\mQ_q\|
\leq C_\eta\bigl(\varepsilon_q
                   +\|\mX\|_{1\to\infty}^2+\beta-t\bigr).
 \label{eqn:rectangular-imaginary-block-error}
\end{align}
Since $\mK(t)^{-1}=\id_n-t^{-2}\mW_{\rm r}$, the resolvent
identity and \eqref{eqn:rectangular-imaginary-block-error} with
$q=\mathrm r$ give
 \begin{align}
 \|\mK(t)-[\mM_{\rm r}(\ii\beta)]^{-1}\| \leq  & \|\mK(t)\|\cdot \|\mM_{\rm r}(\ii\beta)-[\mK(t)]^{-1}\| \cdot \|[\mM_{\rm r}(\ii\beta)]^{-1}\| \\
 \leq & 2\eta^{-2}\|\mM_{\rm r}(\ii\beta)-[\mK(t)]^{-1}\|
 \leq  C_\eta\bigl(\varepsilon_{\rm r}                  +\|\mX\|_{1\to\infty}^2+\beta-t\bigr).
 \end{align}
Substituting these bounds into
\eqref{eqn:rectangular-weighted-gram-error} and using
$\|\mX\|_{1\to\infty}^2\leq\eta^{-1/2}\|\mX\|_{1\to\infty}$ yields
\begin{align}
 \|\mX^*\mK(t)\mX-\mT^*[\mM_{\rm r}(\ii\beta)]^{-1}\mT\|
 \leq C_\eta\bigl(\|\mX\|_{1\to\infty}
                         +\varepsilon_{\rm r}+\beta-t\bigr).
 \label{eqn:rectangular-weighted-gram-error-bound}
\end{align}
Finally, since $\beta>t+u\geq\rho(\widetilde{\mB})$ and $\beta^2>\overline{q_{\rm r}}$, \eqref{eqn:weighted-ihara-bass-imaginary-positivity} gives
\begin{align}
 \mT^*[\mM_{\rm r}(\ii\beta)]^{-1}\mT
 &\succeq-\beta^2\mM_{\rm c}(\ii\beta)
 =\mQ_{\rm c}-\beta^2\id_m-\beta^2\mR_{\rm c}(\beta).
 \label{eqn:block-exact-weighted-lower}
\end{align}
Applying \eqref{eqn:rectangular-imaginary-block-error} with $q=\mathrm c$ and using $\beta^2-t^2\leq C_\eta(\beta-t)$, we have
\begin{align}
 \|-\beta^2\mM_{\rm c}(\ii\beta)
                  -(\mW_{\rm c}-t^2\id_m)\|
 &\leq t^2\|\mM_{\rm c}(\ii\beta)
                         -(\id_m-t^{-2}\mW_{\rm c})\|
       +(\beta^2-t^2)\|\mM_{\rm c}(\ii\beta)\|
 \nonumber\\
 &\leq C_\eta\bigl(\varepsilon_{\rm c}
                 +\|\mX\|_{1\to\infty}^2+\beta-t\bigr),
 \label{eqn:rectangular-imaginary-block-errors}
\end{align}
where we used $\|\mM_{\rm c}(\ii\beta)\|\leq1+\beta^{-2}\|\mQ_{\rm c}\| \leq C_\eta$. Combining \eqref{eqn:rectangular-weighted-gram-error-bound},
\eqref{eqn:block-exact-weighted-lower}, and
\eqref{eqn:rectangular-imaginary-block-errors} yields
\[
 \mX^*\mK(t)\mX
 \succeq\mW_{\rm c}-t^2\id_m
   -C_\eta\bigl(\|\mX\|_{1\to\infty}
       +\varepsilon_{\rm r}+\varepsilon_{\rm c}+\beta-t\bigr)\id_m,
\]
uniformly for $t+u<\beta\leq t+u+1$. Taking the limit as $\beta$ approaches $t+u$ establishes \eqref{eqn:block-lower-quadratic-form}.
\end{proof}

\subsection{Proof of Theorem~\ref{thm:rectangular-largest-singular-value}}
\label{subsec:singular-value-proofs}
We first derive singular-value comparisons for restrictions of a parent
matrix from the Ihara--Bass inequalities.

\begin{lemma}[Singular-value comparisons for restrictions of a parent]
\label{lem:rectangular-singular-value-comparisons}
Let $\mX\in\C^{n\times m}$ be nonempty. For nonempty
$I\subset[n]$ and $J\subset[m]$, define the selected parent variances by
\begin{align}
 \overline{q_{\rm r}}(I)&=\max_{i\in I}\sum_{j=1}^m|\emX_{ij}|^2,
 &
 \overline{q_{\rm c}}(J)&=\max_{j\in J}\sum_{i=1}^n|\emX_{ij}|^2,
 &
 \underline{q_{\rm c}}(J)&=\min_{j\in J}\sum_{i=1}^n|\emX_{ij}|^2.
\end{align}
\begin{enumerate}[label=\textup{(\roman*)},leftmargin=2.5em,itemsep=0.6em]
\item For every
$t>\max\{\rho(\widetilde{\mB}(\mX)),\|\mX\|_{1\to\infty}\}$,
\begin{align}
 \sigma_{\max}(\mX_{I,J})
 &\leq t\sqrt{\left(1+\frac{\overline{q_{\rm r}}(I)}{t^2-\|\mX\|_{1\to\infty}^2}\right)
                    \left(1+\frac{\overline{q_{\rm c}}(J)}{t^2-\|\mX\|_{1\to\infty}^2}\right)}
+\frac{\|\mX\|_{1\to\infty}\sqrt{\overline{q_{\rm r}}(I)\,\overline{q_{\rm c}}(J)}}
       {t^2-\|\mX\|_{1\to\infty}^2}.
 \label{eqn:block-upper-comparison}
\end{align}
\item If $|J|\leq n$, then for every
$\beta>\max\{\rho(\widetilde{\mB}(\mX)),\sqrt{\overline{q_{\rm r}}([n])}\}$,
\begin{align}
 \sigma_{\min}(\mX_{[n],J})
 &\geq\Biggl[\sqrt{\left[\Big(\beta^2-\overline{q_{\rm r}}([n])\Big)
             \left(\frac{\underline{q_{\rm c}}(J)}{\beta^2+\|\mX\|_{1\to\infty}^2}-1\right)
                    \right]_+}
- \frac{\|\mX\|_{1\to\infty}\sqrt{\overline{q_{\rm r}}([n])\cdot \overline{q_{\rm c}}(J)}}
       {\beta^2+\|\mX\|_{1\to\infty}^2}\Biggr]_+.
 \label{eqn:block-lower-comparison}
\end{align}
If $\mX_{I^c,J}=0$ and $|I|\geq|J|$, the same bound holds for
$\sigma_{\min}(\mX_{I,J})$.
\end{enumerate}
%The sets may depend arbitrarily on the entries of $\mX$.
\end{lemma}

\begin{proof}[Proof of Lemma~\ref{lem:rectangular-singular-value-comparisons}]
For \textup{(i)}, Lemma~\ref{lem:weighted-ihara-bass-positivity}\textup{(i)}
gives positivity of the parent block matrix in
\eqref{eqn:block-ihara-entries}. Compress it to $I\sqcup J$.
Its diagonal blocks are $\mM_{\rm r}(t)_{I,I}$ and
$\mM_{\rm c}(t)_{J,J}$, with their variance sums still taken over
all columns and rows of the parent. Since
$\mM_{\rm r}(t)_{I,I}\succeq\id_{|I|}$, the positivity criterion in
Lemma~\ref{lem:schur-complement} gives, with $\mT=t\mX(t)_{I,J}$,
\[
 \mT^*[\mM_{\rm r}(t)_{I,I}]^{-1}\mT
 \prec t^2\mM_{\rm c}(t)_{J,J}.
\]
Bounding the inverse from below by
$\|\mM_{\rm r}(t)_{I,I}\|^{-1}\id_{|I|}$ and using the diagonal
formulas yields
\begin{align}
 \|\mT\|
 &\leq t\sqrt{\|\mM_{\rm r}(t)_{I,I}\|\,
                     \|\mM_{\rm c}(t)_{J,J}\|}
\leq t\sqrt{\left(1+\frac{\overline{q_{\rm r}}(I)}{t^2-\|\mX\|_{1\to\infty}^2}\right)
\left(1+\frac{\overline{q_{\rm c}}(J)}{t^2-\|\mX\|_{1\to\infty}^2}\right)}.
 \label{eqn:block-real-norm-bound}
\end{align}
For $i\in I$ and $j\in J$, the deformation formula gives
\[
 |\emX_{ij}-\emT_{ij}|
 =\frac{|\emX_{ij}|^3}{t^2-|\emX_{ij}|^2}
 \leq\frac{\|\mX\|_{1\to\infty}|\emX_{ij}|^2}{t^2-\|\mX\|_{1\to\infty}^2}.
\]
The Schur test therefore implies
\begin{align}
 \|\mX_{I,J}-\mT\|
 &\leq\frac{\|\mX\|_{1\to\infty}\sqrt{\overline{q_{\rm r}}(I)\,
                                  \overline{q_{\rm c}}(J)}}{t^2-\|\mX\|_{1\to\infty}^2}.
 \label{eqn:block-real-entry-error}
\end{align}
The triangle inequality proves \eqref{eqn:block-upper-comparison}.

For \textup{(ii)}, let $\mT=\ii\beta\mX(\ii\beta)_{[n],J}$.
The diagonal formula gives
\[
 \mM_{\rm r}(\ii\beta)
 \succeq(1-\beta^{-2}\overline{q_{\rm r}}([n]))\id_n\succ0,
 \qquad
 [\mM_{\rm r}(\ii\beta)]^{-1}
 \preceq(1-\beta^{-2}\overline{q_{\rm r}}([n]))^{-1}\id_n.
\]
Compressing the parent weighted Gram inequality
\eqref{eqn:block-exact-weighted-lower} to $J$ yields
\[
 \mT^*[\mM_{\rm r}(\ii\beta)]^{-1}\mT
 \succeq-\beta^2\mM_{\rm c}(\ii\beta)_{J,J}.
\]
For each $j\in J$,
\[
 -[\mM_{\rm c}(\ii\beta)]_{jj}
 =\sum_i\frac{|\emX_{ij}|^2}{\beta^2+|\emX_{ij}|^2}-1
 \geq\frac{\underline{q_{\rm c}}(J)}{\beta^2+\|\mX\|_{1\to\infty}^2}-1.
\]
Combining these bounds gives
\[
 (1-\beta^{-2}\overline{q_{\rm r}}([n]))^{-1}\mT^*\mT
 \succeq\beta^2\left(\frac{\underline{q_{\rm c}}(J)}{\beta^2+\|\mX\|_{1\to\infty}^2}-1\right)
       \id_{|J|}.
\]
Since $n\geq|J|$, taking the least eigenvalue and using
$\mT^*\mT\succeq0$ proves
\begin{align}
 \sigma_{\min}(\mT)^2
 &\geq\left[(\beta^2-\overline{q_{\rm r}}([n]))
       \left(\frac{\underline{q_{\rm c}}(J)}{\beta^2+\|\mX\|_{1\to\infty}^2}-1\right)\right]_+.
 \label{eqn:block-imaginary-singular-value-bound}
\end{align}
For the deformation error, $\beta>\sqrt{\overline{q_{\rm r}}([n])}\geq\|\mX\|_{1\to\infty}$,
so $a\mapsto a/(\beta^2+a^2)$ is increasing on $[0,\|\mX\|_{1\to\infty}]$.
Consequently, for $i\in[n]$ and $j\in J$,
\[
 |\emX_{ij}-\emT_{ij}|
 =\frac{|\emX_{ij}|^3}{\beta^2+|\emX_{ij}|^2}
 \leq\frac{\|\mX\|_{1\to\infty}|\emX_{ij}|^2}{\beta^2+\|\mX\|_{1\to\infty}^2}.
\]
The Schur test gives
\begin{align}
 \|\mX_{[n],J}-\mT\|
 &\leq\frac{\|\mX\|_{1\to\infty}\sqrt{\overline{q_{\rm r}}([n])\,
                                  \overline{q_{\rm c}}(J)}}{\beta^2+\|\mX\|_{1\to\infty}^2}.
 \label{eqn:block-imaginary-entry-error}
\end{align}
By Lemma~\ref{lem:weyl} and nonnegativity of singular values,
\[
 \sigma_{\min}(\mX_{[n],J})
 \geq[\sigma_{\min}(\mT)-\|\mX_{[n],J}-\mT\|]_+,
\]
which proves \eqref{eqn:block-lower-comparison}.
If $\mX_{I^c,J}=0$, then
$\mX_{[n],J}^*\mX_{[n],J}=\mX_{I,J}^*\mX_{I,J}$, so the
least singular values agree when $|I|\geq|J|$.
All steps are pointwise, allowing arbitrary dependence of the selected
sets on the parent entries.
\end{proof}

Now we are ready to prove Theorem~\ref{thm:rectangular-largest-singular-value}. We  first apply the deterministic comparison to the full centered matrix $\mX$ and optimize the deterministic bound. We then control the selected row and column variances using the cutoffs and Bennett's inequality. Finally, the nonbacktracking radius estimate Proposition~\ref{prop:nonbacktracking-radius-bounds}\ref{item:nonbacktracking-bipartite} yields the threshold $\sigma_{\rm R}$. The probability estimates are only involved in the last two steps.

\begin{proof}[Proof of Theorem~\ref{thm:rectangular-largest-singular-value}]
\emph{Step 1: Deterministic comparison.}
If either active set is empty, the norm is zero and the bound is
immediate. Otherwise, suppose, for some $u_{\rm r},u_{\rm c}\geq0$, that
\begin{align}
 \max_{i\in I_\star}\sum_{j=1}^m|\emX_{ij}|^2&\leq u_{\rm r}^2,
 &\max_{j\in J_\star}\sum_{i=1}^n|\emX_{ij}|^2&\leq u_{\rm c}^2.
 \label{eqn:rectangular-parent-variance-envelope}
\end{align}
For $r>\rho(\widetilde{\mB}(\mX))$ and $a,b\geq0$, define the function
\[
 F_r(a,b)\coloneqq\inf_{t\geq r}\sqrt{t^2+a+b+ab/t^2}.
\]
Differentiating in $t^2$ gives the optimizer
$\max\{r,(ab)^{1/4}\}$ and
\begin{align}
 F_r(a,b)&=
 \begin{cases}
  \sqrt a+\sqrt b,&ab\geq r^4,\\
  \sqrt{(r^2+a)(r^2+b)}/r,&ab<r^4.
 \end{cases}
 \label{eqn:rectangular-upper-optimization}
\end{align}
We choose $t_*=\max\{r,\sqrt{u_{\rm r}u_{\rm c}}\}$ and $\lambda=\sqrt{t_*^2+d^{-1}}$. Since $\|\mX\|_{1\to\infty}\leq d^{-1/2}$, we have
\[
 \lambda>\max\{\rho(\widetilde{\mB}(\mX)),\|\mX\|_{1\to\infty}\},\qquad
 \lambda^2-\|\mX\|_{1\to\infty}^2\geq t_*^2,\qquad
 u_{\rm r}u_{\rm c}/t_*^2\leq1.
\]
Applying \eqref{eqn:block-upper-comparison} with $t=\lambda$, $\overline{q_{\rm r}}(I) \leq u_{\rm r}^2$ and $\overline{q_{\rm c}}(J) \leq u_{\rm c}^2$, together with preceding inequalities, we obtain
\begin{align}
 \|\mX^{(\tau)}\|
 &\leq\lambda\sqrt{
       \left(1+\frac{u_{\rm r}^2}{t_*^2}\right)
       \left(1+\frac{u_{\rm c}^2}{t_*^2}\right)}
       +d^{-1/2}
\leq\frac{\lambda}{t_*}
       \sqrt{t_*^2+u_{\rm r}^2+u_{\rm c}^2
                    +\frac{u_{\rm r}^2u_{\rm c}^2}{t_*^2}}
       +d^{-1/2}
 \nonumber\\
 &=\sqrt{1+d^{-1}/t_*^2}\,\,
       F_r(u_{\rm r}^2,u_{\rm c}^2)+d^{-1/2}.
 \label{eqn:rectangular-parent-optimized-bound}
\end{align}

\smallskip\noindent
\emph{Step 2: Variance bounds.}
We now control $u_{\rm c}$ and $u_{\rm r}$. Let $h(t)=(1+t)\log(1+t)-t$. For $q\in\{\rm r,c\}$, define
\begin{align}
 a_q&\coloneqq\overline v_q\,
       h^{-1}\bigl(4\log(N)/(d\overline v_q)\bigr).
 \label{eqn:rectangular-bennett-variance-increments}
\end{align}
For each column, the centered summands
$d|\emX_{ij}|^2-p_{ij}(1-p_{ij})$ are bounded by $1$ in absolute
value and have total variance at most $d\colvarmax$.
Bennett's inequality, Lemma~\ref{lem:Bennett}, therefore gives
\begin{align}
 \P\left(\sum_{i=1}^n|\emX_{ij}|^2>\colvarmax+a_{\rm c}\right)
 &\leq\exp\left[-d\colvarmax\,
                    h(a_{\rm c}/\colvarmax)\right]=N^{-4}.
 \label{eqn:rectangular-largest-variance-tails}
\end{align}
The row analogue and a union bound give
$\P(\mathcal E_{\rm en}^{\rm c})\leq N^{-3}$ for the event
\[
 \mathcal E_{\rm en}\coloneqq\left\{
 \max_{i\in[n]}\sum_{j=1}^m|\emX_{ij}|^2\leq\rowvarmax+a_{\rm r},
 \quad
 \max_{j\in[m]}\sum_{i=1}^n|\emX_{ij}|^2\leq\colvarmax+a_{\rm c}
 \right\}.
\]
Every retained vertex also passed the original upper-degree check.
Using $(\widetilde{\emA}_{ij}-p_{ij})^2\leq
\widetilde{\emA}_{ij}+p_{ij}^2$ and the uniform sparsity bound, on $\mathcal E_{\rm en}$, we have
\begin{align}
 u_q^2&\coloneqq\min\{U_q/d,\overline v_q+a_q\}+C_0p_\star,
 \qquad q\in\{\rm r,c\}.
 \label{eqn:rectangular-largest-variance-cutoff}
\end{align}
Let $s=\sqrt{\log(N)/d}$ and $g_q=[U_q-\overline d_q]_+/d$, then $\varepsilon_N=\min\{g_{\rm r}+g_{\rm c},s\}$.
Using $h(t)\geq t^2/[2(1+t/3)]$, we have
\begin{align}
 a_q&\leq\sqrt{8\overline v_q}\,s+\tfrac83s^2,
 &\sqrt{\overline v_q+a_q}&\leq\sqrt{\overline v_q}+2s.
 \label{eqn:rectangular-bennett-variance-scale}
\end{align}
Since $0\leq\overline d_q/d-\overline v_q\leq C_0p_\star$, and $g_{\rm r}+g_{\rm c}\leq s$ , we have
\begin{align}
 0\leq u_q^2-\overline{\rho_q}^{\,2}
 &\leq\min\{g_q+2C_0p_\star,a_q+C_0p_\star\}
 \leq C(\varepsilon_N+\varepsilon_N^2+p_\star).
 \label{eqn:rectangular-upper-budget-errors}
\end{align}

\smallskip\noindent
\emph{Step 3: Radius bound.}
Apply Proposition~\ref{prop:nonbacktracking-radius-bounds}~\ref{item:nonbacktracking-bipartite}
with full caps, which satisfy \eqref{eqn:rectangular-edge-upper-cutoffs}
for large $N$, zero lower cutoffs, and $\nu=2$. Then
$\P(\mathcal E_{\rm NB}^{\rm c})\leq N^{-2}$, where
\[
 \mathcal E_{\rm NB}\coloneqq
 \left\{\rho(\widetilde{\mB}(\mX))<r\right\},
 \qquad
 r\coloneqq(\rowvarmax\colvarmax)^{1/4}+(C+1)d^{-1/2}.
\]
We work on the event $\mathcal E_{\rm en}\cap\mathcal E_{\rm NB}$.
The profile bounds and
$(\rowvarmax\colvarmax)^{1/2}\geq1-p_\star$ give
$c\leq r\leq C$ and $t_*\geq c$.
By \eqref{eqn:rectangular-upper-optimization}, $F_r$ is uniformly
Lipschitz in $a,b$ on bounded sets, including zero budgets, and
$1$-Lipschitz in $r$. Thus \eqref{eqn:rectangular-upper-budget-errors}
and the definition of $\sigma_{\rm R}$ give, for $\varepsilon_N\leq1$,
\begin{align}
 F_r(u_{\rm r}^2,u_{\rm c}^2)
 &\leq\sigma_{\rm R}+C(\varepsilon_N+p_\star+d^{-1/2}).
 \label{eqn:rectangular-upper-threshold-error}
\end{align}
For $\varepsilon_N>1$, the same estimate follows from
$u_{\rm r}+u_{\rm c}\leq C\varepsilon_N$ and
$F_r(u_{\rm r}^2,u_{\rm c}^2)\leq r+u_{\rm r}+u_{\rm c}
\leq C\varepsilon_N$.
Substitution into \eqref{eqn:rectangular-parent-optimized-bound}, using
$\sqrt{1+d^{-1}/t_*^2}\leq1+C/d$,
$p_\star=o(d^{-1/2})$, and $\sigma_{\rm R}=O(1)$, proves
\eqref{eqn:largest-singular-edge} with failure probability at most
$N^{-3}+N^{-2}\leq N^{-1}$, uniformly over the admissible cutoffs.
\end{proof}

\subsection{Proof of Theorem~\ref{thm:rectangular-smallest-singular-value}}
\label{subsec:rectangular-smallest-singular-value-proof}
The proof proceeds as follows. We first reduce the desired lower bound to a deterministic comparison for the rescaled matrix $\mZ$, in terms of its row and column variances and nonbacktracking radius. We then control variances of the selected rows and columns after rescaling via concentration estimates. Finally, the proof is established via the radius estimates using Proposition~\ref{prop:nonbacktracking-radius-bounds}~\ref{item:nonbacktracking-rescaled}. The probability esimates are only involved in the last two steps.

\begin{proof}[Proof of Theorem~\ref{thm:rectangular-smallest-singular-value}]
\emph{Step 1: Reduction to the rescaled matrix $\mZ$.}
Assume that the active sets are nonempty and, for some
$l_{\rm c}>u_{\rm r}>0$,
\[
 \max_i\sum_j|\emZ_{ij}|^2\leq u_{\rm r}^2,\qquad
 l_{\rm c}^2\leq\min_{j\in J_\star}\sum_i|\emZ_{ij}|^2
 \leq\max_j\sum_i|\emZ_{ij}|^2\leq\colvarmin.
\]
Since $L_{\rm r}=0$, $\mZ$ vanishes outside the rows $I_\star$
passing the original upper cutoff. Thus $|I_\star|>|J_\star|$ since
\[
 |J_\star|l_{\rm c}^2
 \leq\sum_{i\in I_\star}\sum_{j\in J_\star}|\emZ_{ij}|^2
 \leq|I_\star|u_{\rm r}^2,
\]
Since $\mZ_{I_\star,J_\star}=\mX^{(\tau)}\diag(f_j:j\in J_\star)$, with $0<f_j\leq1$, we have
\[
 \sigma_{\min}(\mX^{(\tau)})
 \geq\sigma_{\min}(\mZ_{I_\star,J_\star})
 =\sigma_{\min}(\mZ_{[n],J_\star}).
\]
We fix $\beta>\max\{\rho(\widetilde{\mB}(\mZ)), \sqrt{u_{\rm r}l_{\rm c}}\}>u_{\rm r}$ and denote by
\[
 A_\beta=(\beta^2-u_{\rm r}^2)
                  \left(\frac{l_{\rm c}^2}{\beta^2}-1\right),
 \qquad
 B_\beta=(\beta^2-u_{\rm r}^2)
                  \left(\frac{l_{\rm c}^2}{\beta^2+d^{-1}}-1\right).
\]
By applying \eqref{eqn:block-lower-comparison} in Lemma~\ref{lem:rectangular-singular-value-comparisons} and the facts that $\|\mZ\|_{1\to\infty}\leq d^{-1/2}$ and $\beta^2\geq u_{\rm r}l_{\rm c}$, we have
\[
 \sigma_{\min}(\mX^{(\tau)})
 \geq\sqrt{[B_\beta]_+}
       -\beta^{-2}d^{-1/2}u_{\rm r}\sqrt{\colvarmin}
 \geq\sqrt{[B_\beta]_+}-Cd^{-1/2}.
\]
Moreover, $0\leq A_\beta-B_\beta \leq l_{\rm c}(u_{\rm r}d)^{-1} = O(d^{-1})$. Using $|\sqrt{[x]_+}-\sqrt{[y]_+}|\leq\sqrt{|x-y|}$, we obtain
\[
 \sigma_{\min}(\mX^{(\tau)})
 \geq\sqrt{[A_\beta]_+}-Cd^{-1/2}.
\]
Completing the square and using $\beta\geq\sqrt{u_{\rm r}l_{\rm c}}$, we have
\[
\begin{aligned}
 A_\beta
 &=l_{\rm c}^2+u_{\rm r}^2-\beta^2
                     -\frac{u_{\rm r}^2l_{\rm c}^2}{\beta^2}
 =(l_{\rm c}-u_{\rm r})^2
       -\left(\beta-\frac{u_{\rm r}l_{\rm c}}{\beta}\right)^2,\\
 0\leq\beta-\frac{u_{\rm r}l_{\rm c}}{\beta}
 &=(\beta-\sqrt{u_{\rm r}l_{\rm c}})
       \left(1+\frac{\sqrt{u_{\rm r}l_{\rm c}}}{\beta}\right)
 \leq2(\beta-\sqrt{u_{\rm r}l_{\rm c}}).
\end{aligned}
\]
Therefore $\sqrt{[x^2-y^2]_+}\geq x-y$ for $x,y\geq0$ yields
\[
 \sigma_{\min}(\mX^{(\tau)})
 \geq l_{\rm c}-u_{\rm r}
       -2(\beta-\sqrt{u_{\rm r}l_{\rm c}})-Cd^{-1/2}.
\]
Letting $\beta\downarrow
\max\{\rho(\widetilde{\mB}(\mZ)),\sqrt{u_{\rm r}l_{\rm c}}\}$
now gives
\begin{align}
 \sigma_{\min}(\mX^{(\tau)})
 &\geq l_{\rm c}-u_{\rm r}
 -2\left[\rho(\widetilde{\mB}(\mZ))
                   -\sqrt{u_{\rm r}l_{\rm c}}\right]_+
 -Cd^{-1/2}.
 \label{eqn:rectangular-rescaled-singular-transfer}
\end{align}

\smallskip\noindent
\emph{Step 2: Establish $l_{\rm c}>u_{\rm r}$.}
We only consider the case $\rowvarmax<\colvarmin$; otherwise the lower bound in \eqref{eqn:smallest-singular-edge} is trivial. Recall $\rowvarmax,\colvarmin\asymp1$ and $p_\star=O(d/N)=o(d^{-1/2})$. In the regularized case, define
\[
 \mathcal E_{\rm en}\coloneqq\left\{
 I_\star\ne\varnothing,\quad |J_\star|\geq cm,\quad
 \max_{i\in[n]}\sum_{j=1}^m|\emX_{ij}|^2
       \leq\rowvarmax+a_{\rm r}\right\}.
\]
Since $1\leq L_{\rm c}\leq\underline{\dcol}$ for large $N$ by
\eqref{eqn:rectangular-edge-lower-cutoff}, Lemma~\ref{lem:rectangular-column-retention} and the preceding Bennett estimate give $\P(\mathcal E_{\rm en}^{\rm c})\leq Cd^{-3}+e^{-cm}+N^{-3}$. On $\mathcal E_{\rm en}$, the variance bounds in Step~1 hold with
\begin{align}
 u_{\rm r}^2&=\min\{U_{\rm r}/d,\rowvarmax+a_{\rm r}\}+C_0p_\star,
 &l_{\rm c}^2&=(1-p_\star)^2L_{\rm c}/d,
 \label{eqn:rectangular-bennett-lower-amplitudes}
\end{align}
where $a_{\rm r}$ is from \eqref{eqn:rectangular-bennett-variance-increments}.
Indeed, column contraction decreases row variances and preserves
the bound $Q_j\geq l_{\rm c}^2$, since
$l_{\rm c}^2\leq(1-p_\star)\underline{\dcol}/d\leq\colvarmin$.

In the unthresholded case, $I_\star=[n]$, $J_\star=[m]$, and
$\varepsilon_N=\sqrt{\log(N)/d}$. Define instead
\[
 \mathcal E_{\rm en}\coloneqq\left\{
 \max_{i\in[n]}\sum_{j=1}^m|\emX_{ij}|^2\leq\rowvarmax+C\varepsilon_N,
 \quad
 \min_{j\in[m]}\sum_{i=1}^n|\emX_{ij}|^2\geq\colvarmin-C\varepsilon_N
 \right\}.
\]
Since $d\geq c_{\rm deg}\log(N)$ in \eqref{eqn:rectangular-lower-edge-sparsity}, Bennett's inequality yields $\P(\mathcal E_{\rm en}^{\rm c})\leq2N^{-3}$. On the event $\mathcal E_{\rm en}$, the variance bounds in Step~1 hold with
\begin{align}
 u_{\rm r}^2&=\rowvarmax+C\varepsilon_N,
 &l_{\rm c}^2&=[\colvarmin-C\varepsilon_N]_+,
 \label{eqn:unthresholded-rectangular-variance-event}
\end{align}
In both cases the cutoff and concentration estimates yield
\begin{align}
 |u_{\rm r}-\sqrt{\rowvarmax}|+
 |l_{\rm c}-\sqrt{\colvarmin}|
 &\leq C(\varepsilon_N+p_\star).
 \label{eqn:rectangular-coarse-amplitude-errors}
\end{align}
Here the upper-cutoff margins give $\sqrt{\log d/d}\leq C\varepsilon_N$,
and $\sqrt v-\sqrt{[v-t]_+}\leq t/\sqrt v$ handles the positive part.
For sufficiently large $K$, whenever the target in
\eqref{eqn:smallest-singular-edge} is positive,
\[
 l_{\rm c}-u_{\rm r}
 \geq\sqrt{\colvarmin}-\sqrt{\rowvarmax}
       -C(\varepsilon_N+p_\star)>0.
\]
\smallskip\noindent
\emph{Step 3: Radius bound.}
Proposition~\ref{prop:nonbacktracking-radius-bounds}~\ref{item:nonbacktracking-rescaled}
with $\nu=2$ gives $\P(\mathcal E_{\rm NB}^{\rm c})\leq N^{-2}$ for
\[
 \mathcal E_{\rm NB}\coloneqq\left\{
 \rho(\widetilde{\mB}(\mZ))
 \leq(\rowvarmax\colvarmin)^{1/4}
       +C\bigl(\sqrt{\log d/d}+d^{-1/2}\bigr)\right\}.
\]
We work on the event $\mathcal E_{\rm en}\cap\mathcal E_{\rm NB}$.
Since $\sqrt{\log d/d}\leq C\varepsilon_N$,
\eqref{eqn:rectangular-coarse-amplitude-errors} gives
\begin{align}
 \rho(\widetilde{\mB}(\mZ))
 \leq\sqrt{u_{\rm r}l_{\rm c}}+C(\varepsilon_N+d^{-1/2}).\label{eqn:matching-upperbound-rho-B}
\end{align}
Together with \eqref{eqn:rectangular-rescaled-singular-transfer}, \eqref{eqn:rectangular-coarse-amplitude-errors} and \eqref{eqn:matching-upperbound-rho-B}, this proves \eqref{eqn:smallest-singular-edge}. The total failure probability is at most $Cd^{-3}+e^{-cm}+N^{-1}$ in the regularized case and $N^{-1}$ in the unthresholded case.
\end{proof}

\begin{remark}
\label{rem:block-radius-obstruction}
Variance separation $\rowvarmax<\colvarmin$ gives the positive leading term $\sqrt{\colvarmin}-\sqrt{\rowvarmax}$. To obtain this lower bound from \eqref{eqn:rectangular-rescaled-singular-transfer} also requires \eqref{eqn:matching-upperbound-rho-B}, as ensured by Proposition~\ref{prop:nonbacktracking-radius-bounds} \ref{item:nonbacktracking-rescaled} in our proof.
If \eqref{eqn:matching-upperbound-rho-B} does not hold, then even with variance separation and a coarse radius bound, it is not enough to guarantee the lower bound $l_{\rm c}-u_{\rm r}$. See the counterexample in Subsection~\ref{subsec:fourier-radius-counterexample} for more deails.
\end{remark}

\subsection{Proof of Theorem~\ref{thm:hermitian-edge}}
\label{subsec:eigenvalue-proofs}
We begin by establishing the Hermitian comparison result. Related methods have previously been used in \cite{benaych2020spectral, alt2021extremal, alt2021delocalization}.

\begin{proposition}[Hermitian comparison for a restriction of a parent]
\label{prop:hermitian-comparison}
Let $\mH\in\C^{N\times N}$ be Hermitian, let
$I\subset[N]$ be nonempty, and suppose
$v\geq\max_{i\in I}\sum_{j\in[N]}|\emH_{ij}|^2$.  Then
\begin{align}
 \|\mH_{I,I}\|
 &\leq\inf_{t>\rho(\mB(\mH))}(t+v/t)
       +\|\mH\|_{1\to\infty}.
 \label{eqn:hermitian-comparison}
\end{align}
\end{proposition}

\begin{proof}[Proof of Proposition~\ref{prop:hermitian-comparison}]
Fix $\lambda>\max\{\rho(\mB(\mH)),\|\mH\|_{1\to\infty}\}$ and set
\[
 \mV\coloneqq\diag\left(\sum_j|\emH_{ij}|^2:i\in[N]\right).
\]
Since $\mB(-\mH)=-\mB(\mH)$,
Lemma~\ref{lem:weighted-ihara-bass-positivity}\textup{(i)} applied
to $\pm\mH$ gives $\mM(\lambda)\mp\mH(\lambda)\succ0$.
Moreover, \eqref{eqn:weighted-ihara-bass-matrix} gives
\[
 |\lambda\emH_{ij}(\lambda)-\emH_{ij}|
 \leq\|\mH\|_{1\to\infty}|\emH_{ij}|^2
       (\lambda^2-\|\mH\|_{1\to\infty}^2)^{-1}.
\]
Summing against $|\evx_i\evx_j|$, using
$2|\evx_i\evx_j|\leq|\evx_i|^2+|\evx_j|^2$, and applying the
diagonal formula in \eqref{eqn:weighted-ihara-bass-matrix} yields
\begin{align}
 \pm\mH
 &\preceq\lambda\mM(\lambda)+\|\mH\|_{1\to\infty}
       (\lambda^2-\|\mH\|_{1\to\infty}^2)^{-1}\mV
\preceq\lambda\id_N+(\lambda-\|\mH\|_{1\to\infty})^{-1}\mV.
\end{align}
Since $\mV_{I,I}\preceq v\id_{|I|}$, the proof of \eqref{eqn:hermitian-comparison} is complete by taking $\lambda=t+\|\mH\|_{1\to\infty}$ for $t>\rho(\mB(\mH))$.
\end{proof}

Now, we are ready to prove the main result.
\begin{proof}[Proof of Theorem~\ref{thm:hermitian-edge}]
\emph{Step 1: Deterministic comparison.}
If $I_\star$ is empty, the norm estimate is immediate with operator
norm zero. Otherwise, suppose, for some $u\geq0$, that
\begin{align}
 \max_{i\in I_\star}\sum_{j=1}^N|\emH_{ij}|^2&\leq u^2.
 \label{eqn:hermitian-active-variance-envelope}
\end{align}
For $r\geq1$ and $v\geq0$, define
\[
 G_r(v)\coloneqq\inf_{t\geq r}(t+v/t)
 =\begin{cases}
  2\sqrt v,&v\geq r^2,\\
  r+v/r,&v<r^2.
 \end{cases}
\]
The optimizer is $\max\{r,\sqrt v\}$. Since
$\|\mH\|_{1\to\infty}\leq d^{-1/2}$,
Proposition~\ref{prop:hermitian-comparison}, applied to the full
parent $\mH$ and its restriction to $I_\star$, gives
\begin{align}
 \|\mH^{(\tau)}\|&\leq G_r(u^2)+d^{-1/2},
 \qquad r>\rho(\mB(\mH)).
 \label{eqn:hermitian-parent-optimized-bound}
\end{align}

\smallskip\noindent
\emph{Step 2: Variance bounds.}
Let $h(t)=(1+t)\log(1+t)-t$ and define
\begin{align}
 a&\coloneqq\hermvarmax\,
       h^{-1}\bigl(4\log(N)/(d\hermvarmax)\bigr),
 &u^2&\coloneqq\min\{U/d,\hermvarmax+a\}+C_0p_\star.
 \label{eqn:hermitian-bennett-variance-cutoff}
\end{align}
For each fixed $i$, the centered summands
$d|\emH_{ij}|^2-p_{ij}(1-p_{ij})$ are independent, bounded by $1$
in absolute value, and have total variance at most $d\hermvarmax$.
Bennett's inequality, as in \eqref{eqn:rectangular-largest-variance-tails},
and a union bound give an event $\mathcal E_{\rm en}$ such that
\[
 \max_i\sum_j|\emH_{ij}|^2\leq\hermvarmax+a,
 \qquad \P(\mathcal E_{\rm en}^{\rm c})\leq N^{-3}.
\]
Every retained vertex passed the original upper-degree check, so
$(\emA_{ij}-p_{ij})^2\leq\emA_{ij}+p_{ij}^2$ and
$\sum_jp_{ij}^2/d\leq C_0p_\star$ also give the bound
$U/d+C_0p_\star$. Hence \eqref{eqn:hermitian-active-variance-envelope}
holds on $\mathcal E_{\rm en}$ with the above choice of $u$.
We write $g=[U-d]_+/d$ and $s=\sqrt{\log(N)/d}$, hence
$\varepsilon_N=\min\{g,s\}$. Since
$1-p_\star\leq\hermvarmax\leq1$ and
$a\leq C(s+s^2)$ by \eqref{eqn:rectangular-bennett-variance-scale},
\begin{align}
 0\leq u^2-\overline\rho^{\,2}
 &\leq\min\{g+(C_0+1)p_\star,a+C_0p_\star\}
 \leq C(\varepsilon_N+\varepsilon_N^2+p_\star).
 \label{eqn:hermitian-upper-budget-error}
\end{align}

\smallskip\noindent
\emph{Step 3: Radius bound.}
Apply Proposition~\ref{prop:nonbacktracking-radius-bounds}~\ref{item:nonbacktracking-hermitian} to the original centered matrix $\mH$ with $\nu=2$. This gives an event $\mathcal E_{\rm NB}$ with $\P(\mathcal E_{\rm NB}^{\rm c})\leq N^{-2}$ on which
\[
 \rho(\mB(\mH))<r\coloneqq1+(C+1)d^{-1/2}.
\]
We work on the event $\mathcal E_{\rm en}\cap\mathcal E_{\rm NB}$. The function $G_r(v)$ is $1$-Lipschitz in each variable for $r\geq1$, $v\geq0$, and $G_1(\overline\rho^{\,2})=\lambda_{\rm H}$. Thus \eqref{eqn:hermitian-upper-budget-error} gives, for $\varepsilon_N\leq1$,
\[
 G_r(u^2)\leq\lambda_{\rm H}
       +C(\varepsilon_N+p_\star+d^{-1/2}).
\]
For $\varepsilon_N>1$, the same estimate follows from
$u\leq C\varepsilon_N$ and $G_r(u^2)\leq r+2u\leq C\varepsilon_N$.
Substituting into \eqref{eqn:hermitian-parent-optimized-bound} and
using $p_\star=O(d/N)=o(d^{-1/2})$ proves the norm bound
\eqref{eqn:hermitian-edge-probability}, with failure probability
\begin{align}
 \P((\mathcal E_{\rm en}\cap\mathcal E_{\rm NB})^{\rm c})
 &\leq N^{-3}+N^{-2}\leq N^{-1},
 \label{eqn:hermitian-failure-probability}
\end{align}
uniformly over the admissible cutoffs.
\end{proof}

\subsection{Proof of Proposition~\ref{prop:nonbacktracking-radius-bounds}}
\label{subsec:trace-to-radius}
To establish the proof, we bound trace moments of high powers of the
nonbacktracking matrices and apply Markov's inequality. We first
summarize the relevant estimates for the bipartite degree-cutoff model,
column rescaling,
and the original Hermitian parent.

\begin{lemma}[Trace moments for nonbacktracking radii]
\label{lem:uniform-trace-consequence}
There exist constants $c,C,C_e>0$ such that, for all sufficiently large $N$,
the following holds. The constants depend only on the fixed
hypothesis constants, including the corresponding profile constant
$C_0$.
\begin{enumerate}[label=\textup{(\roman*)},ref=\textup{(\roman*)},leftmargin=2.5em,itemsep=0.6em]
\item \emph{Bipartite model.}
Assume the hypotheses of
Proposition~\ref{prop:nonbacktracking-radius-bounds}\ref{item:nonbacktracking-bipartite},
and choose $R=(\rowvarmax\colvarmax)^{1/4}$.
Then, for
$\widetilde{\mB}=\widetilde{\mB}(\widehat{\mX})$ and every integer
$3\leq\ell\leq\lfloor c\sqrt{d}\log(N)\rfloor$,
\begin{align}
 \E\Tr\widetilde{\mB}^{\ell}(\widetilde{\mB}^{\ell})^*
 &\leq C\ell^7d\,nm\,R^{2\ell}
       \exp(C_e\ell d^{-3}).
 \label{eqn:uniform-bipartite-trace-consequence}
\end{align}

\item \emph{Rescaled bipartite model.}
Assume the hypotheses of
Proposition~\ref{prop:nonbacktracking-radius-bounds}\ref{item:nonbacktracking-rescaled},
and choose $R=(\rowvarmax\colvarmin)^{1/4}$.
For $\widetilde{\mB}=\widetilde{\mB}(\mZ)$, with $\mZ$ defined in
\eqref{eqn:rescaled-column-matrix}, and every integer
$3\leq\ell\leq\lfloor c\sqrt d\log(N)\rfloor$,
\begin{align}
 \E\Tr\widetilde{\mB}^{\ell}(\widetilde{\mB}^{\ell})^*
 &\leq C\ell^7d\,nm\,R^{2\ell}
       \exp\!\left(C_e\ell\sqrt{\log d/d}\right).
 \label{eqn:uniform-rescaled-trace-consequence}
\end{align}

\item \emph{Hermitian parent matrix.}
Assume the hypotheses of
Proposition~\ref{prop:nonbacktracking-radius-bounds}\ref{item:nonbacktracking-hermitian}.
Then, for
$\mB=\mB(\mH)$ and every integer
$3\leq\ell\leq\lfloor c\sqrt d\log(N)\rfloor$,
\begin{align}
 \E\Tr\mB^\ell(\mB^\ell)^*
 &\leq C\ell^7dN^2.
 \label{eqn:uniform-hermitian-trace-consequence}
\end{align}
If, in addition, \eqref{eqn:rectangular-lower-edge-sparsity} holds,
then the exponential factor in \textup{(i)}
converges uniformly to one over the stated range of $\ell$, and may
therefore be absorbed into $C$.
\end{enumerate}
\end{lemma}

\begin{remark}[General trace bounds]
\label{rem:model-trace-general-version}
In Section~\ref{sec:trace-estimates},
Lemma~\ref{lem:general-nonbacktracking-trace-bounds} proves a more
general version for matrices satisfying graph moment assumptions.
The estimates above follow using the moment verifications in
Section~\ref{sec:model-moment-verification} and the independence and
centering of the original Hermitian upper-triangular entries.
\end{remark}

\begin{proof}[Proof of Lemma~\ref{lem:uniform-trace-consequence}]
The proof has two steps: verify the hypotheses of
Lemma~\ref{lem:general-nonbacktracking-trace-bounds} using
Section~\ref{sec:model-moment-verification}, then apply its trace bounds.

\smallskip\noindent
\emph{Step 1: Verify the hypotheses.}
Use sparsity scale $s=R^2d$ for
$\widehat{\mX}/R$ and $\mZ/R$, with $R$ as stated in
\textup{(i)} and \textup{(ii)}, and use $s=d$ for $\mH$.
The uniform sparsity bound \eqref{eqn:rectangular-uniform-sparsity-bound}
gives $\rowvarmax\leq C_{\rm sp}\sqrt{m/n}$ and
$\colvarmax\leq C_{\rm sp}\sqrt{n/m}$.
Since $\colvarmin\leq\colvarmax$, all three normalizing factors satisfy
\begin{align}
 0<r\leq R\leq R_*\coloneqq\max\{1,\sqrt{C_{\rm sp}}\}
 \label{eqn:normalization-radius-bounds}
\end{align}
for some fixed $r>0$, including $R=1$ in the Hermitian case.
Thus $s\asymp d$.

Under \eqref{eqn:rectangular-edge-upper-cutoffs}, with zero lower
cutoffs, Lemma~\ref{lem:bipartite-cutoff-full-trace-verification}
verifies the graph moment condition
\ref{eqn:bipartite-graph-condition} for $\widehat{\mX}/R$.
For column rescaling, the corresponding bounds
are supplied by Lemma~\ref{lem:rescaled-column-moment-verification}.
For the original Hermitian parent $\mH$, independence and centering
make every singleton word expectation vanish. For plural words,
\eqref{eqn:shared-bernoulli-centered-moments} and
Lemma~\ref{lem:weighted-variance-label-sum}\textup{(ii)}, with weights
$w_{uv}=p_{uv}(1-p_{uv})/d$, verify
\ref{eqn:hermitian-graph-condition} with
$K_\ell=\Gamma_\ell=1$ and $\chi_\ell$ from
\eqref{eqn:hermitian-verification-chi}. Indeed, these weights have
row sums at most one and maximum at most $C_0/N$, so
$\chi_\ell\leq C_0$.
For sufficiently small $c>0$, with $c\leq b_{\rm tr}r$ in the
bipartite degree-cutoff case, these verifications apply throughout
$3\leq\ell\leq\lfloor c\sqrt d\log(N)\rfloor$.
Writing $K_\ell,\chi_\ell,\Gamma_\ell$ for the applicable
parameters, with tildes in the bipartite cases, they give
\[
 \chi_\ell,\Gamma_\ell\leq C,\qquad
 K_\ell\leq
 \begin{cases}
  C\exp(C_e\ell d^{-3}),&\text{in \textup{(i)}},\\
  C\exp\!\left(C_e\ell\sqrt{\log d/d}\right),&\text{in \textup{(ii)}},\\
  1,&\text{in \textup{(iii)}}.
 \end{cases}
\]

It remains to check the numerical hypotheses of the general trace
lemma. Let $c_0,d_0$ be its constants and increase $C_{\rm NB}$ so
that $s\geq r^2C_{\rm NB}\geq d_0$.
The bounded graph parameters and $s\asymp d$ give, uniformly
over the stated range,
\[
 \frac{C_{\rm enum}\chi_\ell\ell^3\Gamma_\ell^6s^{9/2}}{N}
 \leq Cc^3\frac{d^6(\log(N))^3}{N}
 \leq Cc^3N^{-\delta_{\rm NB}}
 \leq N^{-\delta_{\rm NB}/2}\leq\frac12
\]
for sufficiently large $N$, using the tilded parameters when
appropriate. The logarithm of the reciprocal is therefore at least
$\frac12\delta_{\rm NB}\log(N)$.
Decreasing $c$ further so that $c\leq c_0r\delta_{\rm NB}/2$
gives
\[
 \ell\leq c\sqrt d\log(N)
 \leq\frac{c_0\delta_{\rm NB}}2\sqrt s\log(N)
 \leq c_0\sqrt s\log
 \frac{N}{C_{\rm enum}\chi_\ell\ell^3\Gamma_\ell^6s^{9/2}}.
\]
This verifies \eqref{eqn:hermitian-saddle-condition} or
\eqref{eqn:bipartite-saddle-condition} at scale $s$.
These numerical conditions do not involve $K_\ell$.

\smallskip\noindent
\emph{Step 2: Apply the trace bounds.}
In all three cases, $s\asymp d$, $\ell\leq C\sqrt s\log(N)$, and
$s^6(\log(N))^3\leq C N^{1-\delta_{\rm NB}}$. Consequently,
\begin{align}
 \ell s^5/N
 \leq Cs^{11/2}\log(N)/N
 \leq C N^{-\delta_{\rm NB}}s^{-1/2}(\log(N))^{-2}
 =o(1).
\end{align}
Since $s\geq1$, the polynomial terms in the general trace bounds satisfy
\begin{align}
 \ell^3\sqrt s+\ell^8s^4/N+\ell^8s^6/N
 \leq C\ell^7s\leq C\ell^7d.
\end{align}
Applying Lemma~\ref{lem:general-nonbacktracking-trace-bounds}\textup{(ii)}
to $\widehat{\mX}/R$ and using
$\widetilde{\mB}(\widehat{\mX})
=R\widetilde{\mB}(\widehat{\mX}/R)$
gives \eqref{eqn:uniform-bipartite-trace-consequence}, with the factor
$R^{2\ell}$ from rescaling and
$\widetilde K_\ell\leq C\exp(C_e\ell d^{-3})$ from Step~1.
Applying the same part to $\mZ/R$ gives
\eqref{eqn:uniform-rescaled-trace-consequence}:
homogeneity again supplies $R^{2\ell}$, and the remaining factor is
the bound on $\widetilde K_\ell$ for the rescaled matrix from Step~1.
Finally, applying part~\textup{(i)} directly to $\mH$ at $s=d$
with $K_\ell=1$ gives
\eqref{eqn:uniform-hermitian-trace-consequence}.
Under \eqref{eqn:rectangular-lower-edge-sparsity},
$\ell d^{-3}\leq C(\log(N))d^{-5/2}=O((\log(N))^{-3/2})$, proving
the asserted uniform convergence of the exponential factor in \textup{(i)}.
\end{proof}

Now, we are ready to prove Proposition~\ref{prop:nonbacktracking-radius-bounds}.
\begin{proof}[Proof of Proposition~\ref{prop:nonbacktracking-radius-bounds}]
For every square matrix $\mC$ and integer $\ell\geq1$,
\begin{align}
 \rho(\mC) = \rho(\mC^{\ell})^{1/\ell}
 \leq \|\mC^{\ell}\|^{1/\ell}
 =\|\mC^{\ell}(\mC^{*})^{\ell}\|^{1/(2\ell)}
 \leq\big(\Tr\mC^\ell(\mC^\ell)^*\big)^{1/(2\ell)}.
 \label{eqn:radius-by-trace}
\end{align}
Fix $\nu>0$. For cases \textup{(i)} and
\textup{(iii)}, choose
\[
 (\mC,R)=
 \begin{cases}
  (\widetilde{\mB}(\widehat{\mX}),
   (\rowvarmax\colvarmax)^{1/4}),&\text{in \textup{(i)}},\\
  (\mB(\mH),1),&\text{in \textup{(iii)}}.
 \end{cases}
\]
Use the fixed bounds $0<r\leq R\leq R_*$ from
\eqref{eqn:normalization-radius-bounds}.
Take $\ell=\lfloor c\sqrt d\log(N)\rfloor$, with the constant $c$
from Lemma~\ref{lem:uniform-trace-consequence} decreased so that
$C_ec\,C_{\rm NB}^{-5/2}\leq1$.
Since $2\ell\geq c\sqrt d\log(N)$ for sufficiently large $N$,
Lemma~\ref{lem:uniform-trace-consequence},
\eqref{eqn:radius-by-trace}, and Markov's inequality give the common
bound below, using $\exp(C_e\ell d^{-3})\geq1$ in case~\textup{(iii)}:
\begin{align}
 \P\left(\rho(\mC)>R+C_\nu d^{-1/2}\right)
 &\leq C\ell^7dN^2\exp(C_e\ell d^{-3})
 \left(1+C_\nu/(R\sqrt d)\right)^{-2\ell}
 \leq N^{4-c\log(1+C_\nu/R_*)},
 \label{eqn:explicit-radius-markov}
\end{align}
where the choice of $c$ gives $\exp(C_e\ell d^{-3})\leq N$,
and $d\leq N^{1/6}$ bounds the polynomial prefactor by $N^3$.
Bernoulli's inequality gives
\[
 \left(1+C_\nu/(R\sqrt d)\right)^{\sqrt d}
 \geq1+C_\nu/R\geq1+C_\nu/R_*.
\]
Choosing $C_\nu$ so that
$c\log(1+C_\nu/R_*)\geq\nu+4$ proves
\eqref{eqn:bipartite-radius-bound} and
\eqref{eqn:hermitian-radius-bound}.

For \textup{(ii)}, choose $\mC=\widetilde{\mB}(\mZ)$,
$R=(\rowvarmax\colvarmin)^{1/4}$, and
$\eta=\sqrt{\log d/d}$. The same upper bound $R_*$ from
\eqref{eqn:normalization-radius-bounds} applies, so choose $C_\nu$ as above.
Using the same power $\ell$, apply Markov's inequality at
\[
 T_\nu=e^{C_e\eta/2}(R+C_\nu d^{-1/2}).
\]
The exponential factor in
\eqref{eqn:uniform-rescaled-trace-consequence} cancels, giving
\[
 \P\{\rho(\mC)>T_\nu\}
 \leq C\ell^7dN^2
       \left(1+C_\nu/(R\sqrt d)\right)^{-2\ell}
 \leq N^{3-c\log(1+C_\nu/R_*)}
 \leq N^{-\nu}.
\]
Since $\eta\leq1$ for sufficiently large $N$, and $R\leq R_*$,
\[
 T_\nu-R
 =R(e^{C_e\eta/2}-1)+C_\nu e^{C_e\eta/2}d^{-1/2}
 \leq C_\nu'(\eta+d^{-1/2}).
\]
Increasing the constant in the conclusion proves
\eqref{eqn:rescaled-column-radius-bound}.
\end{proof}

%%%%%%%%%%%%%%%%%%%%%%%%%%%%%%%%%%%%%%%%%%%%%%
\section{Trace Expansion and Quotient-Graph Enumeration}
\label{sec:trace-estimates}
%%%%%%%%%%%%%%%%%%%%%%%%%%%%%%%%%%%%%%%%%%%%%%

Before we proceed with detailed trace estimates, we first illustrate the advantages of using nonbacktracking matrices over the original sparse matrices.
\begin{itemize}
    \item \emph{Simpler walk enumeration.}
    In $\Tr \mA^{2\ell}$, a walk can repeatedly enter tree branches, retrace its steps, and explore other branches. These configurations contribute significantly to $\E\Tr \mA^{2\ell}$, but counting them sharply is difficult. In contrast, a nonbacktracking walk on a tree must follow a simple path without such repetitions. This leads to a more straightforward enumeration of path pairs in $\E\Tr\bigl(\mB^\ell(\mB^\ell)^*\bigr)$.
    
    \item \emph{Control of rare subgraphs.}
    At bounded degree, rare subgraphs with several interacting cycles, such as tangles in random $d$-regular graphs \cite{Fri08}, can dominate high expected moments. Nonbacktracking paths have enough structure to organize restrictions that exclude these walks and still agree with the original powers on a high-probability event. Bordenave \cite{bordenave2020new} restricts nonbacktracking walks to traversed-edge supports with at most one cycle, obtaining a matrix that equals $\mB^\ell$ with high probability for suitable $\ell\asymp\log N$.
    After removing the trivial mode, high moments of the derived weighted path matrices and the Ihara--Bass formula yield the sharp spectral bound. We note that merely replacing $\mA$ by $\mB$ does not automatically eliminate the rare-event obstruction.
\end{itemize}

In this section, we extend the independent-entry arguments of \cite[Section~5.1]{benaych2020spectral} and \cite[Section~5]{dumitriu2024extreme} to dependent entries satisfying the graph moment conditions in Subsection~\ref{subsec:moment-assumptions}.
The additional difficulty is to control singleton contributions from edges traversed exactly once: these vanish for independent centered entries but may survive under dependence.
We retain these words and combine the core enumeration with graph moment bounds that suppress long chains of singleton edges.

Throughout the section, we fix an integer $\ell\geq3$ and let $d$ be the generic sparsity parameter, which are defined in \eqref{eqn:bipartite-normalization} for the bipartite case and \eqref{eqn:hermitian-sparsity-scale} for the Hermitian case, respectively.

\subsection{Step 1: Trace expansion and admissible words}
\label{subsec:trace-step1}
Recall the host edge set $\vec{\gE}(\gG)$ in \eqref{eqn:ordered-host-edges}. We define the set of nonbacktracking sequences for ordered host edges $e,f$ as follows:
\begin{align}
 \sP_{\gG}^{\ell+2}(e,f)
 \coloneqq
 \left\{
 \begin{aligned}
  \xi=(\xi_{-1},\xi_0,\ldots,\xi_\ell)\in[N]^{\ell+2}:\quad
  &(\xi_{-1},\xi_0)=e, \quad (\xi_{t-1},\xi_t)\in\vec{\gE}(\gG),
  \; 0\leq t\leq\ell\\ 
  &(\xi_{\ell-1},\xi_\ell)=f, \quad \xi_{t-1}\neq\xi_{t+1},
  \; 0\leq t\leq\ell-1
 \end{aligned}
 \right\}.
 \label{eqn:nonbacktracking-sequences}
\end{align}
By repeatedly applying \eqref{eqn:nonbacktracking-definition-bipartite} and \eqref{eqn:nonbacktracking-definition-hermitian}, we obtain the following two equalities:
\begin{subequations}
\begin{align}
 (\mB^\ell)_{ef}
 &=
 \sum_{a_1,\ldots,a_{\ell-1}\in\vec{\gE}(\gK_N)}
 \emB_{ea_1}\cdots\emB_{a_{\ell-1}f}
 =
 \sum_{\xi\in\sP^{\ell+2}_{\gK_N}(e,f)}
 \prod_{t=1}^{\ell}\emH_{\xi_{t-1}\xi_t},
 \label{eqn:nonbacktracking-power-expansion}
 \\
 (\widetilde{\mB}^\ell)_{ef}
 &=
 \sum_{a_1,\ldots,a_{\ell-1}\in\vec{\gE}(\gK_{n,m})}
 \widetilde{\emB}_{ea_1}\cdots
 \widetilde{\emB}_{a_{\ell-1}f}
=
 \sum_{\xi\in\sP^{\ell+2}_{\gK_{n,m}}(e,f)}
 \prod_{t=1}^{\ell}\widetilde{\emH}_{\xi_{t-1}\xi_t},
 \label{eqn:nonbacktracking-power-expansion-bipartite}
\end{align}
\end{subequations}
Squaring \eqref{eqn:nonbacktracking-power-expansion} and \eqref{eqn:nonbacktracking-power-expansion-bipartite}, the Hermitian constraint and \eqref{eqn:nonbacktracking-sequences} yield the following two equalities:
\begin{subequations}
\begin{align}
 \Tr\mB^\ell(\mB^\ell)^*
 =&\,\sum_{e,f\in\vec{\gE}(\gK_N)}
   |(\mB^\ell)_{ef}|^{2}
 =\sum_{e,f\in\vec{\gE}(\gK_N)}\,\,
   \sum_{\xi^{1},\xi^{2}\in\sP^{\ell+2}_{\gK_N}(e,f)}\,\,
   \prod_{t=1}^{\ell}\emH_{\xi^{1}_{t-1}\xi^{1}_t}
\prod_{t=1}^{\ell} \overline{\emH}_{\xi^{2}_{t-1}\xi^{2}_t}\,\,,
 \nonumber\\
 =&\, \sum_{e,f\in\vec{\gE}(\gK_N)}\,\,
   \sum_{\xi^{1},\xi^{2}\in\sP^{\ell+2}_{\gK_N}(e,f)}\,\,
   \prod_{t=1}^{\ell}\emH_{\xi^{1}_{t-1}\xi^{1}_t}\emH_{\xi^{2}_{t}\xi^{2}_{t-1}}\,\,,\label{eqn:two-lifted-paths-expansion}\\
   \Tr\widetilde{\mB}^\ell(\widetilde{\mB}^\ell)^* =&\, \sum_{e,f\in\vec{\gE}(\gK_{n,m})}\,\,
   \sum_{\xi^{1},\xi^{2}\in\sP^{\ell+2}_{\gK_{n,m}}(e,f)}\,\,
   \prod_{t=1}^{\ell}\widetilde{\emH}_{\xi^{1}_{t-1}\xi^{1}_t}\widetilde{\emH}_{\xi^{2}_{t}\xi^{2}_{t-1}}\,\,,\label{eqn:two-lifted-paths-expansion-bipartite}
\end{align}
\end{subequations}
where the two sequences $\xi^{1},\xi^{2}$ share the same predecessor and the endpoint pairs
\begin{align}
 e&=(\xi^{1}_{-1},\xi^{1}_0)=(\xi^{2}_{-1},\xi^{2}_0),
 &
 f&=(\xi^{1}_{\ell-1},\xi^{1}_\ell)=(\xi^{2}_{\ell-1},\xi^{2}_\ell).
 \label{eqn:shared-terminal-pair}
\end{align}
Note that the common predecessor is summed through $e$, subject to $\xi^{1}_{-1}=\xi^{2}_{-1}$. To further simplify \eqref{eqn:two-lifted-paths-expansion} and \eqref{eqn:two-lifted-paths-expansion-bipartite}, we introduce the definition of \emph{admissible word} below.

\begin{definition}[Admissible words]
\label{def:admissible-word}
For $\gG\in\{\gK_N,\gK_{n,m}\}$, we say that $\xi=(\xi_0,\ldots,\xi_{2\ell})\in[N]^{2\ell+1}$ is \emph{$\gG$-admissible} if it satisfies the following conditions:
\begin{align}
 &\xi_0=\xi_{2\ell},\quad \xi_{\ell-1}=\xi_{\ell+1},\quad
   (\xi_{t-1},\xi_t)\in\vec{\gE}(\gG) \,\, (1\leq t\leq2\ell), \quad \xi_{t-1}\neq\xi_{t+1}\,\, \bigl(t\in[2\ell-1]\setminus\{\ell\}\bigr).
\label{eqn:admissible-word}
\end{align}
The visits at indices $0$ and $\ell$, and their vertices $\xi_0$
and $\xi_\ell$, are called the first and second \emph{splices}.
The two splice vertices may coincide.
Furthermore, we define the following sets of $\gG$-admissible words:
\begin{subequations}
\begin{align}
 \sW^{\ell}
 &\coloneqq
 \bigl\{\xi\in[N]^{2\ell+1}:
          \xi\text{ satisfies \eqref{eqn:admissible-word} with }
          \gG=\gK_N\bigr\},
 \label{eqn:admissible-hermitian}
 \\
 \widetilde{\sW}_{j}^{\ell}
 &\coloneqq
 \bigl\{\xi\in[N]^{2\ell+1}:
          \xi\text{ satisfies \eqref{eqn:admissible-word} with }
          \gG=\gK_{n,m},\
          \xi_0\in\gV_j\bigr\},
 \qquad j\in\{1,2\}.
 \label{eqn:admissible-bipartite}
\end{align}
\end{subequations}
\end{definition}
Note that in \eqref{eqn:two-lifted-paths-expansion}, gluing the two paths $\xi^{1},\xi^{2}\in\sP^{\ell+2}_{\gK_N}(e,f)$ produces an admissible word $\xi=(\xi^1_0,\ldots,\xi^1_\ell, \xi^2_{\ell-1},\ldots,\xi^2_0)$, which belongs to $\sW^\ell$ by \eqref{eqn:admissible-hermitian}. Similarly, in \eqref{eqn:two-lifted-paths-expansion-bipartite}, gluing the two paths $\xi^{1},\xi^{2}\in\sP^{\ell+2}_{\gK_{n,m}}(e,f)$ produces an admissible word $\xi=(\xi^1_0,\ldots,\xi^1_\ell, \xi^2_{\ell-1},\ldots,\xi^2_0)$, which belongs to $\widetilde{\sW}_j^\ell$ by \eqref{eqn:admissible-bipartite}, where $j\in\{1,2\}$ is determined by $\xi_0\in\gV_j$. We use Figure~\ref{fig:trace-word-gluing} to illustrate the gluing process.
\begin{figure}[htbp]
\centering
\begin{tikzpicture}[
  x=1.65cm,y=0.68cm,
  vertex/.style={circle,draw=black!60,fill=white,line width=0.6pt,
                 inner sep=1.2pt,minimum size=7mm,font=\scriptsize},
  identified/.style={draw=black!60,fill=white,line width=0.6pt,
                     rounded corners=2pt,inner xsep=4pt,inner ysep=3pt,
                     font=\scriptsize},
  splice/.style={vertex,draw=ThemeColor,line width=1.1pt,
                 fill=ThemeColor!9},
  ghost/.style={vertex,draw=black!35,fill=black!3,text=black!70},
  firstpath/.style={UCSDBlue,line width=1.15pt,
                   -{Latex[length=2.1mm,width=1.5mm]}},
  secondpath/.style={ThemeColor,line width=1.15pt,
                    -{Latex[length=2.1mm,width=1.5mm]}},
  annotation/.style={font=\scriptsize,text=black!65,align=center},
  every path/.style={line cap=round,line join=round}
]
 \node[identified,draw=ThemeColor,line width=1.1pt,fill=ThemeColor!9]
   (s0) at (0,0) {$\xi_0=\xi_{2\ell}$};
 \node[vertex,draw=UCSDBlue,line width=1.1pt,fill=UCSDBlue!9] (x1) at (1.45,0.75) {$\xi_1$};
 \node[identified,draw=UCSDBlue,line width=1.1pt,fill=UCSDBlue!9] (xm) at (3.75,0)
   {$\xi_{\ell-1}=\xi_{\ell+1}$};
 \node[splice] (sl) at (5.2,0) {$\xi_\ell$};
 \node[splice,draw=ThemeColor,line width=1.1pt,fill=ThemeColor!9] (xlast) at (1.45,-0.75) {$\xi_{2\ell-1}$};
 \node[ghost] (u) at (-1.25,0.75) {$u$};

 \draw[black!40,densely dashed,-{Latex[length=1.9mm]}] (u)--(s0);
 \draw[firstpath] (s0)--(x1);
 \draw[firstpath] (x1)--node[fill=white,inner sep=1pt] {$\cdots$} (xm);
 \draw[black!20,line width=2.8pt] (xm)--(sl);
 \draw[firstpath] (xm) to[bend left=13] (sl);
 \draw[secondpath] (sl) to[bend left=13] (xm);
 \draw[secondpath] (xm)--node[fill=white,inner sep=1pt] {$\cdots$} (xlast);
 \draw[secondpath] (xlast)--(s0);

 \node[annotation] at (-1.25,-0.25) {deleted\\predecessor};
 \node[annotation] at (0,0.95) {first splice};
 \node[annotation] at (5.2,0.95) {second splice};
 \node[font=\scriptsize,text=UCSDBlue] at (2.6,1.55)
   {first path $\xi^1$};
 \node[font=\scriptsize,text=ThemeColor] at (2.6,-1.55)
   {$\xi^2$ read backward};
 \node[annotation] at (4.48,-0.9) {shared final support edge};
\end{tikzpicture}
\caption{Trace-word gluing: follow $\xi^1$ in blue and then $\xi^2$
backward in red.  The edge from $u$ is deleted,
$\xi_{2\ell}=\xi_0$ closes the word, and
$\xi_{\ell-1}=\xi_{\ell+1}$ marks the allowed reversal at the shared
final edge.}
\label{fig:trace-word-gluing}
\end{figure}
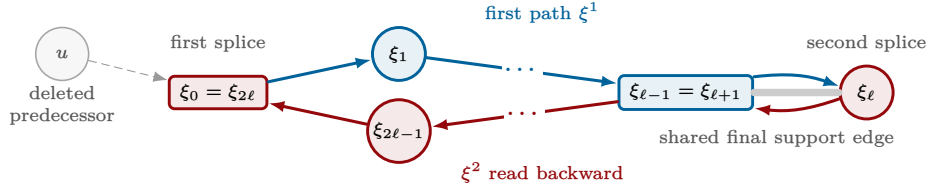

\noindent
Conversely, cutting such a word at its two splices recovers the two paths except for their common predecessor $u=\xi^1_{-1}=\xi^2_{-1}$. The allowed predecessors can be chosen from $[N]\setminus\{\xi_1,\xi_{2\ell-1}\}$ in the Hermitian case and $\gV_{3-j}\setminus\{\xi_1,\xi_{2\ell-1}\}$ in the rectangular case. Thus their multiplicities are
\begin{subequations}\label{eqn:deleted-predecessor-count}
\begin{align}
 c(\xi)
 &\coloneqq
 N-\abs{\{\xi_1,\xi_{2\ell-1}\}},
 &&\xi\in\sW^\ell,
 \label{eqn:deleted-predecessor-number}
 \\
 \widetilde c_j(\xi)
 &\coloneqq
 \abs{\gV_{3-j}}-\abs{\{\xi_1,\xi_{2\ell-1}\}},
 &&\xi\in\widetilde{\sW}_j^\ell,
 \quad j\in\{1,2\}.
 \label{eqn:deleted-predecessor-number-bipartite}
\end{align}
\end{subequations}
Reindexing the two path sums by their glued words and taking expectations now gives
\begin{subequations}\label{eqn:exact-trace-word-expansion}
\begin{align}
 \E\Tr\mB^\ell(\mB^\ell)^*
 &=\sum_{\xi\in\sW^{\ell}}c(\xi)
   \E\prod_{j=1}^{2\ell}\emH_{\xi_{j-1}\xi_j},
 \label{eqn:exact-hermitian-trace-word-expansion}
 \\
 \E\Tr\widetilde{\mB}^{\ell}(\widetilde{\mB}^{\ell})^*
 &=\sum_{j=1}^2\sum_{\xi\in\widetilde{\sW}_{j}^{\ell}}
   \widetilde c_j(\xi)
   \E\prod_{t=1}^{2\ell}
      \widetilde{\emH}_{\xi_{t-1}\xi_t}\,\,,
 \label{eqn:exact-bipartite-trace-word-expansion}
\end{align}
\end{subequations}
where $c(\xi)$ and $\widetilde c_j(\xi)$ are defined in \eqref{eqn:deleted-predecessor-number} and \eqref{eqn:deleted-predecessor-number-bipartite}, respectively.

\begin{remark}
\label{remark:trace-word-expansion}
Unlike the independent-entry arguments in \cite[Section~5.1]{benaych2020spectral} and \cite[Section~5]{dumitriu2024extreme} where words containing a singly traversed edge were discarded due to the centered contribution vanishing, we keep those words in \eqref{eqn:admissible-hermitian} and \eqref{eqn:admissible-bipartite}, since their centered contribution need not vanish under dependence.
\end{remark}

\subsection{Step 2: Decomposition of plural and singleton contributions}
\label{subsec:trace-step2}
To further simplify \eqref{eqn:exact-hermitian-trace-word-expansion} and \eqref{eqn:exact-bipartite-trace-word-expansion}, we group admissible words into relabeling classes and state moment bounds for their plural and singleton contributions. Terms containing edges traversed exactly once may survive under dependence.  We first define the graph and word terminology used to separate them.

\paragraph{Equivalence classes and canonical representatives} An undirected multigraph $\gG=(\gV(\gG),\gE(\gG),\phi)$ consists of two finite sets: the set of vertices $\gV(\gG)$ and the set of edges $\gE(\gG)$, and a map $\phi$ from $\gE(\gG)$ to the unordered sets of one or two elements of $\gV(\gG)$. For $e \in \gE(\gG)$, the set $\phi(e)$ is the set of vertices incident to $e$. The edge $e \in \gE(\gG)$ is a loop if $|\phi(e)|= 1$. The degree $\deg(v)$ of a vertex $v \in \gV(\gG)$ is the number of edges to which it is incident, whereby a loop incident to $v$ counts twice. Moreover, the multigraph $\gG$ is called simply a graph if $\phi$ is injective, i.e. there are no multiple edges (Note that in our convention a graph may have loops.) For a graph $\gG$, we may and shall identify $\gE(\gG)$ with a set of unordered pairs of $\gV(\gG)$, simply identifying $e$ and $\phi(e)$.

A \emph{path} of length $\ell \geq 1$ in $\gG$ is a \emph{word} $w = w_0 e_1 w_1 e_2 \dots e_\ell w_\ell$ where $w_0, \dots, w_\ell \in \gV(\gG)$, $e_1, \dots, e_\ell \in \gE(\gG)$, and $\phi(e_i) = \{w_{i-1}, w_i\}$ for each $i = 1,\dots, \ell$. The length of $w$ is denoted by $|w| = \ell$. We say that $w$ is closed if $w_0 = w_\ell$. For $e \in \gE(\gG)$, the \emph{traversal multiplicity} of $e$ in $w$ is
\begin{align}
 m_e(w) \coloneqq |\{ j \in [\ell] : e_j = e \}|. \label{eqn:path-crossing-count}
\end{align}
In particular, the total number of edge traversals is $\sum_{e \in \gE(\gG)} m_e(w) = |w| = \ell$.

We now define the graph $\gG_\xi$ associated with the word $\xi$, which is connected since the path defined by the word visits every vertex. Also, $\xi$ in \eqref{eqn:admissible-hermitian} and \eqref{eqn:admissible-bipartite} is a closed path in $\gG_\xi$.

\begin{definition}[Word-associated graph]
\label{def:word-associated-graph}
Given a word $\xi=(\xi_0,\ldots,\xi_{2\ell})\in[N]^{2\ell+1}$, let $\gG_\xi = (\gV(\gG_\xi), \gE(\gG_\xi))$ denote the \emph{word-associated graph} with
\begin{align}
 \gV(\gG_\xi)
 &\coloneqq\{\xi_j:0\leq j\leq2\ell\},
 &\gE(\gG_\xi)
 &\coloneqq
 \bigl\{\{\xi_{j-1},\xi_j\}:1\leq j\leq2\ell\bigr\}.
\end{align}
Under the bipartite setting, we have $\gV(\gG_\xi) = \gV_1(\gG_\xi) \cup \gV_2(\gG_\xi)$, where $\gV_{j}(\gG_\xi) \coloneqq \gV(\gG_\xi) \cap \gV_{j}$ for $j = 1, 2$.
\end{definition}

We then identify the equivalence relation between words.
\begin{definition}[Equivalent words]
\label{def:equivalent-words}
Two words $\xi$ and $\widetilde{\xi}$ of length $2\ell$ are \emph{equivalent}, denoted by $\xi \sim \widetilde{\xi}$, if and only if there exists a bijection $\varphi: \gV(\gG_{\xi}) \to \gV(\gG_{\widetilde{\xi}})$ such that $\varphi(\xi_{j}) =\widetilde{\xi}_{j}$ for all $j = 0, 1, \ldots, 2\ell$. Under the bipartite setting, the bijection must preserve $\gV_1$ and $\gV_2$, i.e., $\varphi(\gV_i(\gG_\xi)) = \gV_i(\gG_{\widetilde{\xi}})$ for $i = 1, 2$. The \emph{equivalence class} of a word $\xi$ consists of all words equivalent to $\xi$, and is denoted by $\WordClass{\xi}$.
\end{definition}

If $\xi$ and $\widetilde\xi$ are equivalent, the bijection identifies $\gG_\xi$ with $\gG_{\widetilde\xi}$.  In particular, the two word-associated graphs have the same numbers of vertices and edges, and hence the same genus
\begin{align}
 g(\gG_\xi)
 \coloneqq \abs{\gE(\gG_\xi)}-\abs{\gV(\gG_\xi)}+1.
 \label{eqn:genus-definition}
\end{align}
Note that the genus is invariant with respect to the equivalence class, but does not determine its traversal pattern. To represent that pattern uniquely, we introduce the canonical word.

\begin{definition}[Canonical word]
\label{def:canonical-word}
Let $w = w_0 e_1 w_1 e_2 \dots e_\ell w_\ell$ be a path in a multigraph $\gG$. A pair $(\gG, w)$ is \emph{canonical} if
\begin{enumerate}
  \item $\gV(\gG) = \{w_0, w_1, \dots, w_\ell\} = [s]$ where $s\coloneqq\abs{\gV(\gG)}$;
  \item the vertices in $\gV(\gG)$ are labeled in increasing order of their first appearance in $w$, i.e., if $w_{i} \notin \{w_{0}, \dots, w_{i-1}\}$, then $w_{i} > w_{j}$ for each $0\leq j < i$.
\end{enumerate} 
Under the Hermitian setting, vertices are labeled by $1,2,\ldots$ in order of their first appearance. Under the rectangular setting, this labeling is performed separately in the two vertex classes: the $t$-th new vertex of $\gV_j$ is relabeled by $(j,t)$\footnote{For example, $(7,3,7,9,3,7)$ is labeled as $(1,2,1,3,2,1)$. If $u_1,u_2\in\gV_1$ and $v_1,v_2\in\gV_2$, then $(u_1,v_1,u_2,v_1,u_1,v_2,u_1)$ is relabeled as $((1,1),(2,1),(1,2),(2,1),(1,1),(2,2),(1,1))$, which preserves the vertex class labeling.}. The canonical word is unique up to relabeling.
\end{definition}

Canonical relabeling preserves the traversal pattern, and two words are
equivalent if and only if their canonical representatives coincide.  Hence
each equivalence class contributes exactly one canonical word.  We collect
these canonical representatives in the sets defined below.

\begin{definition}[Canonical representatives of equivalence classes]
\label{def:canonical-equivalence-classes}
Let $\xi^\circ$ denote the canonical word obtained from $\xi$ by the first-appearance relabeling. We define the quotient sets
\begin{align}
 \sQ^{\ell}
 &\coloneqq
 \bigl\{\xi^\circ:\xi\in\sW^{\ell}\bigr\},
 &
 \widetilde{\sQ}_{j}^{\ell}
 &\coloneqq
 \bigl\{\xi^\circ:
          \xi\in\widetilde{\sW}_{j}^{\ell}\bigr\},
 \qquad j\in\{1,2\},
 \label{eqn:canonical-equivalence-class-sets}
\end{align}
which contain only canonical words, with exactly one canonical word representing each equivalence class.
\end{definition}

\paragraph{Plural and singleton contributions}
Whether a word contributes to the moment depends on the presence of edges that are traversed only once. For independent, centered entries, these terms disappear. When dependencies are present, they may persist. Therefore, we distinguish these contributions from those governed by the decay associated with repeated edges.

\begin{definition}[Plural and singleton words]
\label{def:plural-and-singleton-words}
Let $\xi$ be a word with associated graph $\gG_{\xi}$.  An edge $e$ is a \emph{singleton edge} of $\xi$ if $m_e(\xi)=1$, where $m_e(\xi)$ is defined in \eqref{eqn:path-crossing-count}.
\begin{itemize}
 \item The word $\xi$ is \emph{plural} if $m_e(\xi)\geq2$ for every
 $e\in\gE(\gG_{\xi})$.
 \item The word $\xi$ is \emph{singleton} if it contains at least one singleton edge.
\end{itemize}
\end{definition}

Since traversal multiplicities do not change under relabeling, we classify an equivalence class $\WordClass{\xi}$ as plural or singleton based on its canonical representative. This allows us to define the corresponding sets of canonical representatives as follows:
\begin{subequations}
\begin{align}
 \SquarePluralClasses{\ell}
 &\coloneqq
 \{\xi\in\sQ^\ell:\xi\textnormal{ is plural}\},
 &
 \SquareSingletonClasses{\ell}
 &\coloneqq
 \{\xi\in\sQ^\ell:\xi\textnormal{ is singleton}\},
\label{eqn:square-plural-singleton-class-sets}\\
 \RectangularPluralClasses{i}{\ell}
 &\coloneqq
 \{\xi\in\widetilde{\sQ}_i^\ell:
       \xi\textnormal{ is plural}\},
 &
 \RectangularSingletonClasses{i}{\ell}
 &\coloneqq
 \{\xi\in\widetilde{\sQ}_i^\ell:
       \xi\textnormal{ is singleton}\},
 \quad i\in\{1,2\}.
 \label{eqn:bipartite-plural-singleton-class-sets}
\end{align}
\end{subequations}

We now group the trace terms by equivalence class and sum over each family.

\begin{definition}[Class moments]
\label{def:roadmap-class-moments}
For a canonical representative $\xi$ in \eqref{eqn:canonical-equivalence-class-sets}, define the moment of its equivalence class $\WordClass{\xi}$ via \eqref{eqn:square-plural-singleton-class-sets} and \eqref{eqn:bipartite-plural-singleton-class-sets} as follows:
\begin{subequations}
\begin{align}
 \GraphMoment_\ell(\WordClass{\xi})
 &\coloneqq\sum_{\eta\in\WordClass{\xi}}c(\eta)
   \E\prod_{j=1}^{2\ell}\emH_{\eta_{j-1}\eta_j},
 \quad\xi\in\sQ^\ell, \label{eqn:square-signed-class-moment}
 \\
 \widetilde{\GraphMoment}_{\ell,i}(\WordClass{\xi})
 &\coloneqq\sum_{\eta\in\WordClass{\xi}}\widetilde c_i(\eta)
   \E\prod_{j=1}^{2\ell}
      \widetilde{\emH}_{\eta_{j-1}\eta_j},
 \quad\xi\in\widetilde{\sQ}_{i}^{\ell}, \quad i\in\{1,2\}.
 \label{eqn:rectangular-signed-class-moment}
\end{align}
\end{subequations}
\end{definition}
With \eqref{eqn:square-signed-class-moment} and \eqref{eqn:rectangular-signed-class-moment}, we can define the plural and singleton contributions by
\begin{subequations}
\begin{align}
 \SquarePluralContribution
 &\coloneqq
 \sum_{\xi\in\SquarePluralClasses{\ell}}
 \GraphMoment_\ell(\WordClass{\xi}),
 &
 \SquareSingletonContribution
 &\coloneqq
 \sum_{\xi\in\SquareSingletonClasses{\ell}}
 \GraphMoment_\ell(\WordClass{\xi}).
 \label{eqn:square-T0-T1}\\
 \RectangularPluralContribution
 &\coloneqq
 \sum_{i=1}^2
 \sum_{\xi\in\RectangularPluralClasses{i}{\ell}}
 \widetilde{\GraphMoment}_{\ell,i}(\WordClass{\xi}),
 &
 \RectangularSingletonContribution
 &\coloneqq
 \sum_{i=1}^2
 \sum_{\xi\in\RectangularSingletonClasses{i}{\ell}}
 \widetilde{\GraphMoment}_{\ell,i}(\WordClass{\xi}).
 \label{eqn:rectangular-T0-T1}
\end{align}
\end{subequations}

Since the equivalence classes partition the admissible words, regrouping \eqref{eqn:exact-hermitian-trace-word-expansion} and \eqref{eqn:exact-bipartite-trace-word-expansion} by \eqref{eqn:square-T0-T1} and \eqref{eqn:rectangular-T0-T1} gives the exact decomposition
\begin{subequations}
\begin{align}
 \E\Tr\mB^\ell(\mB^\ell)^*
 &=
 \sum_{\xi\in\sQ^\ell}
 \GraphMoment_\ell(\WordClass{\xi})
 =\SquarePluralContribution+\SquareSingletonContribution.
 \label{eqn:trace-plural-singleton-split-square}\\
 \E\Tr\widetilde{\mB}^{\ell}(\widetilde{\mB}^{\ell})^*
 &=
 \sum_{i=1}^2\sum_{\xi\in\widetilde{\sQ}_i^\ell}
 \widetilde{\GraphMoment}_{\ell,i}(\WordClass{\xi})
 =\RectangularPluralContribution
  +\RectangularSingletonContribution.
 \label{eqn:trace-plural-singleton-split-rectangular}
\end{align}
\end{subequations}

\subsection{Trace estimates under general moment bound}\label{subsec:moment-assumptions}
Building on the decompositions in \eqref{eqn:trace-plural-singleton-split-square} and \eqref{eqn:trace-plural-singleton-split-rectangular}, we can derive upper bounds for the trace of nonbacktracking operators as soon as suitable bounds on the moments of relabeling classes are available—even without assuming that the entries are independent. To formulate these assumptions, we first introduce support-chain statistics, which capture the singleton contribution, and will be revisited in the Subsection \ref{subsec:trace-step3}.

\begin{definition}[Suppressible vertices]
\label{def:suppressible-vertices}
Let $\gG$ be a multigraph and let $\gD\subseteq\gV(\gG)$ be a set of
distinguished vertices.  Define the set of \emph{suppressible} vertices by
\begin{align}
 \gS_{\gD}(\gG)
 \coloneqq
 \bigl\{v\in\gV(\gG)\setminus\gD:
              \deg_{\gG}(v)=2\bigr\}.
 \label{eqn:suppressible-vertex-set}
\end{align}
\end{definition}

We group support edges into chains with suppressible interiors, counting
each chain once up to reversal.

\begin{definition}[Suppressible chains and reversal classes]
\label{def:suppressible-chains}
For $\gG,\gD$ as above, define the following two sets.
\begin{enumerate}[label=\textup{(\roman*)},leftmargin=2.2em,itemsep=0.6em]
 \item \emph{Suppressible chains.}
 The set $\sP_{\gD}(\gG)$ consists of paths
 $w=w_0e_1w_1\cdots e_t w_t$, with $t\geq1$, whose edges are pairwise
 distinct, whose internal vertices are pairwise distinct and belong to
 $\gS_{\gD}(\gG)$, and whose endpoints lie outside
 $\gS_{\gD}(\gG)$.  The endpoints may coincide.

 \item \emph{Unoriented chain classes.}
 For $w\in\sP_{\gD}(\gG)$, let
 $w^{\rm rev}\coloneqq w_t e_t w_{t-1}\cdots e_1 w_0$ be its reversal.
 Two chains $w$ and $w'$ are equivalent if $w'=w$ or $w'=w^{\rm rev}$.  Denote the equivalence class by $\ReversalClass{w}\coloneqq\{w,w^{\rm rev}\}$. We identify the set of all such equivalence classes by
 \begin{align}
  \sP_{\gD}^{\circ}(\gG)
  \coloneqq
  \bigl\{\ReversalClass{w}:w\in\sP_{\gD}(\gG)\bigr\}.
  \label{eqn:suppressible-chain-classes}
 \end{align}
\end{enumerate}
\end{definition}

For $\gG=\gG_\xi$ and $\gD=\{\xi_0,\xi_\ell\}$, nonbacktracking
forces the same traversal multiplicity along every chain.  We therefore record the singleton edge count and a statistic of the lengths of the chains they form.

\begin{definition}[Singleton-chain data]
\label{def:singleton-chain-data}
For an admissible canonical word $\xi$, define
\begin{align}
 e_1(\xi)
 &\coloneqq\abs{\{e\in\gE(\gG_\xi):m_e(\xi)=1\}},
 &
 \varsigma(\xi)
 &\coloneqq
 \sum_{\substack{
  \ReversalClass{w}\in\sP_{\{\xi_0,\xi_\ell\}}^\circ(\gG_\xi)\\
  m_e(\xi)=1\ \text{for every edge }e\text{ of }w}}
 \floor{\frac{\abs{w}-1}{2}}.
 \label{eqn:singleton-certificate-count}
\end{align}
Here, $e_1(\xi)$ counts the singleton edges and $\varsigma(\xi)$ is the total number of pairs in a maximum pairing of consecutive internal vertices on each singleton chain, counted once up to reversal. Both quantities are invariant under relabeling.
\end{definition}

\paragraph{Graph moment assumptions.}
The following assumptions bound each signed class moment using
statistics of the original word.  The support weight
$d^{\abs{\gE(\gG_\xi)}-\ell}$ is supplemented by
$d^{-\varsigma(\xi)}$ to control long singleton chains.

\begin{assumption}[Hermitian moment bound]
\label{ass:hermitian-graph-condition}
Let $\mH\in\C^{N\times N}$ be a random Hermitian matrix with zero
diagonal, and let $d\geq1$ be the sparsity parameter.  Fix deterministic
constants $K_\ell,\chi_\ell,\Gamma_\ell\geq1$, uniformly over the
relabeling classes at the $\ell$-th power.  We say that $\mH$ satisfies
$\mathsf{HM}_\ell(d,K_\ell,\chi_\ell,\Gamma_\ell)$ if, for every
$\xi\in\sQ^\ell$,
\begin{align}
 \abs{\GraphMoment_\ell(\WordClass{\xi})}
 &\leq K_\ell N^2
 \left(\frac{\chi_\ell}{N}\right)^{\abs{g(\gG_\xi)}}
\times
 \begin{cases}
  d^{\abs{\gE(\gG_\xi)}-\ell},&\xi\in\SquarePluralClasses{\ell},\\[0.8em]
  \Gamma_\ell^{2\left(\abs{\gE(\gG_\xi)}-
    \abs{\gS_{\{\xi_0,\xi_\ell\}}(\gG_\xi)}\right)}
  d^{\abs{\gE(\gG_\xi)}-\ell-\varsigma(\xi)},&\xi\in\SquareSingletonClasses{\ell}.
 \end{cases}
 \tag*{\(\mathsf{(HM)}_\ell\)}
 \label{eqn:hermitian-graph-condition}
\end{align}
\end{assumption}

\begin{assumption}[Bipartite moment bound]
\label{ass:bipartite-graph-condition}
Let $\mX\in\C^{n\times m}$ be a random rectangular matrix, with
$N=n+m$, and let $\widetilde{\mH}$ be its Hermitian block linearization
in \eqref{eqn:block-linearization}.  Let $d\geq1$ be the sparsity
parameter, and fix deterministic constants
$\widetilde K_\ell,\widetilde\chi_\ell,\widetilde\Gamma_\ell\geq1$
uniformly over both first-splice colors and all relabeling classes at
the $\ell$-th power.  We say that $\widetilde{\mH}$ satisfies
$\mathsf{BM}_\ell(d,\widetilde K_\ell,\widetilde\chi_\ell,
\widetilde\Gamma_\ell)$ if, for every $i\in\{1,2\}$ and
$\xi\in\widetilde{\sQ}_i^\ell$,
\begin{align}
 \abs{\widetilde{\GraphMoment}_{\ell,i}(\WordClass{\xi})}
 &\leq\widetilde K_\ell nm
 \left(\frac{\widetilde\chi_\ell}{N}\right)^{\abs{g(\gG_\xi)}}
\times
 \begin{cases}
  d^{\abs{\gE(\gG_\xi)}-\ell},&\xi\in\RectangularPluralClasses{i}{\ell},\\[0.8em]
  \widetilde\Gamma_\ell^{2\left(\abs{\gE(\gG_\xi)}-
    \abs{\gS_{\{\xi_0,\xi_\ell\}}(\gG_\xi)}\right)}
    d^{\abs{\gE(\gG_\xi)}-\ell -\varsigma(\xi)},&\xi\in\RectangularSingletonClasses{i}{\ell}.
 \end{cases}
 \tag*{\(\mathsf{(BM)}_\ell\)}
 \label{eqn:bipartite-graph-condition}
\end{align}
\end{assumption}

We now state desired trace bounds for non-backtracking matrices.

\begin{lemma}[Nonbacktracking trace bounds]
\label{lem:general-nonbacktracking-trace-bounds}
There are numerical constants $c,C>0$, $d_0\geq1$ and some sufficiently large $C_{\rm enum}\geq 1$ such that the following holds for every integer $\ell\geq3$ and $d\geq d_0$.
\begin{enumerate}[label=\textup{(\roman*)},leftmargin=2.5em,itemsep=0.8em]
 \item \emph{Hermitian case.} If $\mH$ satisfies
 Assumption~\ref{ass:hermitian-graph-condition} and
 \begin{align}
  \frac{C_{\rm enum}\chi_\ell\ell^3\Gamma_\ell^6d^{9/2}}{N}
  &\leq\frac12,
  &\ell&\leq c\sqrt d\log
  \frac{N}{C_{\rm enum}\chi_\ell\ell^3\Gamma_\ell^6d^{9/2}},
  \label{eqn:hermitian-saddle-condition}
 \end{align}
 then $\mB=\mB(\mH)$ satisfies
 \begin{align}
  \E\Tr\mB^\ell(\mB^\ell)^*
  \leq CK_\ell N^2\left(
    \ell^3\sqrt d+\frac{\chi_\ell\ell^8d^4}{N}
    +\frac{\chi_\ell\Gamma_\ell^8\ell^8d^6}{N}
  \right).
  \label{eqn:general-hermitian-trace-bound}
 \end{align}

 \item \emph{Bipartite case.} Suppose that Assumption~\ref{ass:bipartite-graph-condition} holds, the aspect ratio obeys \eqref{eqn:proportional-block-regime}, and
 \begin{align}
  \frac{C_{\rm enum}\widetilde\chi_\ell\ell^3
        \widetilde\Gamma_\ell^6d^{9/2}}{N}
  &\leq\frac12,
  &\ell&\leq c\sqrt d\log
  \frac{N}{C_{\rm enum}\widetilde\chi_\ell\ell^3
             \widetilde\Gamma_\ell^6d^{9/2}},
  \label{eqn:bipartite-saddle-condition}
 \end{align}
 then the nonbacktracking matrix $\widetilde{\mB}$ satisfies
 \begin{align}
  \E\Tr\widetilde{\mB}^{\ell}(\widetilde{\mB}^{\ell})^*
  \leq C_{\ratio}\widetilde K_\ell nm\left(
    \ell^3\sqrt d+\frac{\widetilde\chi_\ell\ell^8d^4}{N}
    +\frac{\widetilde\chi_\ell\widetilde\Gamma_\ell^8\ell^8d^6}{N}
  \right).
  \label{eqn:general-bipartite-trace-bound}
 \end{align}
\end{enumerate}
\end{lemma}

\paragraph{Proof Outline for Lemma~\ref{lem:general-nonbacktracking-trace-bounds}.}
In Step~3, we replace each maximal nonsplice degree-two chain by a
single edge, preserving its endpoints and recording its length.
This preserves the genus and leaves a graph whose size is controlled
by the genus.  We count the resulting graphs and walks after relabeling
their vertices in order of first appearance.  Steps~4 and~5 use these
counts and the graph moment conditions to bound the plural and
singleton contributions.  Step~6 sums the positive-genus terms under
the displayed numerical conditions and includes the genus-zero plural
contribution, proving both trace bounds.

\begin{remark}[Genus summation and admissible powers]
\label{rem:trace-genus-power-range}
The expansion of $\Tr\mB^\ell(\mB^\ell)^*$ pairs nonbacktracking paths whose joined words allow reversals only at two marked splice visits; tree-supported words therefore traverse a single path once in each direction.
Retaining the splices and suppressing other degree-two vertices leaves a core with at most $2g+2$ vertices and $3g+1$ edges for support genus $g$; see \eqref{eqn:core-size}.
Enumeration reduces to the core, its traversal pattern, and its chain lengths.
The conditions \eqref{eqn:hermitian-saddle-condition} and \eqref{eqn:bipartite-saddle-condition} balance the growth of core walks against the suppression of positive-genus contributions in \eqref{eqn:roadmap-genus-sum}.
For our models, \eqref{eqn:rectangular-edge-sparsity-range} permits $3\leq\ell\leq c\sqrt d\log(N)$ by Lemma~\ref{lem:uniform-trace-consequence}.
Taking $\ell\asymp\sqrt d\log(N)$ makes the polynomial prefactor in $N$ contribute $O((\log(N))/\ell)=O(d^{-1/2})$ to the radius bound obtained through \eqref{eqn:radius-by-trace} and Markov's inequality.
Larger powers require sharper genus estimates in the present argument, together with moment bounds valid at those powers.
\end{remark}

\subsection{Step 3: Reduction to a compressed multigraph.} \label{subsec:trace-step3}

The support graph $\gG_{\xi}$ may contain degree-two chains whose
lengths are $O(\ell)$ and are not controlled by the genus.  As in
\cite{benaych2020spectral}, we replace each chain from
$\sP_{\{\xi_0,\xi_\ell\}}^{\circ}(\gG_\xi)$ in
\eqref{eqn:suppressible-chain-classes} by a single weighted edge
with the same endpoints.  This preserves the genus and, as proved
below, leaves $O(g(\gG_\xi)+1)$ vertices and edges.  The walk on
that graph can then be counted separately from the chain lengths.

\begin{definition}[Compressed multigraph]
\label{def:compressed-multigraph}
Let $\xi$ be a canonical admissible word of length $2\ell$ from quotient sets \eqref{eqn:canonical-equivalence-class-sets}. The compressed multigraph $\widehat{\gG}_{\xi}$ obtained from $\gG_\xi$ has
\begin{align}
 \gV(\widehat{\gG}_{\xi})
 &\coloneqq
 \gV(\gG_\xi)\setminus
 \gS_{\{\xi_0,\xi_\ell\}}(\gG_\xi),
 &
 \gE(\widehat{\gG}_{\xi})
 &\coloneqq
 \sP_{\{\xi_0,\xi_\ell\}}^{\circ}(\gG_\xi).
 \label{eqn:compressed-multigraph}
\end{align}
For an edge $\ReversalClass{w}\in\gE(\widehat{\gG}_{\xi})$, where
$w=w_0e_1w_1\cdots e_tw_t$, its endpoints and weight are
\begin{align}
 \widehat\phi(\ReversalClass{w})
 &\coloneqq\{w_0,w_t\},
 &
 \widehat k_{\ReversalClass{w}}
 &\coloneqq t=\abs{w}.
 \label{eqn:compressed-edge-incidence-weight}
\end{align}
Thus, each chain class becomes a weighted edge; parallel edges and loops may arise.
\end{definition}

\noindent
We illustrate the construction with an example shown in Figure~\ref{fig:quotient-graph-process}. Consider the canonical admissible word
\begin{align}
 \xi=(1,2,3,4,5,6,3,7,8,9,10,11,10,9,12,13,3,6,5,4,3,2,1),
 \qquad \ell=11.
\end{align}
Its splice vertices are $1$ and $11$, and its branch vertices are $3$ and $9$.  All other vertices are suppressible. The chains $3$--$7$--$8$--$9$ and $3$--$13$--$12$--$9$ become parallel edges of weight three, while the cycle $3$--$4$--$5$--$6$--$3$ becomes a loop of weight four.  The two chains joining the splices to the branch vertices have weight two.

\begin{figure}[!ht]
\centering
\begin{tikzpicture}[
  x=1.03cm,y=0.72cm,
  vertex/.style={circle,draw=black,fill=white,inner sep=1.3pt,
                 minimum size=5mm,font=\scriptsize},
  suppressible/.style={vertex,fill=black!8,draw=black!55},
  retained/.style={vertex,draw=UCSDBlue,very thick,fill=UCSDBlue!5},
  splice/.style={vertex,draw=ThemeColor,very thick,fill=ThemeColor!8},
  coreedge/.style={very thick,UCSDBlue},
  edgelabel/.style={font=\scriptsize,fill=white,inner sep=1pt},
  every node/.style={font=\scriptsize}
]
 \node at (3.15,2.65) {Support $\gG_\xi$};
 \node[splice]       (l1)  at (0,0) {$1$};
 \node[suppressible] (l2)  at (0.75,0) {$2$};
 \node[retained]     (l3)  at (1.5,0) {$3$};
 \node[suppressible] (l4)  at (0.9,1.05) {$4$};
 \node[suppressible] (l5)  at (1.5,1.9) {$5$};
 \node[suppressible] (l6)  at (2.1,1.05) {$6$};
 \node[suppressible] (l7)  at (2.65,0.75) {$7$};
 \node[suppressible] (l8)  at (3.65,0.75) {$8$};
 \node[retained]     (l9)  at (4.8,0) {$9$};
 \node[suppressible] (l10) at (5.55,0) {$10$};
 \node[splice]       (l11) at (6.3,0) {$11$};
 \node[suppressible] (l12) at (3.65,-0.85) {$12$};
 \node[suppressible] (l13) at (2.65,-0.85) {$13$};
 \draw[thick] (l1)--(l2)--(l3);
 \draw[thick] (l3)--(l4)--(l5)--(l6)--(l3);
 \draw[thick] (l3)--(l7)--(l8)--(l9);
 \draw[thick] (l9)--(l10)--(l11);
 \draw[thick] (l9)--(l12)--(l13)--(l3);

 \draw[very thick,-{Latex[length=2.6mm]},black!55]
   (6.75,0)--(7.45,0);

 \node at (10.1,2.65) {compressed multigraph $\widehat{\gG}_\xi$};
 \node[splice]   (r1)  at (7.9,0) {$1$};
 \node[retained] (r3)  at (9.1,0) {$3$};
 \node[retained] (r9)  at (11.15,0) {$9$};
 \node[splice]   (r11) at (12.35,0) {$11$};
 \draw[coreedge] (r1)--node[edgelabel,above=3pt] {$f_1\,(2)$} (r3);
 \draw[coreedge] (r3) edge[loop above,min distance=8mm]
   node[edgelabel,above=3pt] {$f_2\,(4)$} (r3);
 \draw[coreedge] (r3) to[bend left=35]
   node[edgelabel,above=3pt] {$f_3\,(3)$} (r9);
 \draw[coreedge] (r9)--node[edgelabel,above=3pt] {$f_4\,(2)$} (r11);
 \draw[coreedge] (r3) to[bend right=35]
   node[edgelabel,below=3pt] {$f_5\,(3)$} (r9);
\end{tikzpicture}
\caption{Suppression of degree-two chains.  Red splice vertices and blue branch vertices are retained with their original labels; gray vertices are compressed.  Each edge $f_i$ on the right is labeled by its chain length $k_{f_{i}}$ in parentheses. Here, $\gG_\xi$ has $13$ vertices and $14$ edges, while $\widehat{\gG}_{\xi}$ has $4$ vertices and $5$ edges, both with genus $g(\gG_\xi) = g(\widehat{\gG}_{\xi}) = 2$.}
\label{fig:quotient-graph-process}
\end{figure}
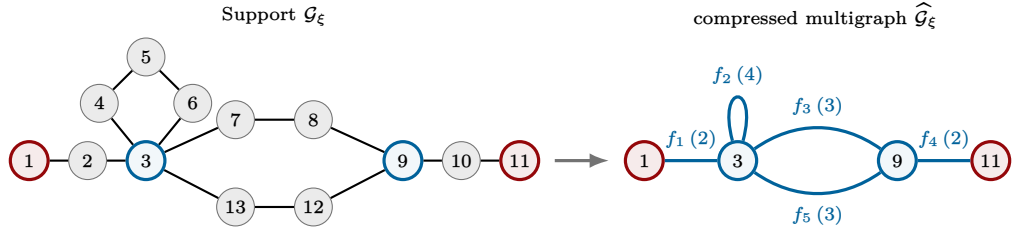

At each suppressible vertex, the nonbacktracking condition requires the walk to continue along a new edge, so $\xi$ breaks into consecutive paths in $\sP_{\{\xi_0,\xi_\ell\}}^{\circ}(\gG_\xi)$. Replacing each such path with a single edge in $\widehat{\gG}_{\xi}$ yields a rooted closed walk $\widehat\xi$ on $\widehat{\gG}_{\xi}$, and we associate to $\xi$ the triple
\begin{align}
 \bigl(\widehat{\gG}_{\xi},\widehat\xi,
       (\widehat k_f)_{f\in\gE(\widehat{\gG}_{\xi})}\bigr).
 \label{eqn:compressed-multigraph-triple}
\end{align}
For instance, in Figure~\ref{fig:quotient-graph-process}, this gives
\begin{align}
 \widehat\xi
 =1\,f_1\,3\,f_2\,3\,f_3\,9\,f_4\,11\,f_4\,9\,f_5\,3\,f_2\,3\,f_1\,1,
\end{align}
where $f_2$ is traversed twice in opposite directions. The word is singleton because each support edge in the chains represented by $f_3$ and $f_5$ is traversed once. The induced walk records both the edge sequence (with orientation) and the vertices. In the rectangular setting, vertex colors and chain-length parities encode the original bipartition, though the compressed multigraph itself may not be bipartite.

The compressed multigraph need not yet be canonical in the sense of Definition~\ref{def:canonical-word}. We therefore relabel it together with its induced walk.

\begin{definition}[Canonical core]
\label{def:canonical-core}
Consider the triple in \eqref{eqn:compressed-multigraph-triple} obtained from a canonical admissible word $\xi$. Define $\tau$ by numbering vertices in order of their first appearance in $\widehat\xi$, so that the $t$-th vertex receives label $t$ in the square setting and label $(j,t)$ in class $j\in \{1, 2\}$ in the bipartite setting.

Apply $\tau$ to $\widehat{\gG}_\xi$ and every vertex visit of $\widehat\xi$, and label the edges in order of their first traversal. The resulting multigraph $\gU$ is the \emph{canonical core} (or simply the \emph{core}) of $\xi$. Let $\zeta$ denote the induced dart\footnote{A \emph{dart} is an edge together with a chosen orientation.  Every edge has two darts; for a loop, these remain distinct even though their initial and terminal vertices coincide.  A \emph{dart walk} records the vertices visited and the dart traversed at each step.  It therefore distinguishes parallel edges and the two directions of traversal of a loop.  These directions specify how to traverse the corresponding compressed chain when reconstructing the original word.} walk on $\gU$, rooted at the visit corresponding to the first splice $\xi_0$.  It is closed, so $\zeta_{\abs{\zeta}}=\zeta_0$. For each $f\in\gE(\gU)$, let $\ReversalClass{w}\in\gE(\widehat{\gG}_\xi)$ be the unique edge relabeled as $f$, and define
\begin{align}
 k_f\coloneqq\widehat k_{\ReversalClass{w}}=\abs{w}.
 \label{eqn:canonical-core-chain-length}
\end{align}
Thus $k_f$ is the number of original support edges represented by $f$, independently of the orientation of the chain.  Let $r$ denote the index of the visit of $\zeta$ corresponding to the second splice $\xi_\ell$. Together, these data define the canonical core map
\begin{align}
 \xi\longmapsto\CoreCode.
 \label{eqn:canonical-core-map}
\end{align}
\end{definition}

The following lemma records the basic properties of the map \eqref{eqn:canonical-core-map}. It is the version of \cite[Lemma~5.7]{benaych2020spectral} applicable to the larger family of admissible words considered here.

\begin{lemma}[Basic properties of the canonical core map]
\label{lem:normalized-core-properties}
Let $\xi\in\sQ^\ell$ in the square setting or $\xi\in\widetilde{\sQ}_i^\ell$ in the rectangular setting.  For $f\in\gE(\gU)$, write $m_f=m_f(\zeta)$ for its traversal multiplicity from \eqref{eqn:path-crossing-count}, counting both orientations. The map \eqref{eqn:canonical-core-map} has the following properties.
\begin{enumerate}[label=\textup{(\roman*)},leftmargin=2.2em,itemsep=0.7em]
 \item \emph{Core structure and genus.}
 The walk $\zeta$ is a rooted closed dart walk that visits every
 vertex and every edge of $\gU$. It is canonical in the sense of Definition~\ref{def:canonical-word}. In particular,
 $\zeta_0=\zeta_{\abs{\zeta}}=1$ in the square setting and
 $\zeta_0=\zeta_{\abs{\zeta}}=(i,1)$ in the rectangular setting.

\item Every vertex in $\gV(\gU)\setminus\{\zeta_0,\zeta_{r}\}$ has degree at least three. Each of the two splice vertices $\zeta_0$ and $\zeta_{r}$ has degree at least one; they are allowed to coincide.

\item \emph{Splice recovery and reconstruction.}
If $f_t$ is the underlying edge of the $t$-th dart traversed by $\zeta$, then the marked visit is the unique index $r$ satisfying
\begin{align}
 \sum_{t=1}^{r}k_{f_t}=\ell.
 \label{eqn:core-mark-recovery}
\end{align}
The code $\CoreCode$ uniquely determines $\xi$. The map
\eqref{eqn:canonical-core-map} is injective on \eqref{eqn:canonical-equivalence-class-sets}.

 \item \emph{Chain lengths and multiplicities.}
 For each $f\in\gE(\gU)$, the chain length $k_f$ is a positive integer,
 and every support edge in the chain represented by $f$ has traversal
 multiplicity $m_f\geq1$.
 In the rectangular setting, the endpoints of $f$ have opposite colors
 if and only if $k_f$ is odd.

 \item Expanding every weighted core edge into its chain gives
 \begin{align}
 \begin{aligned}
  \abs{\gE(\gG_{\xi})}
  &=\sum_{f\in\gE(\gU)}k_f,
  &
  \abs{\gV(\gG_{\xi})}
  &=\abs{\gV(\gU)}
    +\sum_{f\in\gE(\gU)}(k_f-1),
  \\
  2\ell
  &=\sum_{f\in\gE(\gU)}k_fm_f,
  &
  \abs{\zeta}
  &=\sum_{f\in\gE(\gU)}m_f
 \end{aligned}
 \label{eqn:core-identities}
 \end{align}
 In particular, suppressing the internal vertices gives
 \begin{align}
  \abs{\gE(\gU)}
  &=\abs{\gE(\gG_\xi)}
    -\abs{\gS_{\{\xi_0,\xi_\ell\}}(\gG_\xi)}.
  \label{eqn:core-support-edge-count}
 \end{align}

 \item Suppression and canonical relabeling preserve the genus defined in \eqref{eqn:genus-definition}:
 \begin{align}
  g(\gG_{\xi})
  =g(\widehat{\gG}_{\xi})
  =g(\gU) \geq 0.
  \label{eqn:core-genus-preservation}
 \end{align}
 The genus is nonnegative since all three graphs are connected. 
\end{enumerate}
\end{lemma}

The following estimate for the canonical core graph will be used throughout the enumeration.
\begin{lemma}[Canonical core map]
\label{lem:reduced-graph-code}
The following holds for $\CoreCode$ arising from \eqref{eqn:canonical-core-map}.
\begin{enumerate}[label=\textup{(\roman*)},leftmargin=2.2em,itemsep=0.7em]
 \item \emph{Core size.} Every core satisfies
  \begin{align}
   \abs{\gE(\gU)}
   &\leq3g(\gU)+1,
   &\abs{\gV(\gU)}
   &\leq2g(\gU)+2.
   \label{eqn:core-size}
  \end{align}
 \item \emph{Core-walk data.} For fixed $\abs{\gE(\gU)}$,
 $\abs{\gV(\gU)}$, and $\abs{\zeta}$, the number of possible triples
 $(\gU,\zeta,r)$ is at most
  \begin{align}
   C\ell\abs{\gV(\gU)}
   \bigl(C\abs{\gE(\gU)}\bigr)^{\abs{\zeta}}
   \bigl(C\abs{\gV(\gU)}\bigr)^{\abs{\gE(\gU)}},
   \label{eqn:core-walk-count}
  \end{align}
  where $C\geq1$ is a universal constant.
 \item \emph{Chain lengths.} For each fixed triple $(\gU,\zeta,r)$,
 the number of positive integer families $(k_f)_{f\in\gE(\gU)}$
 completing it to a canonical core code $\CoreCode$ is at most
  \begin{align}
   \binom{2\ell}{\abs{\gE(\gU)}}
   \leq
   \left(\frac{6\ell}{\abs{\gE(\gU)}}\right)^{\abs{\gE(\gU)}}.
   \label{eqn:chain-length-count}
  \end{align}
\end{enumerate}
\end{lemma}
According to Lemmas~\ref{lem:normalized-core-properties} and \ref{lem:reduced-graph-code}, any class of genus $g(\gG_{\xi})$ can be described by a canonical core consisting of $O(g(\gG_{\xi})+1)$ edges and vertices. This representation also involves a rooted dart walk with a specified second-splice visit, as well as one positive chain length assigned to each core edge.

\subsection{Step 4: Plural contributions}
\label{subsec:trace-step4}
Assume the plural line of \ref{eqn:hermitian-graph-condition} or
\ref{eqn:bipartite-graph-condition} in the respective setting, with
$d\geq1$.  To bound the plural sums in \eqref{eqn:square-T0-T1} and
\eqref{eqn:rectangular-T0-T1}, we count the canonical classes at each
genus, convert the support weights into core weights, then sum the
class moments.

\paragraph{Classes to be counted.}
Each term defining $\SquarePluralContribution$ represents one equivalence
class $\WordClass{\xi}$.  Assigning distinct labels in $[N]$ to the
vertices of $\gG_\xi$ gives
$\abs{\WordClass{\xi}}=
\prod_{t=0}^{\abs{\gV(\gG_\xi)}-1}(N-t)$.
These labeled words and their coefficients $c(\eta)$ are already summed
in $\GraphMoment_\ell(\WordClass{\xi})$, allowing their moments to depend
on the labels.  The bipartite class moment similarly includes all
injective labelings within each vertex class.  We therefore count the
canonical representatives in the outer sums, with one representative
for each equivalence class.

\paragraph{Counting classes at fixed genus.}
The genus is nonnegative due to connectivity of the word graph. For a plural word, there are at most $\ell$ support edges since every support edge is traversed at least twice. Then
\begin{align}
 0&\leq g(\gG_\xi)\leq\abs{\gE(\gG_\xi)}\leq\ell,
 &
 2\abs{\gE(\gU)}&\leq\abs{\zeta}
 =\sum_fm_f\leq\sum_fk_fm_f=2\ell,
 \label{eqn:roadmap-plural-genus-range}
\end{align}
where the last inequality holds due to $m_f\geq2$ and $k_f\geq1$. Since $g(\gU)=g(\gG_\xi)$, the core edge and vertex counts of every class of genus $g$ belong to
\begin{align}
 \sA_g\coloneqq
 \{(a,v)\in\N^2:1\leq a\leq3g+1,\quad
                  v=a-g+1,\quad1\leq v\leq2g+2\},
 \label{eqn:core-size-index-set}
\end{align}
where the inequalities follow from \eqref{eqn:core-size}.

Fix positive integers $a,v,t$, denoting the core edge count, vertex
count, and walk length, respectively. Let $\sC_\ell(a,v,t)$ be the set of complete codes $\CoreCode$ arising from canonical admissible words
$\xi\in\sQ^\ell$ with
$\abs{\gE(\gU)}=a$, $\abs{\gV(\gU)}=v$, and $\abs{\zeta}=t$.
By injectivity of \eqref{eqn:canonical-core-map},
$\abs{\sC_\ell(a,v,t)}$ equals the number of canonical admissible
classes with these three sizes.  In particular, it bounds the number
of plural classes with these sizes.

To count $\sC_\ell(a,v,t)$, first choose the triple $(\gU,\zeta,r)$.
Equation~\eqref{eqn:core-walk-count} gives at most
$C\ell v(Ca)^t(Cv)^a$ choices.  For each fixed triple, the
multiplicities $m_f$ are fixed, and every chain-length completion
satisfies
$\sum_fk_f\leq\sum_fk_fm_f=2\ell$, since $m_f\geq1$.
There are $\binom{2\ell}{a}$ positive integer families with
$\sum_fk_f\leq2\ell$, by the stars-and-bars count.  Dropping the
remaining admissibility constraints therefore gives the bound
\eqref{eqn:chain-length-count} for each triple.  Summing the number
of completions over all triples yields
\begin{align}
 \abs{\sC_\ell(a,v,t)}
 &\leq C\ell v(Ca)^t(Cv)^a\binom{2\ell}{a}
 \leq C\ell v(Ca)^t(Cv)^a
       \left(\frac{6\ell}{a}\right)^a.
 \label{eqn:trace-overview-code-count}
\end{align}
The same argument applies to the complete codes arising from $\widetilde{\sQ}_i^\ell$ at each fixed first-splice color $i$.  This count includes both plural and singleton classes; restricting to either family only decreases it.

\paragraph{From class counts to core weights.}
The plural moment bounds assign the support weight
$d^{\abs{\gE(\gG_\xi)}-\ell}$ to each class.
Since $m_f\geq2$, $k_f\geq1$, and $d\geq1$,
\eqref{eqn:core-identities} gives
\begin{align}
 d^{\abs{\gE(\gG_\xi)}-\ell}
 =d^{-\frac12\sum_f k_f(m_f-2)}
 \leq d^{-\frac12\sum_f(m_f-2)}
 =d^{\abs{\gE(\gU)}-\abs{\zeta}/2}.
 \label{eqn:roadmap-plural-chain-decay}
\end{align}
For fixed $a,v,t$, this upper bound is $d^{a-t/2}$, independently
of the chain lengths.  Hence multiplying
\eqref{eqn:trace-overview-code-count} by $d^{a-t/2}$ bounds the sum
of the support weights over plural classes with these sizes.

We now sum over the core sizes and walk lengths allowed at genus $g$.
Choose $C_0\geq1$ large enough for the bounds in
Lemma~\ref{lem:reduced-graph-code}. For $s\geq1$, define the weighted
count at genus $g$ by
\begin{align}
 \Theta_\ell^{\rm p}(g;s)
 \coloneqq
 \sum_{(a,v)\in\sA_g}\sum_{t=1}^{2\ell}
 \ell v(C_0a)^t(C_0v)^a
 \left(\frac{6\ell}{a}\right)^a s^{a-t/2},
 \label{eqn:genus-weight-definition}
\end{align}
where $t$ enumerates the walk length, with its actual range $2a\leq t\leq2\ell$ given by \eqref{eqn:roadmap-plural-genus-range}; including the nonnegative terms with $t<2a$ only enlarges the count. By \eqref{eqn:trace-overview-code-count} and \eqref{eqn:roadmap-plural-chain-decay},
\begin{align}
 \sum_{\substack{\xi\in\SquarePluralClasses{\ell}\\
                  g(\gG_\xi)=g}}
 d^{\abs{\gE(\gG_\xi)}-\ell}
 \leq C\Theta_\ell^{\rm p}(g;d),
 \qquad 0\leq g\leq\ell.
 \label{eqn:plural-weighted-class-count}
\end{align}
At $d=1$, the left-hand side counts the plural equivalence classes of genus $g$. Summing over $0\leq g\leq\ell$ therefore counts the terms in the plural sum in \eqref{eqn:square-T0-T1}. The same bound holds for $\RectangularPluralClasses{i}{\ell}$ at each fixed first-splice color $i$.

\paragraph{Summing the class moments.}
The range \eqref{eqn:roadmap-plural-genus-range} limits the genus sum to
$0\leq g\leq\ell$.  Applying the triangle inequality, the plural line of
\ref{eqn:hermitian-graph-condition}, and
\eqref{eqn:plural-weighted-class-count} with $s=d$ gives
\begin{subequations}
\begin{align}
 |\SquarePluralContribution|
 &\leq\sum_{\xi\in\SquarePluralClasses{\ell}}
       \abs{\GraphMoment_\ell(\WordClass{\xi})}
\leq K_\ell N^2
   \sum_{g=0}^{\ell}\left(\frac{\chi_\ell}N\right)^g
   \sum_{\substack{\xi\in\SquarePluralClasses{\ell}\\
                    g(\gG_\xi)=g}}
       d^{\abs{\gE(\gG_\xi)}-\ell}
 \nonumber\\
 &\leq CK_\ell N^2
   \sum_{g=0}^{\ell}\left(\frac{\chi_\ell}N\right)^g
       \Theta_\ell^{\rm p}(g;d).
 \label{eqn:roadmap-hermitian-plural-core-sum}
\end{align}
Applying the same argument to each first-splice color and summing over
$i\in\{1,2\}$ gives
\begin{align}
 |\RectangularPluralContribution|
 &\leq C\widetilde K_\ell nm
   \sum_{g=0}^{\ell}\left(\frac{\widetilde\chi_\ell}N\right)^g
       \Theta_\ell^{\rm p}(g;d).
 \label{eqn:roadmap-bipartite-plural-core-sum}
\end{align}
\end{subequations}
The factor two for the possible first-splice colors is absorbed in $C$.
\paragraph{Bounding the weighted counts.}
To estimate $\Theta_\ell^{\rm p}(g;s)$, collect the walk-length sum in
\begin{align}
 \Psi_\ell(x)\coloneqq\sum_{t=1}^{2\ell}x^t
 =\begin{cases}
  x(1-x^{2\ell})/(1-x),&x\neq1,\\
  2\ell,&x=1,
 \end{cases}
 \qquad x\geq0.
 \label{eqn:plural-finite-geometric-factor}
\end{align}
The following lemma then bounds the remaining sum over core sizes.

\begin{lemma}[Plural weight as a function of genus]
\label{lem:plural-genus-weight}
For $s\geq1$ and $\ell\geq3$,
\begin{align}
 \Theta_\ell^{\rm p}(0;s)
 &\leq C\ell^3s\,\Psi_\ell(C/\sqrt s),
 \label{eqn:roadmap-plural-genus-zero}
 \\
 \Theta_\ell^{\rm p}(g;s)
 &\leq C\ell^3s(C\ell^3s^3)^g
      \Psi_\ell\!\left(\frac{C(g+1)}{\sqrt s}\right),
 \qquad g\geq1.
 \label{eqn:roadmap-plural-fixed-genus-bound}
\end{align}
\end{lemma}

Finally, applying Lemma~\ref{lem:plural-genus-weight} to
\eqref{eqn:roadmap-hermitian-plural-core-sum} and
\eqref{eqn:roadmap-bipartite-plural-core-sum}, and taking $C_{\rm enum}$
large enough to absorb the constants raised to $g$, yields
\begin{subequations}
\begin{align}
 |\SquarePluralContribution|
 &\leq CK_\ell N^2\ell^3d
 \left[
  \Psi_\ell(C/\sqrt d)
  +\sum_{g=1}^{\ell}
       \left(\frac{C_{\rm enum}\chi_\ell\ell^3d^3}{N}\right)^g
       \Psi_\ell\!\left(\frac{C(g+1)}{\sqrt d}\right)
 \right],
 \label{eqn:roadmap-hermitian-plural-genus-sum}
 \\
 |\RectangularPluralContribution|
 &\leq C\widetilde K_\ell nm\ell^3d
 \left[
  \Psi_\ell(C/\sqrt{d})
  +\sum_{g=1}^{\ell}
       \left(\frac{C_{\rm enum}\widetilde\chi_\ell\ell^3d^3}{N}\right)^g
       \Psi_\ell\!\left(\frac{C(g+1)}{\sqrt{d}}\right)
 \right].
 \label{eqn:roadmap-bipartite-plural-genus-sum}
\end{align}
\end{subequations}

\subsection{Step 5: Singleton contributions}
\label{subsec:trace-step5}
Consider the singleton line from \ref{eqn:hermitian-graph-condition} or \ref{eqn:bipartite-graph-condition}, as appropriate, with $\ell\geq3$ and $d\geq1$. To estimate the singleton sums in \eqref{eqn:square-T0-T1} and \eqref{eqn:rectangular-T0-T1}, we proceed by counting the canonical classes at each genus, translating support weights to core weights, and then summing the corresponding class moments. The chain penalty factors present in the singleton moment bounds ensure that the influence of long singleton chains is properly controlled.

\paragraph{Classes to be counted.}
Each term defining $\SquareSingletonContribution$ represents one
equivalence class $\WordClass{\xi}$ with
$\xi\in\SquareSingletonClasses{\ell}$.  The class moment
$\GraphMoment_\ell(\WordClass{\xi})$ already sums all injective
labelings and their coefficients $c(\eta)$.  The bipartite class
moment similarly includes the injective labelings within each vertex
class.  We therefore count the canonical representatives in the
outer sums, with one representative for each singleton equivalence
class.
\paragraph{Counting classes at fixed genus.}
Singleton classes have positive genus, as recorded in the following Lemma.
\begin{lemma}[Positive genus of a singleton class]
\label{lem:singleton-positive-genus}
Every singleton class satisfies $g(\gG_\xi)\geq1$.
\end{lemma}
Every support edge and core edge is traversed at least once.  Together
with $k_f\geq1$, this gives
\begin{align}
 1 \leq g(\gG_\xi)&\leq\abs{\gE(\gG_\xi)}\leq2\ell,
 \qquad
 &\abs{\gE(\gU)}&\leq\abs{\zeta}
 =\sum_fm_f\leq\sum_fk_fm_f=2\ell.
 \label{eqn:singleton-genus-range}
\end{align}
 As $g(\gU)=g(\gG_\xi)$, the core's numbers of edges and vertices are elements of $\sA_g$ as defined in \eqref{eqn:core-size-index-set}.

Let $\sC_\ell^{\rm s}(a,v,t)$ be the subset of $\sC_\ell(a,v,t)$ for singleton classes $\xi\in\SquareSingletonClasses{\ell}$ with $|\gE(\gU)|=a$, $|\gV(\gU)|=v$, and $|\zeta|=t$. By injectivity of \eqref{eqn:canonical-core-map}, $|\sC_\ell^{\rm s}(a,v,t)|$ counts such classes.
To bound this, first choose the triple $(\gU, \zeta, r)$ and then assign chain lengths. By \eqref{eqn:core-walk-count}, there are at most $C\ell v(Ca)^t(Cv)^a$ possible triples. For each, valid chain lengthings number at most $\binom{2\ell}{a}\leq(6\ell/a)^a$ (see \eqref{eqn:chain-length-count}). Thus,
\begin{align}
 \abs{\sC_\ell^{\rm s}(a,v,t)}
 &\leq\abs{\sC_\ell(a,v,t)}
 \leq C\ell v(Ca)^t(Cv)^a
       \left(\frac{6\ell}{a}\right)^a.
 \label{eqn:singleton-code-count}
\end{align}
This is the complete-code bound \eqref{eqn:trace-overview-code-count}
restricted to singleton classes.  The same argument applies to the
codes arising from $\RectangularSingletonClasses{i}{\ell}$ at each
fixed first-splice color $i$.

\paragraph{From class counts to core weights.}
For a canonical singleton word $\xi$ with core code $\CoreCode$, let
\begin{align}
 \setS(\xi)&\coloneqq\{f\in\gE(\gU):m_f=1\},
 &a_{\setS}&\coloneqq|\setS(\xi)|.
\end{align}
The chain-multiplicity property in
Lemma~\ref{lem:normalized-core-properties} identifies the original-word
statistics from Definition~\ref{def:singleton-chain-data} as
\begin{align}
 e_1(\xi)&=\sum_{f\in\setS(\xi)}k_f,
 &\varsigma(\xi)&=
   \sum_{f\in\setS(\xi)}\floor{\frac{k_f-1}{2}}.
 \label{eqn:singleton-core-identities}
\end{align}
Since $m_{f}\geq 2$ for $f \notin \setS(\xi)$, the support weight satisfies
\begin{align}
 d^{\abs{\gE(\gG_\xi)}-\ell}
 &=d^{e_1(\xi)/2}
   \prod_{f\notin\setS(\xi)}d^{k_f(2-m_f)/2}
 \leq d^{e_1(\xi)/2}.
 \label{eqn:singleton-chain-charge-split}
\end{align}
Thus long singleton chains can lead to the weight growing even at fixed genus. For example, consider a cycle of length $2\ell-2$ starting at $\xi_0$, with $u$ the opposite vertex, as shown in Figure~\ref{fig:singleton-cycle-compression}. Attach a new vertex $w=\xi_\ell$ via the pendant\footnote{It is incident to the leaf $w$.} edge $\{u,w\}$. Each arc from $\xi_0$ to $u$ has length $\ell-1$, so traversing an arc and then $\{u,w\}$ yields a nonbacktracking path of length $\ell$ between the splices. Cycle edges are traversed once, while $\{u,w\}$ is traversed twice. Thus,
\begin{align}
 g(\gG_\xi)=1,\quad
 \abs{\gE(\gG_\xi)}=2\ell-1,\quad e_1(\xi)=2\ell-2,\quad
 2^{e_1(\xi)}d^{\abs{\gE(\gG_\xi)}-\ell}=(4d)^{\ell-1}.
 \label{eqn:singleton-extremal-example}
\end{align}
After compression, the two cycle arcs become parallel core edges
between $\xi_0$ and $u$, and the pendant edge remains the third core
edge.  The two arc edges are each traversed once and the pendant edge
twice, so $\abs{\zeta}=1+1+2=4$.  Thus the unpenalized weight
$d^{\ell-1}$ exceeds the plural core weight $d$ when $d>1$.

\begin{figure}[h]
\centering
\begin{tikzpicture}[
  x=1.03cm,y=0.72cm,
  vertex/.style={circle,draw=black,fill=white,inner sep=1.3pt,
                 minimum size=5mm,font=\scriptsize},
  suppressible/.style={vertex,fill=black!8,draw=black!55},
  retained/.style={vertex,draw=UCSDBlue,very thick,fill=UCSDBlue!5},
  splice/.style={vertex,draw=ThemeColor,very thick,fill=ThemeColor!8},
  coreedge/.style={very thick,UCSDBlue},
  edgelabel/.style={font=\scriptsize,fill=white,inner sep=1pt},
  every node/.style={font=\scriptsize}
]
 \node at (2.9,2) {Support $\gG_\xi$};
 \node[splice]       (l0) at (0,0) {$\xi_0$};
 \node[suppressible] (la) at (0.9,1.2) {};
 \node[suppressible] (lb) at (3.6,1.2) {};
 \node[suppressible] (lc) at (0.9,-1.2) {};
 \node[suppressible] (ld) at (3.6,-1.2) {};
 \node[inner sep=2pt] (lt) at (2.25,1.2) {$\cdots$};
 \node[inner sep=2pt] (ls) at (2.25,-1.2) {$\cdots$};
 \node[retained]     (lu) at (4.5,0) {$u$};
 \node[splice]       (lw) at (5.8,0) {$w$};
 \draw[thick] (l0)--(la)--(lt);
 \draw[thick] (lt)--(lb)--(lu);
 \draw[thick] (l0)--(lc)--(ls);
 \draw[thick] (ls)--(ld)--(lu);
 \draw[thick] (lu)--(lw);
 \node[edgelabel] at (2.25,0.65) {$\ell-1$ edges};
 \node[edgelabel] at (2.25,-0.65) {$\ell-1$ edges};

 \draw[very thick,-{Latex[length=2.6mm]},black!55]
   (6.4,0)--(7.25,0);

 \node at (10.1,2) {compressed multigraph $\widehat{\gG}_\xi$};
 \node[splice]   (r0) at (7.9,0) {$\xi_0$};
 \node[retained] (ru) at (10.25,0) {$u$};
 \node[splice]   (rw) at (12.35,0) {$w$};
 \draw[coreedge] (r0) to[bend left=35]
   node[edgelabel,above=3pt] {$f_1\,(\ell-1)$} (ru);
 \draw[coreedge] (ru)--node[edgelabel,above=3pt] {$f_2\,(1)$} (rw);
 \draw[coreedge] (r0) to[bend right=35]
   node[edgelabel,below=3pt] {$f_3\,(\ell-1)$} (ru);
\end{tikzpicture}
\caption{Compression of the cycle example. Red splice vertices $\xi_0$ and $w=\xi_\ell$ and the blue branch vertex $u$ are retained; gray arc vertices are suppressed. Edge labels on the right give chain lengths in parentheses. The two parallel edges are each traversed once, while the pendant edge is traversed twice. Both graphs have genus one, and $\abs{\zeta}=4$.}
\label{fig:singleton-cycle-compression}
\end{figure}
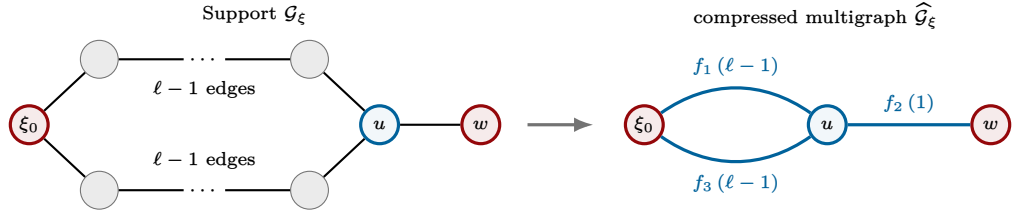

The penalty $d^{-\varsigma(\xi)}$ compensates for this growth, as proved in the following lemma.

\begin{lemma}[Compression of the singleton weight]
\label{lem:singleton-chain-compression}
For $d\geq1$, every canonical singleton word $\xi$ of length $2\ell$
with core code $\CoreCode$ satisfies
\begin{align}
 d^{\abs{\gE(\gG_\xi)}-\ell-\varsigma(\xi)}
 \leq d^{3\abs{\gE(\gU)}/2-\abs{\zeta}/2}.
 \label{eqn:singleton-certificate-core-bound}
\end{align}
\end{lemma}

Apply Lemma~\ref{lem:singleton-chain-compression} with $d=s\geq1$.
Together with \eqref{eqn:core-support-edge-count}, it shows that the
full singleton weight of a class with fixed sizes $a,v,t$ satisfies,
for $\Gamma\geq1$,
\begin{align}
 \Gamma^{2\left(\abs{\gE(\gG_\xi)}-
       \abs{\gS_{\{\xi_0,\xi_\ell\}}(\gG_\xi)}\right)}
 s^{\abs{\gE(\gG_\xi)}-\ell-\varsigma(\xi)}
 &\leq\Gamma^{2a}s^{3a/2-t/2}.
\end{align}
This upper bound is independent of the chain lengths.  Hence
multiplying \eqref{eqn:singleton-code-count} by
$\Gamma^{2a}s^{3a/2-t/2}$ bounds the sum of the singleton weights
over all classes with these sizes.

We now sum over the core sizes and walk lengths allowed at genus $g$.
Using the same $C_0$ as in \eqref{eqn:genus-weight-definition}, define
the weighted count, for $s,\Gamma\geq1$, by
\begin{align}
 \Theta_\ell^{\rm s}(g;s,\Gamma)
 &\coloneqq\sum_{(a,v)\in\sA_g}\sum_{t=1}^{2\ell}
   \ell v(C_0a)^t(C_0v)^a
   \left(\frac{6\ell}{a}\right)^a
   \Gamma^{2a}s^{3a/2-t/2}.
 \label{eqn:singleton-genus-weight-definition}
\end{align}
Here $a,v,t$ enumerate the core edge count, vertex count, and walk length.
Equation~\eqref{eqn:singleton-genus-range} gives $a\leq t\leq2\ell$;
allowing $t<a$ and relaxing the chain-length constraints only enlarges
the sum.
By the weight bound above and \eqref{eqn:singleton-code-count},
\begin{align}
 \sum_{\substack{\xi\in\SquareSingletonClasses{\ell}\\
                  g(\gG_\xi)=g}}
 \Gamma^{2\left(\abs{\gE(\gG_\xi)}-
       \abs{\gS_{\{\xi_0,\xi_\ell\}}(\gG_\xi)}\right)}
 s^{\abs{\gE(\gG_\xi)}-\ell-\varsigma(\xi)}
 &\leq
 \sum_{\substack{\xi\in\SquareSingletonClasses{\ell}\\
                  g(\gG_\xi)=g}}
 \Gamma^{2\abs{\gE(\gU)}}
 s^{3\abs{\gE(\gU)}/2-\abs{\zeta}/2}
 \nonumber\\
 &\leq C\Theta_\ell^{\rm s}(g;s,\Gamma),
 \qquad 1\leq g\leq2\ell.
 \label{eqn:singleton-weighted-class-count}
\end{align}
At $s=\Gamma=1$, the left-hand side counts the singleton equivalence
classes of genus $g$.  Summing over $1\leq g\leq2\ell$ therefore
counts the terms in the singleton sum in \eqref{eqn:square-T0-T1}.
The same bound holds for
$\RectangularSingletonClasses{i}{\ell}$ at each fixed first-splice color.

\paragraph{Summing the class moments.}
Combining \eqref{eqn:singleton-certificate-core-bound} with the edge-count
identity \eqref{eqn:core-support-edge-count} converts the singleton
moment assumptions into
\begin{align}
 |\GraphMoment_\ell(\WordClass{\xi})|
 &\leq K_\ell N^2
   \left(\frac{\chi_\ell}{N}\right)^{g(\gG_\xi)}
   \Gamma_\ell^{2\abs{\gE(\gU)}}
   d^{3\abs{\gE(\gU)}/2-\abs{\zeta}/2},
 \nonumber\\
 |\widetilde{\GraphMoment}_{\ell,i}(\WordClass{\xi})|
 &\leq\widetilde K_\ell nm
   \left(\frac{\widetilde\chi_\ell}{N}\right)^{g(\gG_\xi)}
   \widetilde\Gamma_\ell^{2\abs{\gE(\gU)}}
   d^{3\abs{\gE(\gU)}/2-\abs{\zeta}/2}.
 \label{eqn:singleton-compressed-class-moments}
\end{align}

Lemma~\ref{lem:singleton-positive-genus} and
\eqref{eqn:singleton-genus-range} restrict the genus sum to
$1\leq g\leq2\ell$.  Applying the triangle inequality, the singleton
class-moment bound in \eqref{eqn:singleton-compressed-class-moments}, and
\eqref{eqn:singleton-weighted-class-count} with
$(s,\Gamma)=(d,\Gamma_\ell)$ gives
\begin{subequations}
\begin{align}
 |\SquareSingletonContribution|
 &\leq\sum_{\xi\in\SquareSingletonClasses{\ell}}
       \abs{\GraphMoment_\ell(\WordClass{\xi})}
 \leq CK_\ell N^2
   \sum_{g=1}^{2\ell}\left(\frac{\chi_\ell}N\right)^g
   \Theta_\ell^{\rm s}(g;d,\Gamma_\ell).
 \label{eqn:roadmap-hermitian-singleton-core-sum}
\end{align}
Applying the same argument at each first-splice color with
$(s,\Gamma)=(d,\widetilde\Gamma_\ell)$ and summing over
$i\in\{1,2\}$ gives
\begin{align}
 |\RectangularSingletonContribution|
 &\leq C\widetilde K_\ell nm
   \sum_{g=1}^{2\ell}\left(\frac{\widetilde\chi_\ell}N\right)^g
   \Theta_\ell^{\rm s}(g;d,\widetilde\Gamma_\ell).
 \label{eqn:roadmap-bipartite-singleton-core-sum}
\end{align}
\end{subequations}
The factor two for the first-splice colors is absorbed in $C$.

\paragraph{Bounding the weighted counts.}
The remaining size and walk-length sums are bounded by the next lemma,
whose proof is deferred to Subsection~\ref{sec:trace-combinatorial-proofs}.

\begin{lemma}[Singleton weight as a function of genus]
\label{lem:singleton-genus-weight}
For $s,\Gamma\geq1$, $\ell\geq3$, and $g\geq1$,
\begin{align}
 \Theta_\ell^{\rm s}(g;s,\Gamma)
 &\leq C\ell^3\Gamma^2s^{3/2}
       (C\ell^3\Gamma^6s^{9/2})^g
       \Psi_\ell\!\left(\frac{C(g+1)}{\sqrt s}\right),
 \label{eqn:singleton-fixed-genus-bound}
\end{align}
where $\Psi_\ell$ is defined in
\eqref{eqn:plural-finite-geometric-factor}.
\end{lemma}

Finally, applying Lemma~\ref{lem:singleton-genus-weight} to
\eqref{eqn:roadmap-hermitian-singleton-core-sum} and
\eqref{eqn:roadmap-bipartite-singleton-core-sum}, and taking
$C_{\rm enum}$ large enough to absorb the constants raised to $g$, yields
\begin{subequations}
\begin{align}
 |\SquareSingletonContribution|
 &\leq CK_\ell N^2\ell^3\Gamma_\ell^2d^{3/2}
       \sum_{g=1}^{2\ell}
       \left(\frac{C_{\rm enum}\chi_\ell\ell^3\Gamma_\ell^6d^{9/2}}{N}\right)^g
       \Psi_\ell\!\left(\frac{C(g+1)}{\sqrt d}\right),
 \label{eqn:roadmap-hermitian-singleton-genus-sum}
 \\
 |\RectangularSingletonContribution|
 &\leq C\widetilde K_\ell nm\ell^3
       \widetilde\Gamma_\ell^2d^{3/2}
       \sum_{g=1}^{2\ell}
       \left(\frac{C_{\rm enum}\widetilde\chi_\ell\ell^3
                    \widetilde\Gamma_\ell^6d^{9/2}}{N}\right)^g
       \Psi_\ell\!\left(\frac{C(g+1)}{\sqrt{d}}\right).
 \label{eqn:roadmap-bipartite-singleton-genus-sum}
\end{align}
\end{subequations}

\subsection{Step 6: Genus summation and conclusion}
\label{subsec:trace-step6}

The positive-genus terms in the plural bounds
\eqref{eqn:roadmap-hermitian-plural-genus-sum}--
\eqref{eqn:roadmap-bipartite-plural-genus-sum} and singleton bounds
\eqref{eqn:roadmap-hermitian-singleton-genus-sum}--
\eqref{eqn:roadmap-bipartite-singleton-genus-sum} are controlled by
the common series
\begin{align}
 \sS_\ell(\EnumParameter;s)
 \coloneqq\sum_{g=1}^{2\ell}\EnumParameter^g
   \sum_{t=1}^{2\ell}
      \left(\frac{C_0(g+1)}{\sqrt s}\right)^t,
\label{eqn:roadmap-genus-sum}
\end{align}
where $C_0$ dominates the constants in their geometric factors.

\begin{lemma}[Common genus sum]
\label{lem:roadmap-common-genus-sum}
There are numerical constants $c,C>0$ and $d_0\geq1$ such that
\begin{align}
      \sS_\ell(\EnumParameter;s)\leq C\ell^2\EnumParameter,
      \label{eqn:roadmap-genus-sum-bound}
\end{align}
whenever $s\geq d_0$, $0<\EnumParameter\leq1/2$, and $\ell\leq c\sqrt s\log(\EnumParameter^{-1})$.
\end{lemma}

\begin{proof}[Proof of Lemma~\ref{lem:general-nonbacktracking-trace-bounds}]
      We prove the bound for the Hermitian case. The proof for the bipartite case is similar. Choose $c,d_0$ as in Lemma~\ref{lem:roadmap-common-genus-sum}, increasing $d_0$ so that $C_0/\sqrt{d_0}\leq1/2$.
\begin{align}
 \EnumParameter_{\rm p}
 &\coloneqq\frac{C_{\rm enum}\chi_\ell\ell^3d^3}{N},
 &\EnumParameter_{\rm s}
 &\coloneqq\frac{C_{\rm enum}\chi_\ell\ell^3\Gamma_\ell^6d^{9/2}}{N}.
\end{align}
Since $d,\Gamma_\ell\geq1$, condition~\eqref{eqn:hermitian-saddle-condition}
gives $0<\EnumParameter_{\rm p}\leq\EnumParameter_{\rm s}\leq1/2$
and the required length bound for both ratios.  Hence
Lemma~\ref{lem:roadmap-common-genus-sum} gives
\begin{align}
 \sS_\ell(\EnumParameter;d)
 &\leq C\ell^2\EnumParameter,
 \qquad \EnumParameter\in\{\EnumParameter_{\rm p},\EnumParameter_{\rm s}\}.
 \label{eqn:roadmap-two-genus-sums}
\end{align}
The genus-zero term satisfies
\begin{align}
 \sum_{t=1}^{2\ell}\left(\frac{C_0}{\sqrt d}\right)^t
 \leq\frac C{\sqrt d}.
 \label{eqn:roadmap-genus-zero-geometric-sum}
\end{align}

Substituting \eqref{eqn:roadmap-two-genus-sums} and
\eqref{eqn:roadmap-genus-zero-geometric-sum} into the Hermitian genus sums
\eqref{eqn:roadmap-hermitian-plural-genus-sum} and
\eqref{eqn:roadmap-hermitian-singleton-genus-sum}, and extending the
plural genus range to $2\ell$, gives by
\eqref{eqn:trace-plural-singleton-split-square}
\begin{align}
 \E\Tr\mB^\ell(\mB^\ell)^*
 &\leq |\SquarePluralContribution|+|\SquareSingletonContribution|
\leq CK_\ell N^2\left(
    \ell^3\sqrt d
    +\frac{\chi_\ell\ell^8d^4}{N}
    +\frac{\chi_\ell\Gamma_\ell^8\ell^8d^6}{N}
    \right).
\end{align}

This proves \eqref{eqn:general-hermitian-trace-bound}.
The same argument with tilded parameters and prefactor $nm$, using
\eqref{eqn:bipartite-saddle-condition} and
\eqref{eqn:trace-plural-singleton-split-rectangular}, proves
\eqref{eqn:general-bipartite-trace-bound}.
\end{proof}

%%%%%%%%%%%%%%%%%%%%%%%%%%%%%%%%%%%%%%%%%%%%%%
\subsection{Proofs of the technical lemmas}
\label{sec:trace-combinatorial-proofs}
%%%%%%%%%%%%%%%%%%%%%%%%%%%%%%%%%%%%%%%%%%%%%%

We collect the deferred proofs for the canonical core code, the plural
and singleton genus weights, the singleton-chain compression and
positive-genus lemmas, and the common genus-sum estimate.

\begin{proof}[Proof of Lemma~\ref{lem:normalized-core-properties}]
\leavevmode

\paragraph{(i)}
\emph{Canonical walk.}
By construction, replacing each chain traversal in $\xi$ by its
corresponding dart gives the rooted closed walk $\zeta$.  Because $\xi$
visits every support edge, $\zeta$ visits every core edge and vertex.
The relabeling by $\tau$ puts the surviving vertices in first-appearance
order, separately in the two color classes in the rectangular setting.
The stated root labels follow from rooting the walk at the first splice.

\paragraph{(ii)}
\emph{Vertex degrees.}
Maximality of the chains leaves no nonsplice vertex of degree two.  A
nonsplice vertex cannot have degree one either, since the word would have
to reverse its incident edge there, contrary to the nonbacktracking
condition.  Thus every nonsplice core vertex has degree at least three.
Connectedness gives degree at least one at each splice vertex, whether or
not the two splices coincide.

\paragraph{(iii)}
\emph{Splice recovery and reconstruction.}
Let $f_t$ be the underlying edge of the $t$-th dart of $\zeta$.
Because the second splice is retained, the partial sums of
$(k_{f_t})_{t=1}^{\abs{\zeta}}$ reach $\ell$ at the marked visit $r$.
They are strictly increasing since every chain has positive length, so
\eqref{eqn:core-mark-recovery} determines $r$ uniquely.

Starting from $\CoreCode$, subdivide each core edge $f$ into a fixed
chain of $k_f$ support edges.  Orient this chain by the first dart of
$f$ in $\zeta$, and replace every later occurrence of a dart of $f$ by the
corresponding forward or reverse traversal of the same chain.  The two
orientations of a loop distinguish its two possible traversals.  The root
and the recorded visit $r$ identify the two splices.  Finally, label the
vertices of the expanded walk in order
of first appearance.  In the rectangular setting, retain the core colors,
alternate the colors along each expanded chain, and assign labels
separately within the two vertex classes.  This reconstructs a unique
canonical word $\xi$, proving injectivity.

\paragraph{(iv)}
\emph{Chain lengths and multiplicities.}
Once a nonbacktracking word enters an internal degree-two vertex of a
chain, the edge by which it arrived is forbidden and only the next edge
of the chain remains available.  Every visit therefore traverses the
entire chain, so all its support edges have traversal multiplicity $m_f$.
Each chain is nonempty and each core edge is visited, giving $k_f\geq1$
and $m_f\geq1$.

In the rectangular setting, colors alternate along every support edge of
a chain, so its endpoints have opposite colors if and only if its length
is odd.

\paragraph{(v)}
\emph{Expansion identities.}
The internal vertices belonging to distinct maximal chains are disjoint.
Replacing a core edge $f$ by its chain adds $k_f$ support edges and
$k_f-1$ internal vertices.  Each traversal of $f$ contributes $k_f$
steps to the expanded word.  Summing these contributions proves
\eqref{eqn:core-identities}.  The internal vertices of these chains
are exactly $\gS_{\{\xi_0,\xi_\ell\}}(\gG_\xi)$, so their number is
$\sum_f(k_f-1)=\abs{\gE(\gG_\xi)}-\abs{\gE(\gU)}$.
This proves \eqref{eqn:core-support-edge-count}, also for loop chains
and coincident splices.

\paragraph{(vi)}
\emph{Genus preservation.}
Suppressing one internal degree-two vertex removes one vertex and one
edge, and canonical relabeling changes neither count.  Hence both
operations preserve the genus, proving
\eqref{eqn:core-genus-preservation}.\qedhere
\end{proof}

\begin{proof}[Proof of Lemma~\ref{lem:reduced-graph-code}]
\leavevmode

\paragraph{(i)}
By Lemma~\ref{lem:normalized-core-properties}, every nonsplice core vertex
has degree at least three, each splice vertex has degree at least one, and
$g(\gU)=g(\gG_{\xi})$.  Thus the
handshaking identity gives, in the distinct-splice case,
\begin{align}
 2\abs{\gE(\gU)}
 =\sum_{u\in\gV(\gU)}\deg_{\gU}(u)
 \geq3\bigl(\abs{\gV(\gU)}-2\bigr)+2
 =3\abs{\gV(\gU)}-4.
 \label{eqn:roadmap-core-degree-count}
\end{align}
If the splice vertices coincide, the sharper bound
$2\abs{\gE(\gU)}\geq3\abs{\gV(\gU)}-2$ holds, so
\eqref{eqn:roadmap-core-degree-count} is valid in both cases.  Since
\eqref{eqn:core-genus-preservation} gives
$\abs{\gV(\gU)}=\abs{\gE(\gU)}-\abs{g(\gG_\xi)}+1$, we obtain
\begin{align}
 \abs{\gE(\gU)}&\leq3\abs{g(\gG_\xi)}+1,
 &\abs{\gV(\gU)}&\leq2\abs{g(\gG_\xi)}+2,
\end{align}
which proves \eqref{eqn:core-size}.

\paragraph{(ii)}
To count the triples $(\gU,\zeta,r)$, number the core edges
in order of first traversal.
At each of the $\abs{\zeta}$ steps, choose one of the
$\abs{\gE(\gU)}$ edge numbers and one of its two orientations; this costs
at most $\bigl(2\abs{\gE(\gU)}\bigr)^{\abs{\zeta}}$.
Only the first traversal of an edge can reveal a new endpoint.
For each such traversal there are at most $2\abs{\gV(\gU)}$ choices
for the normal endpoint labels, contributing
$\bigl(2\abs{\gV(\gU)}\bigr)^{\abs{\gE(\gU)}}$.
Finally, allowing independent choices of the root, the color of the first
splice in the rectangular case, and the marked visit $r$ costs at most
$C\ell\abs{\gV(\gU)}$.  Choosing $r$ independently only enlarges the
count, since valid codes require compatibility with the chain lengths
through \eqref{eqn:core-mark-recovery}.  Multiplying these
factors proves \eqref{eqn:core-walk-count}.

\paragraph{(iii)}
Fix the core-walk data $(\gU,\zeta,r)$ and count their completions
by chain lengths.  The weighted identity in
\eqref{eqn:core-identities} and $m_f\geq1$ imply
\begin{align}
 \sum_{f\in\gE(\gU)}k_f
 \leq\sum_{f\in\gE(\gU)}k_fm_f=2\ell.
\end{align}
We relax the weighted constraint to this inequality.  Stars and bars then
gives
\begin{align}
 \#\{(k_f):k_f\geq1,\ \textstyle\sum_fk_f\leq2\ell\}
 &\leq\sum_{t=\abs{\gE(\gU)}}^{2\ell}
       \binom{t-1}{\abs{\gE(\gU)}-1}
 =\binom{2\ell}{\abs{\gE(\gU)}}
 \nonumber\\
 &\leq\left(\frac{2\mathrm e\ell}{\abs{\gE(\gU)}}\right)^{\abs{\gE(\gU)}}
 \leq\left(\frac{6\ell}{\abs{\gE(\gU)}}\right)^{\abs{\gE(\gU)}}.
 \label{eqn:roadmap-chain-length-calculation}
\end{align}
This is an upper bound because we relaxed the true weighted identity; in
particular, we did not replace it by the generally false equality
$\sum_fk_f=2\ell$.  Equations \eqref{eqn:chain-length-count} and
\eqref{eqn:roadmap-chain-length-calculation} give the third assertion.
\end{proof}

\begin{proof}[Proof of Lemma~\ref{lem:plural-genus-weight}]
For $g=0$, the index set contains only $(a,v)=(1,2)$.
Substitution in \eqref{eqn:genus-weight-definition} gives the first
inequality with $\ell^2$ in place of $\ell^3$.
For $g\geq1$ and $(a,v)\in\sA_g$, the definition gives
\begin{align}
 g\leq a\leq3g+1,\qquad
 v=a-g+1\leq a,\qquad v\leq2g+2,\qquad
 |\sA_g|\leq2g+2.
 \label{eqn:roadmap-plural-core-ranges}
\end{align}
In each summand, combine the chain count with the factors raised to the
number of core edges:
\begin{align}
 (C_0v)^a\left(\frac{6\ell}{a}\right)^a s^a
 &\leq(C\ell s)^a
 \leq(C\ell s)^{3g+1}
 \leq C\ell s(C\ell^3s^3)^g.
 \label{eqn:plural-core-power-elimination}
\end{align}
The remaining walk factor is at most
$(C(g+1)/\sqrt s)^t$.  Finally,
\begin{align}
 \sum_{(a,v)\in\sA_g}\ell v
 \leq\ell(2g+2)^2\leq C\ell\,4^g.
\end{align}
Absorb $4^g$ into the constant raised to $g$ and sum over $t$ using
$\Psi_\ell$.  This proves the positive-genus bound with an $\ell^2$
prefactor, which we enlarge to $\ell^3$ in both cases.
\end{proof}

\begin{proof}[Proof of Lemma~\ref{lem:singleton-genus-weight}]
The definition of $\sA_g$ in \eqref{eqn:core-size-index-set} gives
$v\leq a\leq3g+1$ and $|\sA_g|,v\leq2g+2$.  Consequently,
\begin{align}
 (C_0v)^a\left(\frac{6\ell}{a}\right)^a\Gamma^{2a}s^{3a/2}
 &\leq(C\ell\Gamma^2s^{3/2})^a
 \leq C\ell\Gamma^2s^{3/2}
       (C\ell^3\Gamma^6s^{9/2})^g.
\end{align}
The remaining walk factor is bounded by
$(C(g+1)/\sqrt s)^t$, and
$\sum_{(a,v)\in\sA_g}\ell v\leq C\ell\,4^g$.
Absorb $4^g$ into the constant raised to $g$ and sum over $t$ to obtain
$\Psi_\ell$.  This proves the displayed bound, even with $\ell^2$
in place of its prefactor $\ell^3$.
\end{proof}

\begin{proof}[Proof of Lemma~\ref{lem:singleton-chain-compression}]
The core identities give the exact factorization
\begin{align}
 d^{\abs{\gE(\gG_\xi)}-\ell}
 =d^{\frac12\sum_{f\in\gE(\gU)}k_f(2-m_f)}.
 \label{eqn:certificate-chain-factorization}
\end{align}
By \eqref{eqn:singleton-core-identities}, a singleton chain of length
$k$ contributes $d^{k/2-\floor{(k-1)/2}}$ after the penalty is included.
For both parities,
\begin{align}
 d^{k/2-\floor{(k-1)/2}}
 =\begin{cases}
   d,&k\textnormal{ even},\\
   d^{1/2},&k\textnormal{ odd}
  \end{cases}
 \leq d.
 \label{eqn:roadmap-one-singleton-chain-payment}
\end{align}
For every repeated chain, $m_f\geq2$ and $k_f\geq1$, so
\begin{align}
 d^{k_f(2-m_f)/2}\leq d^{(2-m_f)/2}.
 \label{eqn:certificate-plural-payment}
\end{align}
Multiplying these bounds and using
$\abs{\zeta}=\sum_fm_f$ yields
\begin{align}
 d^{\abs{\gE(\gG_\xi)}-\ell-\varsigma(\xi)}
 &\leq d^{a_{\setS}+\frac12\sum_{f\notin\setS(\xi)}(2-m_f)}
  =d^{\abs{\gE(\gU)}+a_{\setS}/2-\abs{\zeta}/2}
 \leq d^{3\abs{\gE(\gU)}/2-\abs{\zeta}/2},
 \label{eqn:roadmap-total-singleton-payment}
\end{align}
where the last step uses $a_{\setS}\leq\abs{\gE(\gU)}$ and $d\geq1$.
This proves \eqref{eqn:singleton-certificate-core-bound} in both hosts.
\end{proof}

\begin{proof}[Proof of Lemma~\ref{lem:singleton-positive-genus}]
Suppose that a singleton word had $\abs{g(\gG_\xi)}=0$.  Its connected support graph would be
a tree.  Every support vertex other than the splice vertices has degree at
least two: visiting a degree-one nonsplice vertex would force an immediate
reversal.

Assume first that the splice vertices are distinct.  Every leaf is then a
splice vertex.  The elementary tree identity
\begin{align}
 \#\{\textnormal{leaves}\}
 =2+\sum_{\deg(u)\geq3}(\deg(u)-2)
\end{align}
shows that the two splices are the only leaves and that no vertex has
degree at least three.  The support is therefore the unique path joining
the splices.  Each of the two nonbacktracking halves traverses this path
once, so every support edge has multiplicity two, contradicting the presence
of a singleton edge.  If the splice vertices coincide, either half would be a
positive-length closed nonbacktracking walk in a tree, which is impossible.
Thus every singleton class has $\abs{g(\gG_\xi)}\geq1$.
\end{proof}

\begin{proof}[Proof of Lemma~\ref{lem:roadmap-common-genus-sum}]
We denote $L_{\EnumParameter}\coloneqq\log(\EnumParameter^{-1})$.  Since $\EnumParameter\leq1/2$, we have
$L_{\EnumParameter}\geq\log2$.  For every $x\geq0$,
\begin{align}
 \sum_{t=1}^{2\ell}x^t\leq2\ell(1+x^{2\ell}).
\end{align}
Using $\EnumParameter^g=e^{-L_{\EnumParameter}g}$ in \eqref{eqn:roadmap-genus-sum}, we obtain
\begin{align}
 \sS_\ell(\EnumParameter;s)
 &\leq2\ell\sum_{g=1}^{2\ell}e^{-L_{\EnumParameter}g}
   +2\ell\sum_{g=1}^{2\ell}\exp\{\Phi_{\EnumParameter}(g)\}
 \leq4\ell \EnumParameter
   +2\ell\sum_{g=1}^{2\ell}\exp\{\Phi_{\EnumParameter}(g)\},
 \label{eqn:roadmap-genus-sum-reduction}
\end{align}
where we define $\Phi_{\EnumParameter}(g)$ as follows:
\begin{align}
 \Phi_{\EnumParameter}(g)
 \coloneqq-L_{\EnumParameter}g+2\ell
   \log\left(\frac{C_0(g+1)}{\sqrt s}\right).
\end{align}
As a function of real $g>-1$, $\Phi_{\EnumParameter}$ is strictly concave and has its
unique stationary point at
\begin{align}
 g_*\coloneqq\frac{2\ell}{L_{\EnumParameter}}-1.
\end{align}
If $g_*\leq1$, then $\Phi_{\EnumParameter}$ decreases on $[1,\infty)$.  After increasing
$d_0$ if necessary,
\begin{align}
 \Phi_{\EnumParameter}(g)
 \leq\Phi_{\EnumParameter}(1)
 =-L_{\EnumParameter}+2\ell\log\left(\frac{2C_0}{\sqrt s}\right)
 \leq-L_{\EnumParameter}.
\end{align}
If $g_*>1$, then $\ell>L_{\EnumParameter}$.  The length assumption
$\ell\leq c\sqrt s\,L_{\EnumParameter}$ gives
\begin{align}
 \frac{2C_0\ell}{L_{\EnumParameter}\sqrt s}\leq2C_0c\leq1
\end{align}
once $c$ is chosen sufficiently small.  Evaluation at the global maximum
then yields
\begin{align}
 \Phi_{\EnumParameter}(g_*)
 &=L_{\EnumParameter}-2\ell
   +2\ell\log\left(\frac{2C_0\ell}{L_{\EnumParameter}\sqrt s}\right)
 \leq L_{\EnumParameter}-2\ell
 \leq-L_{\EnumParameter}.
\end{align}
Thus $\Phi_{\EnumParameter}(g)\leq-L_{\EnumParameter}$ for every integer
$1\leq g\leq2\ell$ in both cases.  Substitution into
\eqref{eqn:roadmap-genus-sum-reduction} gives
\begin{align}
 \sS_\ell(\EnumParameter;s)
 \leq4\ell \EnumParameter+4\ell^2\EnumParameter
 \leq C\ell^2\EnumParameter,
\end{align}
as claimed.
\end{proof}

\section{Future Directions}
\label{sec:future-directions}
We conclude this paper by discussing some future directions.

\paragraph{Non-Hermitian sparse random matrices}
Upper bound on spectral radius of non-Hermitian sparse random matrices with independent entries was established in \cite{benaych2020spectral} as well. It would be interesting to explore how thresholding affects the spectral radius in the non-Hermitian setting, similar to the results in Theorems~\ref{thm:rectangular-largest-singular-value} and~\ref{thm:hermitian-edge} for Hermitian matrices. It would also be valuable to investigate how the trace estimates developed in Section~\ref{sec:trace-estimates} can be extended to the non-Hermitian setting.

\paragraph{Hypergraph extensions}
Let $\bm{\mathcal{A}}\in \R^{\otimes \ell}$ be the adjacency tensor of a hypergraph $\gH$ for some $\ell \geq 3$. The adjacency matrix $\mA$ of $\gH$ is formed by contracting $\bm{\mathcal{A}}$, so that $\emA_{ij}$ equals the number of hyperedges containing the vertex pair $(i,j)$. Consequently, the entries of $\mA$ are dependent in an essential way.
In \cite{au2023spectral}, it was demonstrated that the joint spectral distribution of these contracted tensor ensembles can be closely approximated by a semicircular family. The works \cite{dumitriu2025partial, dumitriu2026optimal} established concentration results for regularized adjacency matrices of sparse hypergraphs, though without explicit sharp norm bounds. Since both the nonbacktracking matrix $\mB$ of $\gH$ and the related Ihara--Bass formula have been developed in \cite{stephan2022sparse}, it would be of interest to derive sharp norm bounds for these objects using techniques analogous to those in Theorems~\ref{thm:rectangular-largest-singular-value} and~\ref{thm:hermitian-edge}.

\section*{Acknowledgements}
The author would like to thank Prof. Ioana Dumitriu for stimulating discussions and her valuable feedback on the early version of this manuscript.

\printbibliography

\newpage
\appendix
\section{Numerical Experiments}
\label{sec:numerical-experiments}

\subsection{Experiments for the bipartite model}

We consider the homogeneous bipartite Erd\H{o}s--R\'enyi graph $\gG(n,m,p)$ with parameters
\[
 n/m=16,\qquad N=n+m\in\{850,1700,3400,6800\},\qquad
 d=0.4\log N,\qquad p=d/\sqrt{mn}.
\]
The three panels of Figure~\ref{fig:numerical-bipartite-edges} present the smallest and largest singular values of the centered active blocks in the three cases: $(U_{\rm r},U_{\rm c},L_{\rm c})=(m,n,0)$ (no thresholding), $(U_{\rm r},U_{\rm c},L_{\rm c})
=(\lfloor1.5d\rfloor,\lfloor5d\rfloor,0)$, and $(U_{\rm r},U_{\rm c},L_{\rm c})=(m,n,\lceil3d\rceil)$, respectively. The lower row cutoff is $L_{\rm r}=0$ throughout. At each size, we sample 20 independent graphs and apply all three restrictions to each graph using Algorithm~\ref{alg:row-column-projection}, retaining the original centering and normalization in
\eqref{eqn:compressed-rectangular-matrix}. Circles show individual trials and solid curves show their means. Dashed curves show the leading thresholds of Theorems~\ref{thm:rectangular-largest-singular-value}
and~\ref{thm:rectangular-smallest-singular-value}, which in all panels are
\[
 \sigma_{\rm L}=\tfrac32\sqrt{1-p},\qquad
 \sigma_{\rm R}=\tfrac52\sqrt{1-p}.
\]
These left-edge examples explore parameters beyond the sufficient
hypotheses of Theorem~\ref{thm:rectangular-smallest-singular-value}:
the coefficient $0.4$ in $d=0.4\log N$ is not chosen to satisfy its
sufficiently large logarithmic sparsity requirement, and the lower
cutoffs do not satisfy \eqref{eqn:rectangular-edge-lower-cutoff}
asymptotically. Even $L_{\rm c}=\lceil3d\rceil$ gives
$(1-p)^2L_{\rm c}/d\to3$, whereas
$\colvarmin=4(1-p)\to4$.
\begin{figure}[H]
\centering
\includegraphics[width=0.99\linewidth]{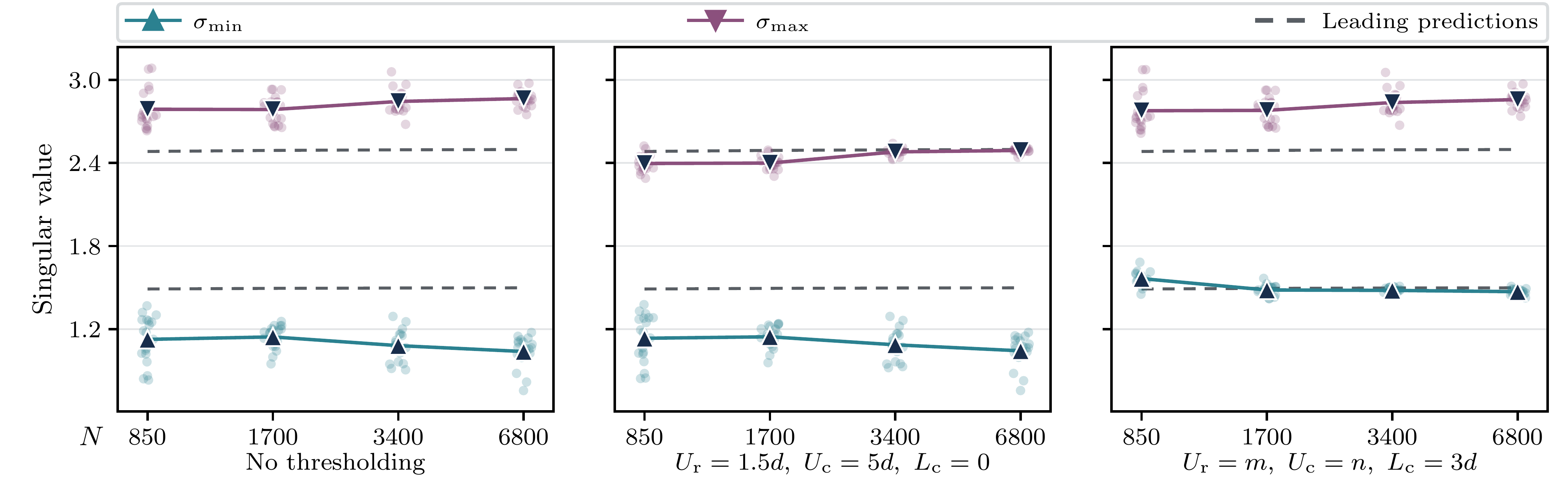}
\caption{Bipartite singular-value edges across graph sizes under
three degree restrictions.}
\label{fig:numerical-bipartite-edges}
\end{figure}

Figure~\ref{fig:numerical-bipartite-nonbacktracking-spectra} uses
one graph from the same model with $n=3200$, $m=200$,
$d=0.4\log3400$, and $p=d/\sqrt{mn}$.
The first two panels use the same cutoffs as the edge experiment;
the third uses $(U_{\rm r},U_{\rm c},L_{\rm c})
=(m,n,\lceil4d\rceil)$, with $L_{\rm r}=0$ throughout.
The nonbacktracking matrices are formed from the uncentered
adjacency, with the original normalization by $\sqrt d$ as in
\eqref{eqn:nonbacktracking-definition-bipartite}.
Every eigenvalue location, including zeros from the host completion,
is plotted.  Hollow markers indicate the repeated eigenvalues
in all three panels:
\begin{align}
 z=0,
 \qquad
 z=\pm d^{-1/2},
 \qquad
 z=\pm\ii d^{-1/2}.
\end{align}
Stars mark nonzero extrema on the real and imaginary axes.
The unit circle is shown as a reference.

\begin{figure}[H]
\centering
\includegraphics[width=0.99\linewidth]{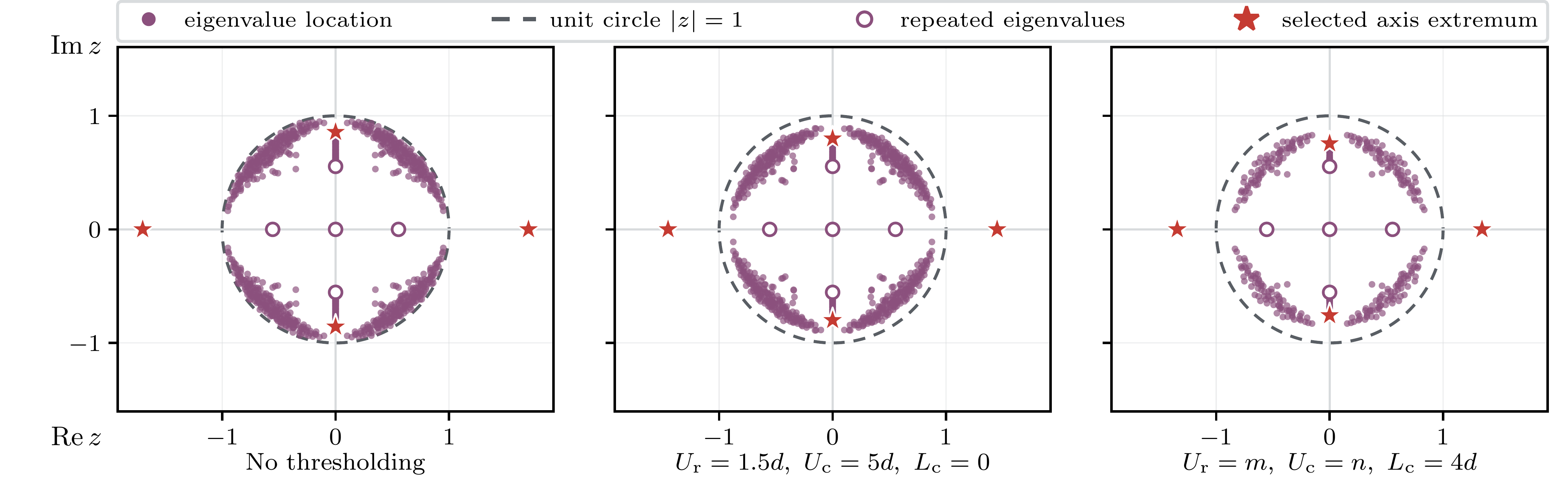}
\caption{Bipartite nonbacktracking spectra at $N=3400$ under
three degree restrictions.}
\label{fig:numerical-bipartite-nonbacktracking-spectra}
\end{figure}

\subsection{Experiments for the Hermitian model}

We consider the homogeneous undirected Erd\H{o}s--R\'enyi graph
$\gG(N,p)$ without self-loops, with parameters
\[
 N\in\{1000,2000,4000,8000\},\qquad
 d=\log N,\qquad p=d/(N-1).
\]
The three panels of Figure~\ref{fig:numerical-hermitian-edge} present
the smallest and largest eigenvalues of the centered active matrices
in the three cases: $U=N-1$ (no thresholding), $U=\lfloor2d\rfloor$,
and $U=\lfloor1.5d\rfloor$, respectively. The lower cutoff is $L=0$
throughout. At each size, we sample 20 independent graphs and apply
all three restrictions to each graph using
Algorithm~\ref{alg:row-column-projection}, retaining the original
centering and normalization in \eqref{eqn:compressed-hermitian-matrix}.
Circles show individual trials and solid curves show their means.
Dashed curves show the leading thresholds of
Theorem~\ref{thm:hermitian-edge}, which in all panels are
\[
 \pm(2-p).
\]
\begin{figure}[H]
\centering
\includegraphics[width=0.99\linewidth]{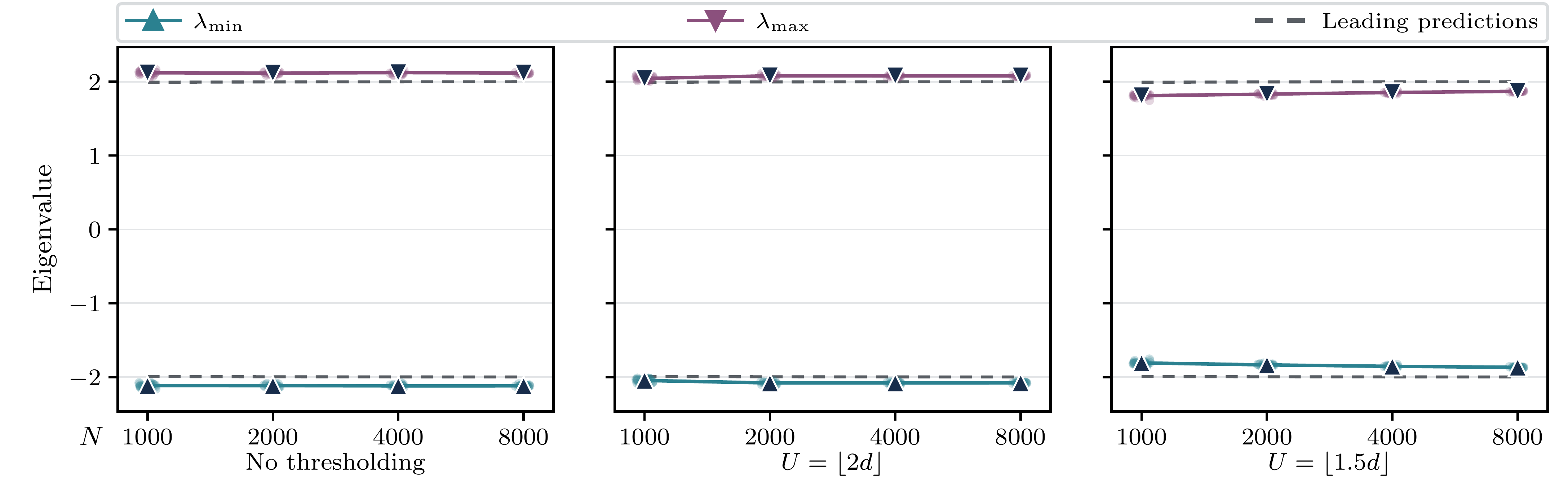}
\caption{Hermitian spectral edges across graph sizes under
three upper degree cutoffs.}
\label{fig:numerical-hermitian-edge}
\end{figure}

Figure~\ref{fig:numerical-er-nonbacktracking-spectra} uses
one graph from the same model with $N=800$, $d=\log800$,
and $p=d/(N-1)$.
The panels use the same three upper cutoffs as the edge experiment,
with $L=0$ throughout.
The nonbacktracking matrices are formed from the uncentered
adjacency, with the original normalization by $\sqrt d$ as in
\eqref{eqn:nonbacktracking-definition-hermitian}.
Every eigenvalue location, including zeros from the host completion,
is plotted.  Hollow markers indicate the repeated eigenvalues
in all three panels:
\begin{align}
 z=0,
 \qquad
 z=\pm d^{-1/2}.
\end{align}
Stars mark only the positive real-axis extreme in each panel.
The unit circle is shown as a reference.

\begin{figure}[H]
\centering
\includegraphics[width=0.99\linewidth]{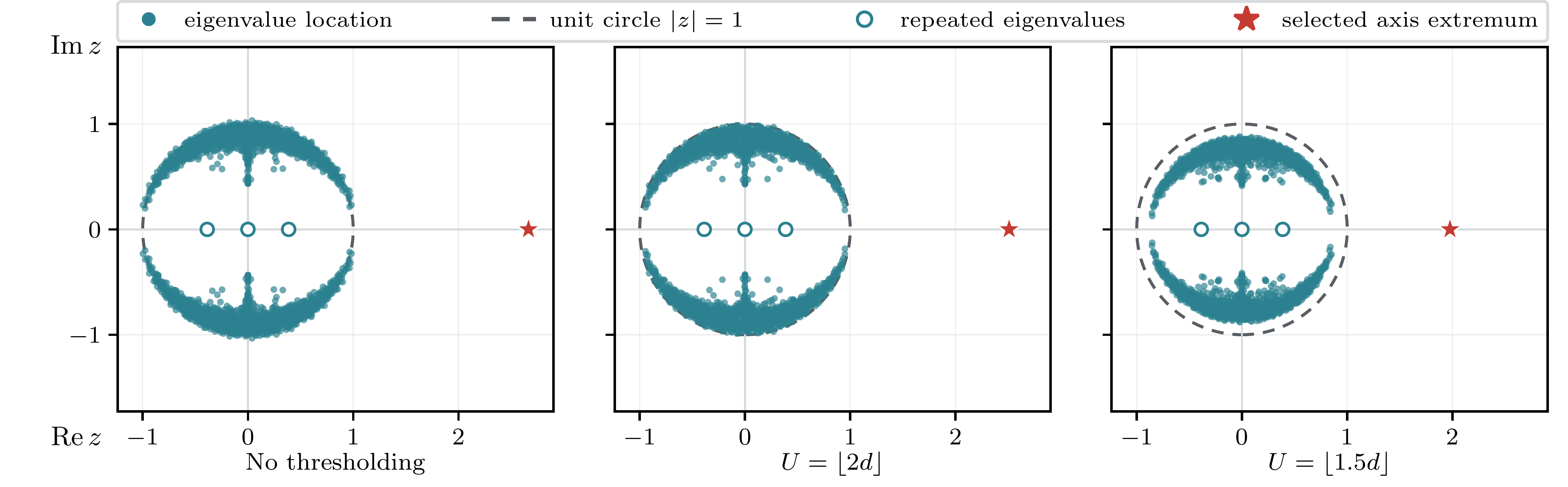}
\caption{Hermitian nonbacktracking spectra at $N=800$ under
three upper degree cutoffs.}
\label{fig:numerical-er-nonbacktracking-spectra}
\end{figure}

\subsection{Variance separation without a matching radius bound}
\label{subsec:fourier-radius-counterexample}

We illustrate Remark~\ref{rem:block-radius-obstruction} with a
deterministic family having $l_{\rm c}>u_{\rm r}$ and vanishing
entries, but whose nonbacktracking radius stays above the optimizer
$\sqrt{u_{\rm r}l_{\rm c}}$ and whose smallest singular value is below
$l_{\rm c}-u_{\rm r}$. For each $k$,
take $\mX\in\C^{16k\times2k}$ to be the tensor product of the
normalized $k\times k$ Fourier matrix with the $16\times2$ matrix
whose first column is constantly $1/\sqrt8$ and whose second column
has eight copies each of $(3+4\ii)/(5\sqrt8)$ and
$(3-4\ii)/(5\sqrt8)$.  The latter matrix has Gram matrix
$\left[\begin{smallmatrix}2&6/5\\6/5&2\end{smallmatrix}\right]$.
Hence the row and column squared norms are $u_{\rm r}^2=1/4$
and $l_{\rm c}^2=2$, respectively, while all entries have modulus
$(8k)^{-1/2}\to0$. The squared singular values are $4/5$ and $16/5$,
so variance separation holds but
\[
 \sigma_{\min}(\mX)=\sqrt{4/5}
 <l_{\rm c}-u_{\rm r}=\sqrt2-1/2.
\]

By Lemma~\ref{lem:bipartite-ihara-bass}, the nonpole eigenvalues
of $\widetilde{\mB}$ are the singular parameters of the vertex
matrix in \eqref{eqn:block-ihara-entries}.  Singular-value
decomposition therefore gives either
$\lambda^2=-1/4+(8k)^{-1}$ or
\[
 \bigl(\lambda^2+1/4-(8k)^{-1}\bigr)
 \bigl(\lambda^2+2-(8k)^{-1}\bigr)
 -\lambda^2\sigma^2=0,
 \qquad \sigma^2\in\{4/5,16/5\}.
\]
For $\sigma^2=4/5$, the larger-modulus negative root in $\lambda^2$ is
\[
 \lambda^2=\frac1{8k}-\frac{29+\sqrt{41+160/k}}{40}.
\]
It gives nonpole eigenvalues, since poles have $\lambda^2=(8k)^{-1}>0$.
The two limiting polynomials are
$\lambda^4+(29/20)\lambda^2+1/2$ and
$\lambda^4-(19/20)\lambda^2+1/2$.
Comparing their roots, the remaining branch, and the poles yields
\[
 \sqrt{u_{\rm r}l_{\rm c}}=2^{-1/4}
 <\lim_{k\to\infty}\rho(\widetilde{\mB})
 =\sqrt{(29+\sqrt{41})/40}<1.
\]
Thus $\rho(\widetilde{\mB})<1$ holds for large $k$, while every
admissible $\beta>\rho(\widetilde{\mB})$ stays separated from
the optimizer $\sqrt{u_{\rm r}l_{\rm c}}$.
Variance separation therefore needs a matching radius bound to
yield the edge $l_{\rm c}-u_{\rm r}$ through the comparison.
The column variances are already equal, so rescaling them to their
minimum leaves this example unchanged. Multiplying by an independent
common Rademacher sign centers the entries without changing these
quantities, but does not give the independence of the graph model.

%%%%%%%%%%%%%%%%%%%%%%%%%%%%%%%%%%%%%%%%%%%%%%%%%%%%%%%%%%%%%%%%%%%%%%
%%%%%%%%%%%%%%%%%  Degree Cutoffs and Vertex Retention %%%%%%%%%%%%%%%
%%%%%%%%%%%%%%%%%%%%%%%%%%%%%%%%%%%%%%%%%%%%%%%%%%%%%%%%%%%%%%%%%%%%%%

\section{Degree Cutoffs and Vertex Retention}
\label{sec:regularization-thresholds}

We state the cutoff regimes used in the spectral estimates and bound
the number of vertices removed by
Algorithm~\ref{alg:row-column-projection}. The common setup below
specifies the sparsity ranges, cutoff margins, and deletion bound.
The two subsections then treat the bipartite and Hermitian models,
including the degree-tail estimates used later in the moment arguments.
All model notation and normalizations are those of
Section~\ref{sec:introduction}. We use the concentration inequalities
collected in Section~\ref{sec:technical-lemmas}.
For the Hermitian model, write $\underline d\coloneqq\min_{u\in[N]}\mu_u$
for the minimum expected degree, which enters the lower-cutoff margins.

\begin{assumption}\label{assumption:cutoff-sparsity-range}
Suppose that the following holds for some fixed constants $c_{\rm deg}>0$ and $0<\delta_0<1$
\begin{subequations}
\begin{align}
 c_{\rm deg}\log N\leq d, & \qquad &
 d^{\,6}(\log N)^3\leq N^{1-\delta_0},
 \label{eqn:bipartite-cutoff-sparsity-range}\\
 c_{\rm deg}\log N\leq d, & \qquad &
 d^{\,6}(\log N)^3\leq N^{1-\delta_0}.
 \label{eqn:hermitian-cutoff-sparsity-range}
\end{align}
\end{subequations}
\end{assumption}
These ranges are used for the retention estimates in this section and
the smallest singular-value result in Section~\ref{sec:introduction}.
The largest singular-value and Hermitian eigenvalue bounds require
only the weaker sparsity conditions stated locally in their subsections,
as do parts~\ref{item:nonbacktracking-bipartite}
and~\ref{item:nonbacktracking-hermitian}
of Proposition~\ref{prop:nonbacktracking-radius-bounds}.
Recall that $\Gamma \geq 1$ is the fixed upper bound for $\ratio$ in the Assumption \ref{assumption:proportional-dimension-regime}. For the cutoff estimates, we choose
\begin{align}
 C_0=
 \begin{cases}
  \max\{1,C_{\rm sp}(\sqrt\Gamma+\Gamma^{-1/2})\},
       &\text{in the bipartite model},\\
  \max\{1,C_{\rm sp}\},&\text{in the Hermitian model}.
 \end{cases}
 \label{eqn:cutoff-profile-constant}
\end{align}
where $C_{\rm sp}$ is the constant in Assumptions~ \ref{assumption:rectangular-uniform-sparsity} and \ref{ass:uniform-Hermitian-sparsity}, respectively. In particular, the expected degree is at most $C_0d$ or $C_0d$ for every vertex in the bipartite model or Hermitian model, respectively.

\paragraph{Cutoff regimes.} We consider two different cutoff regimes. In the first regime, only upper cutoffs are performed while the lower thresholds are zero. In the second regime, both lower and upper cutoffs are performed.
\begin{enumerate}[label=\textup{(\Alph*)},leftmargin=2.5em,itemsep=0.6em]
\item\label{item:cutoff-upper-only}
\emph{Upper cutoffs only.}  In the bipartite model,
$L_{\rm r}=L_{\rm c}=0$, and in the Hermitian model, $L=0$.
The upper margins in the bipartite model satisfy
\eqref{eqn:rectangular-edge-upper-cutoffs}.
In the Hermitian model, the corresponding condition is
\begin{align}
 U-d\geq5\sqrt{C_0d\log d}.
 \label{eqn:hermitian-cutoff-upper-only}
\end{align}
\item\label{item:cutoff-strong-band}
\emph{Strong bands.}  The nonvacuous margins in the bipartite model
satisfy
\begin{align}
 \underline{\drow}-L_{\rm r}&\geq w_N(d),
 &U_{\rm r}-\overline{\drow}&\geq w_N(d),\label{eqn:cutoff-strong-band-row}\\
 \underline{\dcol}-L_{\rm c}&\geq w_N(d),
 &U_{\rm c}-\overline{\dcol}&\geq w_N(d).
 \label{eqn:cutoff-strong-band-column}
\end{align}
In the Hermitian model, they satisfy
\begin{align}
 \underline d-L\geq w_N(d),\qquad U-d\geq w_N(d).
 \label{eqn:hermitian-cutoff-strong-band}
\end{align}
The deterministic functions $j_N(d)$, $r_N(d)$, and $w_N(d)$ are defined as
\begin{align}
j_N(d)&=\lceil\sqrt d\log N\rceil,&r_N(d)&=\log(2N)+j_N(d)\log(4d),\\
w_N(d)&=2j_N(d)+\sqrt{2C_0d\,r_N(d)}+\tfrac23r_N(d).      \label{eqn:cutoff-strong-width}
\end{align}
\end{enumerate}
By maintaining a gap between the lower cutoffs and the minimum expected degrees in \eqref{eqn:cutoff-strong-band-row}, \eqref{eqn:cutoff-strong-band-column}, and \eqref{eqn:hermitian-cutoff-strong-band}, we accommodate degree fluctuations and thus guarantee that both sets retained after cutoff are nonempty with high probability.

\paragraph{Cutoff margins and deletion bounds.}
We now derive theoretical estimates on the sizes of the active sets by quantifying the number of vertices removed by Algorithm~\ref{alg:row-column-projection}. These estimates rely on the separation between the cutoff values and the extremal expected degrees, which we measure using normalized margins defined for both the bipartite and Hermitian models:
\begin{align}
 \varepsilon_{\rm B}
 &\coloneqq \frac{1}{d}\min \left\{ d,\,\,
 \underline{\drow}-L_{\rm r},\,\,
 U_{\rm r}-\overline{\drow},\,\,
 \underline{\dcol}-L_{\rm c},\,\,
 U_{\rm c}-\overline{\dcol}\right\},\qquad
 &\varepsilon_{\rm H}
 &\coloneqq
 \frac{1}{d}\min\left\{\underline{d} - L,\,\, U - d\right\}.
 \label{eqn:model-retention-margins}
\end{align}
Suppose $d>0$, $0<\varepsilon\leq 1$, and $k$ is an integer with $2\leq k\leq N$ and $2eC_0k<\varepsilon N$. Define $q = \lceil k/2 \rceil$ and $c_0 = (16C_0 + 8/3)^{-1}$. Then, we define the deletion bound by
\begin{align}
 \mathfrak{D}_N(k;d,\varepsilon)
 \coloneqq \min\biggl\{1,\quad
 \exp\left[q\log\frac{2eN}{q}-c_0\varepsilon^2dq\right]
 +
 \exp\left[k\log\frac{eN}{k}
 -\frac{\varepsilon dk}{4}
       \log\frac{\varepsilon N}{2eC_0k}\right]\biggr\}.
 \label{eqn:cutoff-deletion-bound}
\end{align}
In particular, $\mathfrak{D}_N(k;d,\varepsilon)\to0$ as $N\to\infty$ whenever $k\ll \varepsilon N$ and $\log(N/k) \ll\varepsilon^2d$.
The two subsections below apply this deletion bound with
$(d,\varepsilon)=(d,\varepsilon_{\rm B})$ and
$(d,\varepsilon_{\rm H})$, respectively, and prove the resulting
lower bounds on the cardinalities of the active sets.

\subsection{Bipartite model}
\label{subsec:bipartite-cutoff-retention}

We first state the deletion and retention bounds, then derive the
degree estimates used in their proofs.

\begin{lemma}[Bipartite deletion bounds and retained dimensions]
\label{lem:bipartite-cutoff-deletion-bounds}
\label{item:deletion-bipartite}
Assume \eqref{eqn:rectangular-uniform-sparsity-bound} and
$\varepsilon_{\rm B}>0$.  For every integer $k$ in the
range of \eqref{eqn:cutoff-deletion-bound} with
$\varepsilon=\varepsilon_{\rm B}$,
\begin{align}
 \P\bigl((n-|I_\star|)+(m-|J_\star|)\geq k\bigr)
 &\leq\mathfrak D_N(k;d,\varepsilon_{\rm B}).
 \label{eqn:model-cutoff-deletion-tail}
\end{align}
In particular, with probability at least
$1-\mathfrak D_N(k;d,\varepsilon_{\rm B})$, both
$|I_\star|>n-k$ and $|J_\star|>m-k$.
Under \eqref{eqn:bipartite-cutoff-sparsity-range}, the respective cutoff
conditions give, for all sufficiently large $N$,
\begin{subequations}
\begin{align}
 \P\bigl((n-|I_\star|)+(m-|J_\star|)>Nd^{-2}\bigr)
 &\leq d^{-2}
 &&\text{under \ref{item:cutoff-upper-only}},
 \label{eqn:upper-only-deletion-fraction}\\
 \P\bigl(I_\star\ne[n]\ \text{or}\ J_\star\ne[m]\bigr)
 &\leq(4d)^{-j_N(d)}
 &&\text{under \ref{item:cutoff-strong-band}}.
 \label{eqn:strong-band-no-deletion}
\end{align}
\end{subequations}
Under the condition \eqref{eqn:proportional-block-regime}, either \ref{item:cutoff-upper-only} or \ref{item:cutoff-strong-band} implies
\begin{align}
 |I_\star|/n\longrightarrow1,\qquad |J_\star|/m\longrightarrow1,
 \qquad |I_\star|/|J_\star|=\ratio(1+o_\P(1)).
 \label{eqn:cutoff-active-dimensions}
\end{align}
The ratio is evaluated on $J_\star\ne\varnothing$, an event whose
probability tends to one because $|J_\star|/m\to1$.
\end{lemma}

\paragraph{Degree estimates.}
We record the degree estimates used for retention here and for
moment verification in Section~\ref{sec:model-moment-verification}.  Write $D_i^{\rm r}=\sum_j\widetilde{\emA}_{ij}$ and
$D_j^{\rm c}=\sum_i\widetilde{\emA}_{ij}$ for the initial degrees.
Using the support events in \eqref{eqn:row-column-support-events}, let
\begin{align}
 \widetilde\Omega
 &=\bigcap_i\sR_i(\widetilde{\mA})
   \cap\bigcap_j\sC_j(\widetilde{\mA}).
 \label{eqn:bipartite-verification-band-event}
\end{align}
For a deterministic set of edges $\sF\subseteq[n]\times[m]$ and a
binary array $\mZ=(\emZ_{ij})_{(i,j)\in\sF}\in\{0,1\}^{\sF}$,
let $\widetilde{\mA}^{\sF\leftarrow\mZ}$ denote the adjacency matrix
obtained by replacing the entries indexed by $\sF$ with their prescribed
values in $\mZ$.  More precisely,
\begin{align}
 \bigl(\widetilde{\mA}^{\sF\leftarrow\mZ}\bigr)_{ij}
 &\coloneqq
 \begin{cases}
  \emZ_{ij}, &(i,j)\in\sF,\\
  \widetilde{\emA}_{ij}, &(i,j)\notin\sF.
 \end{cases}
 \label{eqn:bipartite-forced-adjacency}
\end{align}
Thus $\emZ_{ij}=1$ forces an edge to be present and $\emZ_{ij}=0$
forces it to be absent.  Every unforced entry remains unchanged and
retains its original independent Bernoulli law.  The array $\mZ$ is
indexed only by $\sF$; it does not specify values for the other entries.

For example, take $n=3$, $m=2$, $\sF=\{(1,2),(3,1)\}$, and
prescribe $\emZ_{12}=0$, $\emZ_{31}=1$.  For the realization below,
\eqref{eqn:bipartite-forced-adjacency} gives
\begin{align}
 \widetilde{\mA}
 &=\begin{pmatrix}0&1\\1&0\\0&1\end{pmatrix},
 &\widetilde{\mA}^{\sF\leftarrow\mZ}
 &=\begin{pmatrix}0&0\\1&0\\1&1\end{pmatrix}.
\end{align}
Only the entries $(1,2)$ and $(3,1)$ have been replaced.

For integers $\ell\geq0$, define the largest degree-band failure
probability over all assignments to at most $2\ell$ edges by
\begin{align}
 \varepsilon_{\ell,{\rm B}}
 &\coloneqq\max_{\substack{\sF\subseteq[n]\times[m],\ |\sF|\leq2\ell\\
                            \mZ\in\{0,1\}^{\sF}}}
    \P\bigl(\widetilde{\mA}^{\sF\leftarrow\mZ}
                            \notin\widetilde\Omega\bigr).
 \label{eqn:bipartite-forced-band-probability}
\end{align}
For upper cutoffs, define the single-vertex tail envelope and the
correction used for joint degree events by
\begin{align}
 q_{\theta,{\rm B}}
 &\coloneqq\max\left\{
   \max_{i:U_{\rm r}<m}
     e^{-\theta(U_{\rm r}+1)+(e^\theta-1)\mu_i^{\rm r}},
   \max_{j:U_{\rm c}<n}
     e^{-\theta(U_{\rm c}+1)+(e^\theta-1)\mu_j^{\rm c}}
            \right\},\label{eqn:bipartite-q-theta-B}\\
 \alpha_{\ell,{\rm B}}
 &\coloneqq e^{2\theta}q_{\theta,{\rm B}}
   \exp\!\left[\tfrac12(e^\theta-1)^2p_\star(2\ell+1)\right],
 \qquad\theta>0.
 \label{eqn:bipartite-alpha-l-B}
\end{align}
Empty maxima are zero.  We use $x^0=1$ for $x\geq0$ in the word
estimates below.

\begin{lemma}[Bipartite degree tails]
\label{lem:bipartite-degree-tail-control}
The initial degrees and forced-band probabilities satisfy the following
bounds.
\begin{enumerate}[label=\textup{(\roman*)},leftmargin=2.2em,itemsep=0.6em]
\item For every $\theta>0$ and every row $i$ and column $j$,
\begin{align}
 \P(D_i^{\rm r}>U_{\rm r})&\leq q_{\theta,{\rm B}},
 &\P(D_j^{\rm c}>U_{\rm c})&\leq q_{\theta,{\rm B}}.
 \label{eqn:bipartite-upper-cap-tail-probabilities}
\end{align}
\item If every nonvacuous row and column margin is at least $2\ell+h$,
where $\ell\geq0$ is an integer and $h>0$, then
\begin{align}
 \varepsilon_{\ell,{\rm B}}
 &\leq\min\left\{1,2N\exp\!\left[
 -\frac{h^2}{2(\max\{\overline{\drow},\overline{\dcol}\}+h/3)}
                             \right]\right\}.
 \label{eqn:bipartite-band-margin-tail}
\end{align}
\end{enumerate}
\end{lemma}

\begin{proof}[Proof of Lemma~\ref{lem:bipartite-degree-tail-control}]
Apply the exponential form of Lemma~\ref{lem:Chernoff} to each
initial degree with threshold $U_{\rm r}+1$ or $U_{\rm c}+1$.
Full caps cannot be exceeded, so this gives the first assertion.

Forcing at most $2\ell$ edges changes every degree by at most
$2\ell$. Under the stated margins, any forced-band violation thus
requires an original degree to deviate from its mean by more than $h$.
Apply Lemma~\ref{lem:Bennett}, in the form
\eqref{eqn:bennett-quadratic-tail}, with $K=1$ and variance budget
$V=\max\{\overline{\drow},\overline{\dcol}\}$.
A union bound over at most $2N$ tails, uniform in the forced entries,
proves \eqref{eqn:bipartite-band-margin-tail}.
\end{proof}

\paragraph{Deletion and retention proofs.}
A lower cutoff can cause further deletions after the first round.
The following argument controls these losses throughout the algorithm.

\begin{proof}[Proof of Lemma~\ref{lem:bipartite-cutoff-deletion-bounds}]
A deletion either starts from an atypical initial degree or is caused by
neighbors lost in earlier rounds.  We bound these two possibilities
separately, using the two terms of
$\mathfrak D_N(k;d,\varepsilon_{\rm B})$.
Fix an admissible $k$ and choose $q=\lceil k/2\rceil$.
Let the buffered initial violations be
\begin{align}
 Q={}&\bigl(\{i:L_{\rm r}>0,\ D_i^{\rm r}
                  <L_{\rm r}+\varepsilon_{\rm B}d/2\}
          \cup\{i:U_{\rm r}<m,\ D_i^{\rm r}>U_{\rm r}\}\bigr)
 \nonumber\\[-2pt]
 &\sqcup\bigl(\{j:L_{\rm c}>0,\ D_j^{\rm c}
                  <L_{\rm c}+\varepsilon_{\rm B}d/2\}
          \cup\{j:U_{\rm c}<n,\ D_j^{\rm c}>U_{\rm c}\}\bigr).
 \label{eqn:retention-buffered-seeds}
\end{align}
Every vertex of $Q$ has a signed degree deviation of at least
$\varepsilon_{\rm B}d/2$.  For a fixed set of $q$ vertices,
choose signs $\sigma_i,\tau_j\in\{-1,0,1\}$, nonzero exactly on that
set.  Their signed degree sum is
\[
 \sum_i\sigma_i(D_i^{\rm r}-\mu_i^{\rm r})
 +\sum_j\tau_j(D_j^{\rm c}-\mu_j^{\rm c})
 =\sum_{i,j}(\sigma_i+\tau_j)(\widetilde{\emA}_{ij}-p_{ij}).
\]
The summands are independent, bounded by two, and have total
variance at most
\[
 \sum_{i,j}(\sigma_i+\tau_j)^2p_{ij}(1-p_{ij})
 \leq2\sum_i\sigma_i^2\mu_i^{\rm r}
       +2\sum_j\tau_j^2\mu_j^{\rm c}
 \leq2C_0d q.
\]
Lemma~\ref{lem:Bennett}, with $K=2$, $V=2C_0d q$,
and $t=\varepsilon_{\rm B}d q/2$, bounds this signed
upper tail by $e^{-c_0\varepsilon_{\rm B}^2d q}$.
The union over at most $(2eN/q)^q$ signed sets gives the first term
of \eqref{eqn:cutoff-deletion-bound}.

Now suppose that at least $k$ vertices are deleted and $|Q|<q$.
List $Q$ first, including any vertices of $Q$ that survive, followed by
all deleted vertices outside $Q$ in round order.  Among the first $k$
vertices, at least $k-q+1\geq k/2$ lie outside $Q$.  Each could only be
deleted through its lower cutoff, having lost more than
$\varepsilon_{\rm B}d/2$ neighbors in strictly earlier
rounds.  Charge each lost edge to its later endpoint.  The first $k$
vertices thus induce at least $\varepsilon_{\rm B}d k/4$
edges, with no edge charged twice.  Edges between vertices deleted in
the same round are never counted as earlier losses.

For a fixed $k$-vertex set $S$, the induced-edge count is a Bernoulli
sum with mean at most $C_0d k^2/(2N)$.
Since $2eC_0k<\varepsilon_{\rm B}N$, its mean is smaller than
$t=\varepsilon_{\rm B}d k/4$. Lemma~\ref{lem:Chernoff}
therefore gives
\[
 \P\bigl(e(S)\geq\varepsilon_{\rm B}d k/4\bigr)
 \leq\left(\frac{2eC_0k}{\varepsilon_{\rm B}N}
                                        \right)^{\varepsilon_{\rm B}d k/4}.
\]
A union over $\binom Nk\leq(eN/k)^k$ sets proves the second term.
Combining the two events and bounding their probability by one gives
\eqref{eqn:model-cutoff-deletion-tail}.  If the total loss is less
than $k$, each part loses fewer than $k$ vertices, which proves the
stated finite bounds on $|I_\star|$ and $|J_\star|$.

For upper cutoffs alone, the first deletion round is terminal because
degrees only decrease.  Choose
$\theta=\sqrt{\log d/(C_0d)}$.
Using \eqref{eqn:rectangular-edge-upper-cutoffs} in the exponential bound from
Lemma~\ref{lem:Chernoff}, with $e^\theta-1-\theta\leq\theta^2$, gives
\begin{align}
 q_{\theta,{\rm B}}
 &\leq\exp\!\left[-5\theta\sqrt{C_0d\log d}
                          +C_0d\theta^2\right]
 =d^{-4}.
 \label{eqn:upper-only-retention-tail}
\end{align}
The expected total deletion is at most $Nd^{-4}$;
Markov's inequality proves \eqref{eqn:upper-only-deletion-fraction}.
For strong bands, apply \eqref{eqn:bipartite-band-margin-tail} with
no forced edges and
$h=\sqrt{2C_0d\,r_N(d)}+\tfrac23r_N(d)$.
The failure probability is at most
$2Ne^{-r_N(d)}=(4d)^{-j_N(d)}$;
outside this event every initial degree meets its band, so no vertex
is deleted.  This proves \eqref{eqn:strong-band-no-deletion}.
In either cutoff regime the total deletion is $o_\P(N)$.
Since $n$ and $m$ are proportional to $N$, dividing the two losses by
their dimensions proves \eqref{eqn:cutoff-active-dimensions}.
\end{proof}

For the smallest singular-value theorem, we also need nonempty active sets
when the lower column cutoff is close to the minimum expected degree.
The next lemma uses $L_{\rm r}=0$ to control the surviving columns
without requiring a positive lower-degree margin.

\begin{lemma}[Retention below the minimum expected column degree]
\label{lem:rectangular-column-retention}
Assume \eqref{eqn:proportional-block-regime},
\eqref{eqn:rectangular-uniform-sparsity-bound},
\eqref{eqn:bipartite-cutoff-sparsity-range}, and
\eqref{eqn:rectangular-edge-upper-cutoffs}, with $L_{\rm r}=0$ and
$1\leq L_{\rm c}\leq\underline{\dcol}$.
There are constants $c,C>0$, depending only on the fixed hypothesis
constants, such that, for all sufficiently large $N$,
\begin{align}
 \P\bigl(I_\star=\varnothing\ \text{or}\ |J_\star|<cm\bigr)
 &\leq e^{-cm}+Cd^{-3}.
 \label{eqn:rectangular-column-retention}
\end{align}
\end{lemma}

\begin{proof}[Proof of Lemma~\ref{lem:rectangular-column-retention}]
Let $D_i^{\rm r},D_j^{\rm c}$ be the original degrees.  We first find many columns whose original degrees meet both cutoffs, and then bound how many of them lose a neighbor to an upper row cutoff.

The profile assumptions give $p_\star=o(1)$ and
$\underline{\dcol}\geq d/C_0$.  Hence
$v_j\coloneqq\Var(D_j^{\rm c})\geq(1-p_\star)\mu_j^{\rm c}\geq1$
for large $N$.  For $X_j=D_j^{\rm c}-\mu_j^{\rm c}$, independence
of the Bernoulli summands gives
\begin{align}
 \E X_j^4\leq3v_j^2+v_j\leq4v_j^2.
\end{align}
H\"older's inequality and $\E X_j=0$ imply
\begin{align}
 \E|X_j|\geq\frac{v_j^{3/2}}{(\E X_j^4)^{1/2}}
 \geq\frac{\sqrt{v_j}}2,
 \qquad
 \E(X_j)_+=\tfrac12\E|X_j|\geq\frac{\sqrt{v_j}}4.
\end{align}
By Cauchy--Schwarz,
$\E(X_j)_+\leq\sqrt{v_j\P(X_j>0)}$, so
$\P(D_j^{\rm c}>\mu_j^{\rm c})\geq1/16$.
Since $L_{\rm c}\leq\mu_j^{\rm c}$ and
\eqref{eqn:upper-only-retention-tail} bounds the upper violation by
$d^{-4}$, each column meets its original band with probability at
least $1/32$ for large $N$.
The column degrees are independent. If $W$ counts these columns,
then $\E W\geq m/32$. The lower-tail form of
Lemma~\ref{lem:Chernoff}, with $\delta=1/2$, gives
\begin{align}
 \P(W<m/64)\leq e^{-m/256}.
\end{align}

Let $R$ count all edges incident to rows removed by their upper
cutoff.  If $U_{\rm r}=m$, then $R=0$.  Otherwise, with
$\theta=\sqrt{\log d/(C_0d)}$, condition on
$\widetilde{\emA}_{ij}=1$ and apply Lemma~\ref{lem:Chernoff}
to the remaining independent row entries to obtain
\begin{align}
 \E\bigl[\widetilde{\emA}_{ij}
                 \indi{D_i^{\rm r}>U_{\rm r}}\bigr]
 &\leq p_{ij}\exp\{-\theta U_{\rm r}
                      +(e^\theta-1)(\mu_i^{\rm r}-p_{ij})\}
 \leq e^\theta q_{\theta,{\rm B}}p_{ij}.
\end{align}
Since $q_{\theta,{\rm B}}\leq d^{-4}$ and
$\sum_{i,j}p_{ij}\leq C_0dm$, this yields
$\E R\leq eC_0md^{-3}$ for large $N$.  Markov's inequality gives
\begin{align}
 \P(R\geq m/128)\leq128eC_0d^{-3}.
\end{align}
At most $R$ of the $W$ candidate columns have a neighbor in a removed
row.  Every other candidate retains all its neighbors, so its degree
still lies in $[L_{\rm c},U_{\rm c}]$.  Since $L_{\rm r}=0$, retained
rows cannot be deleted after column removal, and these columns survive
every subsequent round.  On $W\geq m/64$ and $R<m/128$ we therefore
have $|J_\star|\geq W-R>m/128$.
Because $L_{\rm c}\geq1$, a surviving column also ensures
$I_\star\ne\varnothing$.
A union bound proves the assertion, for example with $c=1/256$
and $C=128eC_0$.
\end{proof}

\paragraph{Retained variances.}
For $x\in\{0,1\}$ and $0\leq p\leq p_\star$,
\begin{align}
 (1-p_\star)^2x\leq(x-p)^2\leq x+p^2.
 \label{eqn:bernoulli-variance-comparison}
\end{align}
Summing these inequalities along active rows and columns, and using
their terminal degree bounds, gives
\begin{align}
 \sum_{j\in J_\star}|\emX^{(\tau)}_{ij}|^2
 &\leq\frac{U_{\rm r}+\sum_jp_{ij}^2}{d},
 &\frac{(1-p_\star)^2L_{\rm c}}{d}
 &\leq\sum_{i\in I_\star}|\emX^{(\tau)}_{ij}|^2
 \leq\frac{U_{\rm c}+\sum_ip_{ij}^2}{d}
 \label{eqn:cutoff-active-column-variances}
\end{align}
for every $i\in I_\star$ and $j\in J_\star$.

\subsection{Hermitian model}
\label{subsec:hermitian-cutoff-retention}

The symmetric model has one active vertex set and a common degree
band. Its retention bounds follow from the same two mechanisms:
initial degree deviations and losses of neighbors in earlier rounds.

\begin{lemma}[Hermitian deletion bounds and retained dimension]
\label{lem:hermitian-cutoff-deletion-bounds}
\label{item:deletion-hermitian}
Assume \eqref{eqn:uniform-sparsity-bound} and
$\varepsilon_{\rm H}>0$.  For every integer $k$ in the
range of \eqref{eqn:cutoff-deletion-bound} with
$\varepsilon=\varepsilon_{\rm H}$,
\begin{align}
 \P\bigl(N-|I_\star|\geq k\bigr)
 &\leq\mathfrak D_N(k;d,\varepsilon_{\rm H}).
 \label{eqn:hermitian-cutoff-deletion-tail}
\end{align}
Thus $|I_\star|>N-k$ with probability at least
$1-\mathfrak D_N(k;d,\varepsilon_{\rm H})$.
Under \eqref{eqn:hermitian-cutoff-sparsity-range}, for all sufficiently large $N$,
\begin{subequations}
\begin{align}
 \P\bigl(N-|I_\star|>Nd^{-2}\bigr)&\leq d^{-2}
 &&\text{under \ref{item:cutoff-upper-only}},
 \label{eqn:Hermitian-upper-only-deletion-fraction}\\
 \P\bigl(I_\star\ne[N]\bigr)&\leq(4d)^{-j_N(d)}
 &&\text{under \ref{item:cutoff-strong-band}}.
 \label{eqn:Hermitian-strong-band-no-deletion}
\end{align}
\end{subequations}
In either case $|I_\star|/N\to1$.
\end{lemma}

\paragraph{Degree estimates.}
For the probability estimates, write $D_u=\sum_v\emA_{uv}$ for the
original degree.  For the common band $[L,U]$, define the event
that every original degree lies in the band by
\begin{align}
 \Omega\coloneqq\Omega(\mA;L,U)
 &=\bigcap_{u\in[N]}\{L\leq D_u\leq U\}
  =\bigcap_u\sR_u(\mA)=\bigcap_u\sC_u(\mA).
 \label{eqn:hermitian-support-event}
\end{align}
On the event in \eqref{eqn:hermitian-support-event}, the projection
algorithm removes no vertices.
We also need to control its failure probability after fixing a small
number of edges.  The forcing operation replaces both symmetric
entries corresponding to each unordered edge in $\sF\subseteq\binom{[N]}2$, using the same
prescribed value from $\mZ\in\{0,1\}^{\sF}$ for both entries.
For an integer $\ell\geq0$, define
\begin{align}
 \varepsilon_{\ell,{\rm H}}
 &\coloneqq\max_{\substack{\sF\subseteq\binom{[N]}2,\ |\sF|\leq2\ell\\
                    \mZ\in\{0,1\}^{\sF}}}
       \P(\mA^{\sF\leftarrow\mZ}\notin\Omega).
 \label{eqn:hermitian-forced-band-probability}
\end{align}
The upper-tail parameters are
\begin{align}
 q_{\theta,{\rm H}}
 &\coloneqq\max_{u:U<N-1}
    \exp\{-\theta(U+1)+(e^\theta-1)\mu_u\},\nonumber\\
 \alpha_{\ell,{\rm H}}
 &\coloneqq e^{2\theta}q_{\theta,{\rm H}}
    \exp\left[\tfrac12(e^\theta-1)^2p_\star(2\ell+1)\right],
 \qquad\theta>0,
 \label{eqn:hermitian-upper-cap-parameters}
\end{align}
with an empty maximum equal to zero.

\begin{lemma}[Hermitian degree tails]
\label{lem:hermitian-degree-tail-control}
The initial degrees and forced-band probabilities satisfy the following
bounds.
\begin{enumerate}[label=\textup{(\roman*)},leftmargin=2.2em,itemsep=0.6em]
\item For every $\theta>0$ and every vertex $u$,
\begin{align}
 \P(D_u>U)&\leq q_{\theta,{\rm H}}.
 \label{eqn:hermitian-verification-cap-probability}
\end{align}
\item If every nonvacuous margin is at least $2\ell+h$, where
$\ell\geq0$ is an integer and $h>0$, then
\begin{align}
 \varepsilon_{\ell,{\rm H}}
 &\leq\min\left\{1,2N\exp\left[-\frac{h^2}{2(d+h/3)}\right]\right\}.
 \label{eqn:hermitian-verification-margin-tail}
\end{align}
\end{enumerate}
\end{lemma}

\begin{proof}[Proof of Lemma~\ref{lem:hermitian-degree-tail-control}]
For each vertex, Lemma~\ref{lem:Chernoff} applied to its independent
incident edges gives the first assertion. Forcing at most $2\ell$
unordered edges changes every degree by at most $2\ell$.
Thus the stated margins reduce each forced-band violation to an
original degree deviation of more than $h$. Lemma~\ref{lem:Bennett},
in the form \eqref{eqn:bennett-quadratic-tail} with $K=1$ and $V=d$,
and a union bound over at most $2N$ tails give the second assertion.
\end{proof}

\paragraph{Deletion and retention proof.}
\begin{proof}[Proof of Lemma~\ref{lem:hermitian-cutoff-deletion-bounds}]
Use the seed-and-cascade argument of
Lemma~\ref{lem:bipartite-cutoff-deletion-bounds}, with
$q=\lceil k/2\rceil$ and
\[
 Q=\{u:L>0,\ D_u<L+\varepsilon_{\rm H}d/2\}
   \cup\{u:U<N-1,\ D_u>U\}.
\]
For signs supported on $q$ vertices, the independent-edge sum is
\begin{align}
 \sum_u\sigma_u(D_u-\mu_u)
 &=\sum_{u<v}(\sigma_u+\sigma_v)(\emA_{uv}-p_{uv}),\nonumber\\
 \Var\!\left(\sum_u\sigma_uD_u\right)
 &\leq2\sum_u\sigma_u^2\mu_u\leq2C_0dq.\nonumber
\end{align}
Thus Lemma~\ref{lem:Bennett}, with $K=2$ and $V=2C_0dq$,
and the same union over signed sets give the first term of
$\mathfrak D_N(k;d,\varepsilon_{\rm H})$.
If $|Q|<q$ and at least $k$ vertices are deleted, the same ordering
of deletions produces a $k$-vertex set spanning at least
$\varepsilon_{\rm H}dk/4$ edges. Its expected edge count is at most
$p_\star\binom{k}{2}\leq C_0dk^2/(2N)$, so
Lemma~\ref{lem:Chernoff} and a union over these sets give the second
term. This proves \eqref{eqn:hermitian-cutoff-deletion-tail} and the
finite retained-dimension bound.

For upper cutoffs, substitute $\theta=\sqrt{\log d/(C_0d)}$ into
\eqref{eqn:hermitian-upper-cap-parameters}. The margin
\eqref{eqn:hermitian-cutoff-upper-only} gives, exactly as in
\eqref{eqn:upper-only-retention-tail},
\begin{align}
 q_{\theta,{\rm H}}\leq d^{-4}.
 \label{eqn:hermitian-upper-only-retention-tail}
\end{align}
Only the first round deletes vertices, hence
$\E(N-|I_\star|)\leq Nd^{-4}$; Markov's inequality proves
\eqref{eqn:Hermitian-upper-only-deletion-fraction}.
For strong bands, apply \eqref{eqn:hermitian-verification-margin-tail}
with $\ell=0$ and
$h=\sqrt{2C_0d\,r_N(d)}+\tfrac23r_N(d)$.
The failure probability is at most
$2Ne^{-r_N(d)}=(4d)^{-j_N(d)}$, proving
\eqref{eqn:Hermitian-strong-band-no-deletion}.
Both estimates imply $|I_\star|/N\to1$.
\end{proof}

\paragraph{Retained variances.}
For every active row $u\in I_\star$, summing the upper inequality in
\eqref{eqn:bernoulli-variance-comparison} and using the terminal degree
bound gives
\begin{align}
 \sum_{v\in I_\star}|\emH^{(\tau)}_{uv}|^2
 &\leq\frac{U+\sum_vp_{uv}^2}{d}.
 \label{eqn:cutoff-active-hermitian-variances}
\end{align}

\section{Verification of Graph Moment Conditions}
\label{sec:model-moment-verification}

We verify the graph moment conditions used in
Lemma~\ref{lem:uniform-trace-consequence} and their extension to
Hermitian degree cutoffs. After recording scalar
moment bounds and a common weighted label sum, we treat bipartite
degree cutoffs, column rescaling, and Hermitian degree
cutoffs. Each verification identifies the
moment parameters and the range of powers for which they hold.
The degree-tail estimates are those of
Section~\ref{sec:regularization-thresholds}. Expectations remain under
the original law $\P$, with the random projections and column factors
inside the word products.

\Needspace{7\baselineskip}
\subsection{Common tools}
\label{subsec:graph-moment-common-tools}

For $B\sim\Ber(p)$, $|B-p|\leq1$ and
\begin{align}
 \E|B-p|&=2p(1-p),
 &\E|B-p|^k
 &=p(1-p)\bigl((1-p)^{k-1}+p^{k-1}\bigr)
 \leq p(1-p),\qquad k\geq2.
 \label{eqn:shared-bernoulli-centered-moments}
\end{align}

\begin{lemma}[Weighted label sums]
\label{lem:weighted-variance-label-sum}
Let $\xi$ be an admissible word and write
$g=g(\gG_\xi)$.
\begin{enumerate}[label=\textup{(\roman*)},leftmargin=2.2em,itemsep=0.6em]
\item In the bipartite setting, let $w_{ij}\geq0$, $N=n+m$, and
suppose the deterministic row and column budgets satisfy
\[
 \max_i\sum_jw_{ij}\leq\rowvarmax,\qquad
 \max_j\sum_iw_{ij}\leq\colvarmax,\qquad
 \rowvarmax\colvarmax\leq1.
\]
Here $\rowvarmax,\colvarmax$ denote budgets for these weights.
Set
\[
 M=\max\{1,\rowvarmax,\colvarmax\},\qquad
 \chi=\max\{1,NM\max_{i,j}w_{ij}\}.
\]
For either first-splice color $j\in\{1,2\}$,
\[
 \sum_{\varphi}\widetilde c_j(\varphi(\xi))
       \prod_{e\in\gE(\gG_\xi)}w_{\varphi(e)}
 \leq Mnm\left(\frac{\chi}{N}\right)^g,
\]
where the sum is over injective part-preserving maps into the
ambient vertex sets and $w_{\varphi(e)}$ uses the row--column
ordering of the edge endpoints.
\item In the Hermitian setting, let $w_{uv}=w_{vu}\geq0$,
$w_{uu}=0$, and $\max_u\sum_vw_{uv}\leq1$.
With $\chi=\max\{1,N\max_{u,v}w_{uv}\}$, one has
\[
 \sum_{\varphi}c(\varphi(\xi))
       \prod_{e\in\gE(\gG_\xi)}w_{\varphi(e)}
 \leq N^2\left(\frac{\chi}{N}\right)^g,
\]
where the sum is over injective maps into $[N]$.
\end{enumerate}
\end{lemma}

\begin{proof}
For \textup{(i)}, first take $\xi_0\in\gV_1$ and root a spanning
tree of $\gG_\xi$ at $\xi_0$.  There are $g$ nontree edges.
Bound their weights by $\max_{i,j}w_{ij}$ and the predecessor
multiplicity by $m$.  Dropping injectivity can only increase the
remaining nonnegative sum.  Eliminate tree leaves successively:
a child in $\gV_2$ costs at most $\rowvarmax$, and a child in
$\gV_1$ costs at most $\colvarmax$.  Summing the root over its
$n$ possible labels gives
\begin{align}
 \sum_{\varphi}\widetilde c_1(\varphi(\xi))
       \prod_{e\in\gE(\gG_\xi)}w_{\varphi(e)}
 &\leq nm\bigl(\max_{i,j}w_{ij}\bigr)^g
       \rowvarmax^{\abs{\gV_2(\gG_\xi)}}
       \colvarmax^{\abs{\gV_1(\gG_\xi)}-1}.
 \label{eqn:plural-bipartite-tree-sum}
\end{align}
Every leaf of the support graph is a splice vertex, since a visit
to any other leaf would force an immediate reversal.  There are
at most two splice vertices in $\gV_1$ and at most one in
$\gV_2$.  Counting degrees in each part therefore gives
\[
 \abs{\gE(\gG_\xi)}\geq2\abs{\gV_1(\gG_\xi)}-2,
 \qquad
 \abs{\gE(\gG_\xi)}\geq2\abs{\gV_2(\gG_\xi)}-1.
\]
Together with the definition of $g$, these inequalities imply
\begin{align}
 \left|\abs{\gV_2(\gG_\xi)}-\abs{\gV_1(\gG_\xi)}+1\right|
 &\leq g+1.
 \label{eqn:part-balance-roadmap}
\end{align}
Pairing row and column factors in
\eqref{eqn:plural-bipartite-tree-sum} costs at most one per pair
because $\rowvarmax\colvarmax\leq1$; each unpaired factor is at
most $M$.  By \eqref{eqn:part-balance-roadmap}, their product is
at most $M^{g+1}$.  The definition of $\chi$ proves the claim.
Interchanging the parts gives the same bound when
$\xi_0\in\gV_2$.

For \textup{(ii)}, root a spanning tree at $\xi_0$ and bound each
nontree edge weight by $\chi/N$.  The root label and predecessor
contribute at most $N^2$, and each leaf elimination costs at most
one.  This proves the Hermitian bound.
\end{proof}

\Needspace{7\baselineskip}
\subsection{Bipartite degree cutoffs}
\label{subsec:bipartite-moment-verification}

\paragraph{Entry moments.}
\begin{lemma}[Bipartite entry moment bounds]
\label{lem:projected-rectangular-sparse-bounds}
Under \eqref{eqn:rectangular-uniform-sparsity-bound}, the zero-padded
matrix $\widehat{\mX}$ satisfies Definition~\ref{def:sparse-random-matrix}
at scale $d$ with moment constant $2C_0$.
\end{lemma}

\begin{proof}
Projection decreases each absolute entry, so
$|\widehat{\emX}_{ij}|\leq d^{-1/2}$.
By \eqref{eqn:shared-bernoulli-centered-moments}, for every $k\geq1$,
\[
 \E|\widehat{\emX}_{ij}|^k
 \leq\frac{2p_{ij}}{d^{k/2}}
 \leq\frac{2C_0}{Nd^{(k-2)/2}}.
\]
\end{proof}

\paragraph{graph moment verification.}
The projections in $\widehat{\mX}=\mP_{\rm r}\mX\mP_{\rm c}$
depend on the sampled graph, so the projected entries need not be
independent or centered.  Plural words are controlled by absolute
moments and Lemma~\ref{lem:weighted-variance-label-sum}; singleton
words require additional degree-tail estimates.

Choose $R=(\rowvarmax\colvarmax)^{1/4}$. We verify
\ref{eqn:bipartite-graph-condition} for $\widehat{\mX}/R$
at scale $R^2d$. The weights
\[
 w_{ij}=\frac{p_{ij}(1-p_{ij})}{R^2d}
\]
have row and column budgets $\rowvarmax/R^2$ and
$\colvarmax/R^2$, whose product is one. Degree-tail estimates
continue to use the original scale $d$.

For general bands, the event $\widetilde\Omega$ in
\eqref{eqn:bipartite-verification-band-event} ensures that no deletion
starts.  The forced failure probability
$\varepsilon_{\ell,{\rm B}}$ in
\eqref{eqn:bipartite-forced-band-probability} controls its complement
after prescribing any at most $2\ell$ support edges.  We require
\begin{align}
 \varepsilon_{\ell,{\rm B}}\leq(4R^2d)^{-\ell}.
 \label{eqn:bipartite-forced-band-condition}
\end{align}
This gives the singleton-chain decay even with positive lower
cutoffs and deletion cascades, under the original graph law.

For upper-only cutoffs, survival is determined by the initial degree
checks: subsequent degrees can only decrease.  We instead use joint
violations at vertices incident to singleton edges, requiring
\begin{align}
 L_{\rm r}=L_{\rm c}=0,\qquad
 32R^2d\,\alpha_{\ell,{\rm B}}\leq1.
 \label{eqn:bipartite-upper-cap-condition}
\end{align}
Here $\alpha_{\ell,{\rm B}}$ is defined in
\eqref{eqn:bipartite-alpha-l-B}, with corrections for support edges
and edges shared by the degree events.  This criterion remains useful
when forcing $2\ell$ edges could itself violate a cap, making the
general-band failure probability equal to one.

In both cases, use the variance-label parameter
\begin{align}
 \widetilde\chi_\ell
 &\coloneqq\max\left\{1,
 \frac{N\max\{\rowvarmax,\colvarmax\}
                  \max_{i,j}p_{ij}(1-p_{ij})}
             {R^4d}\right\}.
 \label{eqn:plural-projected-bipartite-parameters}
\end{align}

\begin{lemma}[Bipartite graph moment verification]
\label{lem:bipartite-graph-verification}
Assume \eqref{eqn:rectangular-uniform-sparsity-bound},
$d\geq1$, $R=(\rowvarmax\colvarmax)^{1/4}$, and $R^2d\geq1$.
Let $\ell\geq3$ be an integer. For $\widehat{\mX}/R$ and its
block linearization, both lines of
\ref{eqn:bipartite-graph-condition} hold at sparsity scale
$R^2d$, uniformly over the canonical classes and both
first-splice colors, under either of the following conditions.
In both cases use $\widetilde\chi_\ell$ from
\eqref{eqn:plural-projected-bipartite-parameters}.
\begin{enumerate}[label=\textup{(\roman*)},leftmargin=2.2em,itemsep=0.6em]
\item\label{item:bipartite-graph-general-bands}
Under \eqref{eqn:bipartite-forced-band-condition}, take
\begin{align}
 \widetilde K_\ell
 &=\frac{\max\{\rowvarmax,\colvarmax\}}
              {R^2},
 &\widetilde\Gamma_\ell&=2.
 \label{eqn:bipartite-general-band-verification}
\end{align}
\item\label{item:bipartite-graph-upper-caps}
Fix $\theta>0$ and evaluate $q_{\theta,{\rm B}}$ and
$\alpha_{\ell,{\rm B}}$ at this same value of $\theta$ using
\eqref{eqn:bipartite-alpha-l-B}.  If
\eqref{eqn:bipartite-upper-cap-condition} holds for this choice, take
\begin{align}
 \widetilde K_\ell
 &=\frac{2\max\{\rowvarmax,\colvarmax\}}
              {R^2}
              e^{2\ell\alpha_{\ell,{\rm B}}},
 &\widetilde\Gamma_\ell&=4e^\theta.
 \label{eqn:bipartite-upper-cap-verification}
\end{align}
Here $\theta$ is the Chernoff exponential parameter in the
degree-tail bounds \eqref{eqn:bipartite-upper-cap-tail-probabilities}. For example, when $d>1$ and the upper margins satisfy \eqref{eqn:rectangular-edge-upper-cutoffs}, the choice $\theta=\sqrt{\log d/(C_0d)}$ gives $q_{\theta,{\rm B}}\leq d^{-4}$, as shown in \eqref{eqn:upper-only-retention-tail}.

\end{enumerate}
\end{lemma}

\begin{proof}[Proof of Lemma~\ref{lem:bipartite-graph-verification}]
Fix a canonical word $\xi$ and an injective, part-preserving
labeling $\varphi$ of its support.  Summing over these labelings
enumerates its relabeling class exactly once.  Write
\[
 W=(R^2d)^{-\ell}
   \prod_{e\in\gE(\gG_\xi)}
       (\widetilde{\emA}_{\varphi(e)}-p_{\varphi(e)})^{m_e(\xi)},
 \qquad
 S=\indi{\text{all vertices of }\varphi(\gG_\xi)\text{ survive}}.
\]
The projected word product is $SW$, since a projected entry is zero
unless both endpoints survive and $\sum_e m_e(\xi)=2\ell$.
A support edge of zero variance makes $W=0$ almost surely; below we
consider only positive support variances.  Independence factors
$\E|W|$ over the distinct support edges.  By
\eqref{eqn:shared-bernoulli-centered-moments}, each singleton edge
contributes $2p_{\varphi(e)}(1-p_{\varphi(e)})$, and each edge of
multiplicity at least two contributes at most
$p_{\varphi(e)}(1-p_{\varphi(e)})$.
There are $e_1(\xi)$ singleton edges, so, using $0\leq S\leq1$, we obtain
\begin{align}
      |\E[SW]| \leq \E[S|W|] \leq \E[|W|]
 &\leq2^{e_1(\xi)}(R^2d)^{-\ell}
       \prod_{e\in\gE(\gG_\xi)}
                p_{\varphi(e)}(1-p_{\varphi(e)}).
 \label{eqn:shared-absolute-word-bound}
\end{align}

\smallskip\noindent\emph{Plural words.}
Since $e_1(\xi)=0$ for a plural word, \eqref{eqn:shared-absolute-word-bound} simplifies to
\begin{align}
 |\E[SW]| \leq \E[S|W|] \leq \E[|W|]
 &\leq(R^2d)^{-\ell}
       \prod_{e\in\gE(\gG_\xi)}
                p_{\varphi(e)}(1-p_{\varphi(e)}).
 \label{eqn:bipartite-plural-word-bound}
\end{align}
Apply Lemma~\ref{lem:weighted-variance-label-sum} to
\[
 w_{ij}=\frac{p_{ij}(1-p_{ij})}{R^2d},
 \qquad
 \max_i\sum_jw_{ij}=\frac{\rowvarmax}{R^2},
 \qquad
 \max_j\sum_iw_{ij}=\frac{\colvarmax}{R^2}.
\]
The two budgets have product one, and their maximum is
$\max\{\rowvarmax,\colvarmax\}/R^2$.
Restoring the word normalization contributes
$(R^2d)^{\abs{\gE(\gG_\xi)}-\ell}$ and yields
\begin{align}
 \abs{\widetilde{\GraphMoment}_{\ell,i}(\WordClass{\xi})}
 &\leq
 \frac{\max\{\rowvarmax,\colvarmax\}}
      {R^2}\,nm
 \left(\frac{\widetilde\chi_\ell}{N}\right)^{\abs{g(\gG_\xi)}}
 (R^2d)^{\abs{\gE(\gG_\xi)}-\ell}.
 \label{eqn:plural-projected-bipartite-bound}
\end{align}
Interchanging the parts proves the same bound for the other color.
This label sum uses only admissibility and applies to the singleton
estimates below as well. The variance formulas and
\eqref{eqn:bipartite-normalization} give
\[
 R^2d\geq(1-p_\star)
       \sqrt{\overline{\drow}\,\overline{\dcol}}
 \geq(1-p_\star)d.
\]
The profile bound gives
$\max\{\rowvarmax,\colvarmax\}\leq C_0$ and
$p_\star\leq C_0d/N$. Hence, whenever $p_\star\leq1/2$,
\begin{align}
 1\leq
 \frac{\max\{\rowvarmax,\colvarmax\}}{R^2}
 &\leq2C_0,
 &1\leq\widetilde\chi_\ell&\leq4C_0^2.
 \label{eqn:bipartite-profile-constant-bounds}
\end{align}

\smallskip\noindent\emph{Singleton words under general bands.}
For singleton words, the degree estimates will provide
an additional factor. A singleton factor is centered and independent of every other word
factor, so $\E W=0$.  On the event $\widetilde\Omega$ in \eqref{eqn:bipartite-verification-band-event} there is no deletion. Consequently, $\widetilde\Omega\subseteq\{S=1\}$ and $1-S\le\mathbf1_{\widetilde\Omega^{\rm c}}$, thus
\begin{align}
 |\E[SW]|
 &=|\E[(S-1)W]|
 \leq\E[|W|\indi{\widetilde\Omega^{\rm c}}].
 \label{eqn:shared-singleton-cancellation}
\end{align}
To separate the absolute word size from the survival event, define
the tilted probability law
$\mathds{Q}_{\xi,\varphi}$ by
\begin{align}
 \frac{d\mathds{Q}_{\xi,\varphi}}{d\P}
 &\coloneqq\frac{|W|}{\E|W|}.
 \label{eqn:word-tilted-law}
\end{align}
The density factors over support edges; the other edges retain their
original independent laws. Thus
$\E[|W|\indi A]=\E|W|\,\mathds{Q}_{\xi,\varphi}(A)$ for every event $A$.

Condition on the support-edge values under $\mathds{Q}_{\xi,\varphi}$.
There are at most $2\ell$ such edges, and all remaining edges still
have their original independent laws.  For each choice of support
values, the conditional probability of $\widetilde\Omega^{\rm c}$
is therefore at most $\varepsilon_{\ell,{\rm B}}$, as defined in \eqref{eqn:bipartite-forced-band-probability}.
That definition maximizes over both the positions of the fixed edges
and their prescribed values, so the same bound applies to every
labeling $\varphi$.
Averaging over those values and using \eqref{eqn:shared-absolute-word-bound}
yields
\begin{align}
 |\E[SW]|
 &\leq\varepsilon_{\ell,{\rm B}}\E|W|
 \leq\varepsilon_{\ell,{\rm B}}2^{e_1(\xi)}
 (R^2d)^{-\ell}
 \prod_{e\in\gE(\gG_\xi)}p_{\varphi(e)}(1-p_{\varphi(e)}).
 \label{eqn:shared-forced-singleton-word-bound}
\end{align}
We next express the factor $2^{e_1(\xi)}$ in terms of the statistics
in the singleton clause of \ref{eqn:bipartite-graph-condition}.
A singleton chain of length $k$ contributes $k$ to $e_1(\xi)$ and
$\lfloor(k-1)/2\rfloor$ to $\varsigma(\xi)$, as defined in
\eqref{eqn:singleton-certificate-count}.  Since
$k\leq2\lfloor(k-1)/2\rfloor+2$, summing over singleton chains
gives at most $2\varsigma(\xi)$ plus twice their number.  By
\eqref{eqn:core-support-edge-count}, the number of all support chains
is $\abs{\gE(\gG_\xi)}-
\abs{\gS_{\{\xi_0,\xi_\ell\}}(\gG_\xi)}$, which also bounds the
number of singleton chains.  Moreover,
$\lfloor(k-1)/2\rfloor\leq k/2$ on each chain and
$e_1(\xi)\leq2\ell$.  Hence
\begin{align}
 e_1(\xi)
 &\leq2\varsigma(\xi)+2\left(
       \abs{\gE(\gG_\xi)}-
       \abs{\gS_{\{\xi_0,\xi_\ell\}}(\gG_\xi)}\right),
 &0\leq\varsigma(\xi)&\leq e_1(\xi)/2\leq\ell.
 \label{eqn:shared-singleton-chain-count}
\end{align}
Since $\ell-\varsigma(\xi)\geq0$ and $R^2d\geq1$, using \eqref{eqn:bipartite-forced-band-condition} and this chain count, we obtain
\begin{align}
 \varepsilon_{\ell,{\rm B}}2^{e_1(\xi)}
 &\leq
 (4R^2d)^{-\ell}4^{\varsigma(\xi)}
 2^{2\left(\abs{\gE(\gG_\xi)}-
       \abs{\gS_{\{\xi_0,\xi_\ell\}}(\gG_\xi)}\right)}
 \nonumber\\
 &\leq
 2^{2\left(\abs{\gE(\gG_\xi)}-
       \abs{\gS_{\{\xi_0,\xi_\ell\}}(\gG_\xi)}\right)}
 (R^2d)^{-\varsigma(\xi)}.
 \label{eqn:bipartite-forced-singleton-payment}
\end{align}

The factor in \eqref{eqn:bipartite-forced-singleton-payment} depends
only on the canonical word.  Apply
Lemma~\ref{lem:weighted-variance-label-sum} to the remaining variance
product in \eqref{eqn:shared-forced-singleton-word-bound}, exactly as
for plural words.  This gives the singleton clause of
\ref{eqn:bipartite-graph-condition} with the parameters in
\eqref{eqn:bipartite-general-band-verification} and proves
item~\ref{item:bipartite-graph-general-bands}.

\smallskip\noindent\emph{Singleton words under upper-only cutoffs.}
Assume \eqref{eqn:bipartite-upper-cap-condition} and fix a singleton
word $\xi$ and an injective, part-preserving labeling $\varphi$.
Since $L_{\rm r}=L_{\rm c}=0$, the algorithm removes precisely those
vertices whose original degrees exceed their upper caps.  Indeed,
after the first simultaneous deletion, the degrees of the retained
vertices can only decrease, so all remaining vertices still satisfy
their caps and the algorithm stops.  Thus survival can be expressed
directly in terms of the original degrees.

For each $v\in\varphi(\gV(\gG_\xi))$, let $B_v$ indicate that its
original degree exceeds $U_{\rm r}$ if $v\in\gV_1$, or $U_{\rm c}$
if $v\in\gV_2$.  At a full cap this indicator is identically zero.
Every vertex of the word survives exactly when all these indicators
vanish.  Expanding their product gives
\begin{align}
 S=\prod_{v\in\varphi(\gV(\gG_\xi))}(1-B_v)
  =\sum_{T\subseteq\varphi(\gV(\gG_\xi))}
       (-1)^{|T|}\prod_{v\in T}B_v.
 \label{eqn:shared-survival-expansion}
\end{align}
Consider any term $\E[W\prod_{v\in T}B_v]$ appearing in this expansion. If the set $T$ contains neither endpoint of a singleton edge, the corresponding centered Bernoulli variable is, under the original law $\P$, independent of both all remaining factors in $W$ and all $B_v$ with $v\in T$, thus its expectation vanishes and the entire term is zero.
Consequently, only sets $T$ that contain at least one endpoint of every singleton edge can contribute. These are the vertex covers of the singleton edges. Taking absolute values and using the tilted law in \eqref{eqn:word-tilted-law}, we obtain
\begin{align}
 |\E[SW]|
 &\leq
 \sum_{\substack{T\subseteq\varphi(\gV(\gG_\xi))\\
                  T\text{ covers the singleton edges}}}
 \E\left[|W|\prod_{v\in T}B_v\right]
 \nonumber\\
 &=\E|W|
 \sum_{\substack{T\subseteq\varphi(\gV(\gG_\xi))\\
                  T\text{ covers the singleton edges}}}
 \mathds{Q}_{\xi,\varphi}(B_v=1\text{ for every }v\in T).
\end{align}

It remains to bound the joint probability for a fixed set $T$.
If $T$ contains a vertex with a full upper cap, this probability is
zero because the corresponding $B_v$ vanishes identically.  We may
therefore assume in the following calculation that every vertex of
$T$ has a nonfull upper cap.
Let $\sF=\varphi(\gE(\gG_\xi))$, following the notation for fixed edges as in \eqref{eqn:bipartite-forced-adjacency}. We will consider an enlarged event. For each realization, replace all entries in $\sF$ by one and leave all other entries
unchanged. The probability of this enlarged event is at least as large as the original event, under the upper-cap only regime. After this replacement, the support edges contribute the deterministic quantity $\sum_{v\in T}\deg_{\varphi(\gG_\xi)}(v)$, and only the entries outside $\sF$ remain random. Since $W$ depends only on the support entries within $\sF$, it remains independent under $\P$ from any event $A$ that involves only entries outside of $\sF$. For such an event,
\[
 \mathds{Q}_{\xi,\varphi}(A)
 =\frac{\E[|W|\indi A]}{\E|W|}
 =\frac{\E|W|\,\P(A)}{\E|W|}
 =\P(A).
\]
Furthermore, if a vertex is removed, its degree is either at least $U_{\rm r}+1$ if in $\gV_{1}$, or at least $U_{\rm c}+1$ if in $\gV_{2}$. Thus by Handshake Lemma, per the event being true, the total number of edges in $T$ is at least $(U_{\rm r}+1)|T\cap\gV_1|+(U_{\rm c}+1)|T\cap\gV_2|$. We then have
\begin{align}
 &\mathds{Q}_{\xi,\varphi}(B_v=1, \, \forall v\in T)
 = \mathds{Q}_{\xi,\varphi}\left(
 \begin{gathered}
 D_i^{\rm r}\geq U_{\rm r}+1,\,\forall i\in T\cap\gV_1,\\
 D_j^{\rm c}\geq U_{\rm c}+1,\,\forall j\in T\cap\gV_2
 \end{gathered}\right)
 \nonumber\\
\leq & \P\left(
 \begin{aligned}
 &\sum_{v\in T}\deg_{\varphi(\gG_\xi)}(v)
  +\sum_{i\in T\cap\gV_1}\sum_{j:(i,j)\notin\sF}
       \widetilde{\emA}_{ij}
  +\sum_{j\in T\cap\gV_2}\sum_{i:(i,j)\notin\sF}
       \widetilde{\emA}_{ij}\\
 &\qquad\geq (U_{\rm r}+1)|T\cap\gV_1|
               +(U_{\rm c}+1)|T\cap\gV_2|
 \end{aligned}\right)
 \nonumber\\
=&\P\left(
 \begin{aligned}
 &\sum_{\substack{e\notin\sF\\
                   e\text{ has exactly one endpoint in }T}}
       \widetilde{\emA}_e
  +2\sum_{\substack{e\notin\sF\\
                   e\text{ has both endpoints in }T}}
       \widetilde{\emA}_e\\
 &\qquad\geq (U_{\rm r}+1)|T\cap\gV_1|
               +(U_{\rm c}+1)|T\cap\gV_2|
               -\sum_{v\in T}\deg_{\varphi(\gG_\xi)}(v)
 \end{aligned}\right).
\end{align}
For the same $\theta>0$ as in \eqref{eqn:bipartite-alpha-l-B}, Markov and independence give
\begin{align}
 &\mathds{Q}_{\xi,\varphi}(B_v=1\text{ for every }v\in T)
 \nonumber\\
 &\quad\leq
 \exp\!\left\{-\theta\bigl[(U_{\rm r}+1)|T\cap\gV_1|
                      +(U_{\rm c}+1)|T\cap\gV_2|\bigr]
       +\theta\sum_{v\in T}\deg_{\varphi(\gG_\xi)}(v)\right\}
 \nonumber\\
 &\qquad\times
 \E\exp\!\left\{
 \theta\sum_{\substack{e\notin\sF\\
                   e\text{ has exactly one endpoint in }T}}
       \widetilde{\emA}_e
 +2\theta\sum_{\substack{e\notin\sF\\
                   e\text{ has both endpoints in }T}}
       \widetilde{\emA}_e\right\}
 \nonumber\\
 &\quad=
 \exp\!\left\{-\theta\bigl[(U_{\rm r}+1)|T\cap\gV_1|
                      +(U_{\rm c}+1)|T\cap\gV_2|\bigr]
       +\theta\sum_{v\in T}\deg_{\varphi(\gG_\xi)}(v)\right\}
 \nonumber\\
 &\qquad\times
 \prod_{\substack{e\notin\sF\\
                   e\text{ has exactly one endpoint in }T}}
       \bigl[1+p_e(e^\theta-1)\bigr]
 \prod_{\substack{e\notin\sF\\
                   e\text{ has both endpoints in }T}}
       \bigl[1+p_e(e^{2\theta}-1)\bigr].
\end{align}
Apply $1+x\leq e^x$ to the products and separate the endpoint
contributions using
\[
 p_e(e^{2\theta}-1)
 =2p_e(e^\theta-1)+p_e(e^\theta-1)^2.
\]
The last term corrects for an edge shared by two selected vertices.
The resulting exponent satisfies
\begin{align}
 &\sum_{\substack{e\notin\sF\\
                   e\text{ has exactly one endpoint in }T}}
       p_e(e^\theta-1)
 +\sum_{\substack{e\notin\sF\\
                   e\text{ has both endpoints in }T}}
       p_e(e^{2\theta}-1)
 \nonumber\\
 &\quad=(e^\theta-1)
 \left(\sum_{i\in T\cap\gV_1}
             \sum_{j:(i,j)\notin\sF}p_{ij}
       +\sum_{j\in T\cap\gV_2}
             \sum_{i:(i,j)\notin\sF}p_{ij}\right)
 +(e^\theta-1)^2
 \sum_{\substack{e\notin\sF\\
                   e\text{ has both endpoints in }T}}p_e
 \nonumber\\
 &\quad\leq(e^\theta-1)
 \left(\sum_{i\in T\cap\gV_1}\mu_i^{\rm r}
       +\sum_{j\in T\cap\gV_2}\mu_j^{\rm c}\right)
 +\tfrac12(e^\theta-1)^2p_\star|T|^2.
\end{align}
Combining with the threshold prefactor gives one factor at most $q_{\theta,{\rm B}}$ per selected vertex, by \eqref{eqn:bipartite-q-theta-B}:
\begin{align}
 &\mathds{Q}_{\xi,\varphi}(B_v=1\text{ for every }v\in T)
 \nonumber\\
\leq &
 \prod_{i\in T\cap\gV_1}
       e^{-\theta(U_{\rm r}+1)+(e^\theta-1)\mu_i^{\rm r}}
 \prod_{j\in T\cap\gV_2}
       e^{-\theta(U_{\rm c}+1)+(e^\theta-1)\mu_j^{\rm c}}
\times
 \exp\left\{\theta\sum_{v\in T}\deg_{\varphi(\gG_\xi)}(v)
       +\tfrac12(e^\theta-1)^2p_\star|T|^2\right\}
 \nonumber\\
\leq & q_{\theta,{\rm B}}^{|T|}
 \exp\left\{\theta\sum_{v\in T}\deg_{\varphi(\gG_\xi)}(v)
       +\tfrac12(e^\theta-1)^2p_\star|T|^2\right\}.
 \label{eqn:shared-joint-upper-tail}
\end{align}
The final bound also holds for sets $T$ containing a vertex with a
full upper cap, since their joint probability is zero.  Thus it may
be used in the unrestricted sum over vertex covers above.
We now use the degree-two vertices inside the support chains.
Define the selected internal vertices by
\[
 T_{\rm int}
 =T\cap\varphi\bigl(\gS_{\{\xi_0,\xi_\ell\}}(\gG_\xi)\bigr).
\]
Each has support degree two.  The sum of the degrees of all other
support vertices is
$2(\abs{\gE(\gG_\xi)}-
\abs{\gS_{\{\xi_0,\xi_\ell\}}(\gG_\xi)})$, since the total
degree is $2\abs{\gE(\gG_\xi)}$.  It follows that
\[
 \sum_{v\in T}\deg_{\varphi(\gG_\xi)}(v)
 \leq2|T_{\rm int}|
 +2\left(\abs{\gE(\gG_\xi)}-
         \abs{\gS_{\{\xi_0,\xi_\ell\}}(\gG_\xi)}\right).
\]
Also $|T|\leq2\ell$, because the closed word has $2\ell$ edge
traversals.  Thus $|T|^2\leq(2\ell+1)|T|$, and the definition in \eqref{eqn:bipartite-q-theta-B} and \eqref{eqn:bipartite-alpha-l-B}  gives
\[
 q_{\theta,{\rm B}}^{|T|}
 \exp\!\left[\tfrac12(e^\theta-1)^2p_\star|T|^2\right]
 \leq\bigl(e^{-2\theta}\alpha_{\ell,{\rm B}}\bigr)^{|T|}.
\]
Substituting these bounds into \eqref{eqn:shared-joint-upper-tail}
yields
\begin{align}
 \mathds{Q}_{\xi,\varphi}(B_v=1\text{ for every }v\in T)
 &\leq
 e^{2\theta\left(\abs{\gE(\gG_\xi)}-
        \abs{\gS_{\{\xi_0,\xi_\ell\}}(\gG_\xi)}\right)}
 \alpha_{\ell,{\rm B}}^{|T_{\rm int}|}
 \bigl(e^{-2\theta}\alpha_{\ell,{\rm B}}\bigr)^{|T\setminus T_{\rm int}|}
 \nonumber\\
 &\leq
 e^{2\theta\left(\abs{\gE(\gG_\xi)}-
        \abs{\gS_{\{\xi_0,\xi_\ell\}}(\gG_\xi)}\right)}
 \alpha_{\ell,{\rm B}}^{|T_{\rm int}|}.
 \label{eqn:shared-internal-vertex-payment}
\end{align}
Here \eqref{eqn:bipartite-upper-cap-condition} and
$R^2d\geq1$ imply
$\alpha_{\ell,{\rm B}}\leq1$, so the last factor on the first line
is at most one.  Each selected internal vertex retains a factor
$\alpha_{\ell,{\rm B}}$; the other selected vertices are accounted
for by the exponential factor.

To sum over covers, first choose the noninternal vertices.
Each such vertex has degree at least one, so their number is at
most their total degree
$2(\abs{\gE(\gG_\xi)}-
\abs{\gS_{\{\xi_0,\xi_\ell\}}(\gG_\xi)})$.  Consequently there
are at most
$2^{2(\abs{\gE(\gG_\xi)}-
\abs{\gS_{\{\xi_0,\xi_\ell\}}(\gG_\xi)})}$ choices of this
part of $T$.  For each such choice, we enlarge the allowed selections
inside the chains by treating both endpoints of every singleton
chain as already selected, without adding a weight for them.
This removes the covering requirements at the ends of the chain;
only edges joining two internal vertices still need to be covered.

A singleton chain of length $k\geq3$ has $k-1$ internal vertices.
Among the edges joining them, one can choose
$\lfloor(k-1)/2\rfloor$ pairwise disjoint edges.  A cover must select
at least one endpoint of each, hence at least that many internal
vertices.  There are at most $2^{k-1}$ selections altogether.
Since $\alpha_{\ell,{\rm B}}\leq1$, the sum of their weights is
therefore at most
\[
 2^{k-1}\alpha_{\ell,{\rm B}}^{\lfloor(k-1)/2\rfloor}
 \leq(8\alpha_{\ell,{\rm B}})^{\lfloor(k-1)/2\rfloor}.
\]
The second inequality uses
$k-1\leq3\lfloor(k-1)/2\rfloor$ for $k\geq3$.
The argument also applies when the chain endpoints coincide, since
its internal vertices still form a path.  Distinct chain interiors
are disjoint, and the exponents
$\lfloor(k-1)/2\rfloor$ sum to $\varsigma(\xi)$ by
\eqref{eqn:singleton-certificate-count}.  All other internal
vertices, including those on singleton chains of length at most
two, may be selected freely in this upper bound.  Each contributes
$1+\alpha_{\ell,{\rm B}}$, and there are at most $2\ell$ of them.
Thus
\begin{align}
 \sum_{\substack{T\subseteq\varphi(\gV(\gG_\xi))\\
                  T\text{ covers the singleton edges}}}
       \alpha_{\ell,{\rm B}}^{|T_{\rm int}|}
\leq &
 2^{2\left(\abs{\gE(\gG_\xi)}-
        \abs{\gS_{\{\xi_0,\xi_\ell\}}(\gG_\xi)}\right)}
 (8\alpha_{\ell,{\rm B}})^{\varsigma(\xi)}
 (1+\alpha_{\ell,{\rm B}})^{2\ell}
 \nonumber\\
\leq &
 2^{2\left(\abs{\gE(\gG_\xi)}-
        \abs{\gS_{\{\xi_0,\xi_\ell\}}(\gG_\xi)}\right)}
 (8\alpha_{\ell,{\rm B}})^{\varsigma(\xi)}
 e^{2\ell\alpha_{\ell,{\rm B}}}.
\end{align}

Multiplying this cover sum by the exponential factor in
\eqref{eqn:shared-internal-vertex-payment} and the absolute word
bound \eqref{eqn:shared-absolute-word-bound} bounds $|\E[SW]|$.
By \eqref{eqn:shared-singleton-chain-count}, the factor
$2^{e_1(\xi)}$ is at most
$4^{\varsigma(\xi)}
2^{2(\abs{\gE(\gG_\xi)}-
\abs{\gS_{\{\xi_0,\xi_\ell\}}(\gG_\xi)})}$.
Its first factor combines with the chain selections as
\[
 4^{\varsigma(\xi)}
 (8\alpha_{\ell,{\rm B}})^{\varsigma(\xi)}
 =(32\alpha_{\ell,{\rm B}})^{\varsigma(\xi)}.
\]
Its second factor combines with the count of noninternal selections
and the degree exponential to give
$(4e^\theta)^{2(\abs{\gE(\gG_\xi)}-
\abs{\gS_{\{\xi_0,\xi_\ell\}}(\gG_\xi)})}$.
Allowing an additional factor of two to match the constants stated
in the lemma, we obtain
\begin{align}
 |\E[SW]|
 &\leq2e^{2\ell\alpha_{\ell,{\rm B}}}
 (32\alpha_{\ell,{\rm B}})^{\varsigma(\xi)}
 (4e^\theta)^{2\left(\abs{\gE(\gG_\xi)}-
       \abs{\gS_{\{\xi_0,\xi_\ell\}}(\gG_\xi)}\right)}
 (R^2d)^{-\ell}
       \prod_{e\in\gE(\gG_\xi)}p_{\varphi(e)}(1-p_{\varphi(e)}).
 \label{eqn:bipartite-upper-cap-word-bound}
\end{align}
Condition \eqref{eqn:bipartite-upper-cap-condition} supplies the
required decay along singleton chains:
\[
 (32\alpha_{\ell,{\rm B}})^{\varsigma(\xi)}
 \leq(R^2d)^{-\varsigma(\xi)}.
\]
All factors other than the variance product in
\eqref{eqn:bipartite-upper-cap-word-bound} are uniform over the labelings.  Applying
Lemma~\ref{lem:weighted-variance-label-sum} to its variance product
therefore gives the singleton clause with
\[
 \widetilde K_\ell=
 \frac{2\max\{\rowvarmax,\colvarmax\}}
      {R^2}e^{2\ell\alpha_{\ell,{\rm B}}},
 \qquad \widetilde\Gamma_\ell=4e^\theta.
\]
Since $2e^{2\ell\alpha_{\ell,{\rm B}}}\geq1$, these parameters also
cover the plural bound, proving
item~\ref{item:bipartite-graph-upper-caps}.

\end{proof}

\paragraph{Verification over the full trace range.}
We verify the graph bounds up to powers of order $\sqrt d\log N$
by checking the cutoff probability
criteria throughout that range.  The forcing probabilities and
$\alpha_{\ell,{\rm B}}$ are nondecreasing in $\ell$, so the same
tail estimates and Chernoff parameter apply to every smaller power.

\begin{lemma}[Bipartite cutoffs verify the full trace range]
\label{lem:bipartite-cutoff-full-trace-verification}
Under \eqref{eqn:rectangular-uniform-sparsity-bound},
the sparsity conditions in Proposition~\ref{prop:nonbacktracking-radius-bounds}\ref{item:nonbacktracking-bipartite}, and condition
\ref{item:cutoff-upper-only} or \ref{item:cutoff-strong-band}, choose
$R=(\rowvarmax\colvarmax)^{1/4}$. There are fixed $b_{\rm tr},C_{\rm tr}>0$ and $N_0$, depending only on
the fixed model constants and the aspect-ratio bound, such that for
$N\geq N_0$ and every integer
$3\leq\ell\leq\lfloor b_{\rm tr}R\sqrt{d}\log N\rfloor$,
$\widehat{\mX}/R$ satisfies
\ref{eqn:bipartite-graph-condition} at scale $R^2d$, with
\begin{align}
 \widetilde\chi_\ell,\widetilde\Gamma_\ell&\leq C_{\rm tr},
 &\widetilde K_\ell&\leq C_{\rm tr}
                 \exp(2\ell d^{-3}).
\end{align}
Under strong bands, $\widetilde K_\ell\leq C_{\rm tr}$.
\end{lemma}

\begin{proof}[Proof of Lemma~\ref{lem:bipartite-cutoff-full-trace-verification}]
We verify the two probability criteria at the original degree scale
$d$ for graph moments normalized at scale $R^2d$.
The sparsity restriction implies $p_\star\leq C_0d/N=o(1)$,
so $p_\star\leq1/2$ for all sufficiently large $N$.
The variance formulas and \eqref{eqn:cutoff-profile-constant} give
\[
 1-p_\star\leq R^2
 \leq\frac{\sqrt{\overline{\drow}\,\overline{\dcol}}}{d}
 \leq C_{\rm sp}\leq C_0.
\]
Indeed, the lower bound follows as in the proof of
Lemma~\ref{lem:bipartite-graph-verification}, and the upper
bound uses
$\sqrt{\overline{\drow}\,\overline{\dcol}}
\leq\sqrt{nm}\,p_\star\leq C_{\rm sp}d$.
Increasing $C_{\rm NB}$ to at least two ensures $R^2d\geq1$.
Moreover, \eqref{eqn:bipartite-profile-constant-bounds} gives
\[
 \widetilde\chi_\ell\leq4C_0^2,\qquad
 \frac{\max\{\rowvarmax,\colvarmax\}}{R^2}\leq2C_0.
\]
Initially choose $0<b_{\rm tr}\leq(2\sqrt{C_0})^{-1}$.
For every power in the stated range, the sparsity restriction yields
\[
 2\ell+1\leq\sqrt d\log N+1
 \leq N^{(1-\delta_{\rm NB})/12}(\log N)^{3/4}+1=o(N).
\]
Thus $2\ell+1\leq N/4$ for all such powers once $N$ is sufficiently
large, uniformly over the allowed values of $d$.

For upper caps, fix
$\theta=\sqrt{\log d/(C_0d)}\leq1$ for the entire range of $\ell$.
The upper margins give $q_{\theta,{\rm B}}\leq d^{-4}$ by
\eqref{eqn:upper-only-retention-tail}.  Using
$e^\theta-1\leq2\theta$ and $p_\star\leq C_0d/N$, we obtain
\begin{align}
 \tfrac12(e^\theta-1)^2p_\star(2\ell+1)
 &\leq2\theta^2\frac{C_0d}{N}(2\ell+1)
 =\frac{2(2\ell+1)\log d}{N}
 \leq\tfrac12\log d.
\end{align}
Choose $C_{\rm NB}$ large enough that
$2\theta\leq\frac12\log d$ and
$32C_0d^{-2}\leq1$ whenever
$d\geq C_{\rm NB}$.  The definition
\eqref{eqn:bipartite-alpha-l-B} then gives
\begin{align}
 \alpha_{\ell,{\rm B}}
 &\leq d^{-4}\exp\!\left(\tfrac12\log d+\tfrac12\log d\right)
 =d^{-3},
 &32R^2d\,\alpha_{\ell,{\rm B}}
 &\leq32C_0d^{-2}\leq1.
 \label{eqn:bipartite-full-trace-upper-cap-control}
\end{align}
Since $L_{\rm r}=L_{\rm c}=0$, this verifies
\eqref{eqn:bipartite-upper-cap-condition}.
Lemma~\ref{lem:bipartite-graph-verification}\ref{item:bipartite-graph-upper-caps}
therefore supplies the graph moment bounds, with
\[
 \widetilde K_\ell\leq4C_0\exp(2\ell d^{-3}),\qquad
 \widetilde\chi_\ell\leq4C_0^2,\qquad
 \widetilde\Gamma_\ell=4e^\theta\leq4e.
\]
We retain the exponential factor in $\widetilde K_\ell$, since $d$
is allowed to remain constant as $N$ grows.

For strong bands, every nonvacuous margin is at least $w_N(d)$.
Apply \eqref{eqn:bipartite-band-margin-tail} at power $j_N(d)$ with
\[
 h=w_N(d)-2j_N(d)
 =\sqrt{2C_0d\,r_N(d)}+\tfrac23r_N(d)>0.
\]
This choice satisfies
\[
 h^2-2r_N(d)(C_0d+h/3)
 =\tfrac23r_N(d)\sqrt{2C_0d\,r_N(d)}\geq0.
\]
Using the degree budget $\max\{\overline{\drow},\overline{\dcol}\}
\leq C_0d$ in the tail bound gives
\[
 \varepsilon_{j_N(d),{\rm B}}
 \leq2N\exp\!\left[-\frac{h^2}{2(C_0d+h/3)}\right]
 \leq2Ne^{-r_N(d)}=(4d)^{-j_N(d)}.
\]
To obtain the criterion at scale $R^2d$, note that
\[
 \ell\leq b_{\rm tr}R\sqrt d\log N
 \leq\tfrac12\sqrt d\log N\leq\tfrac12j_N(d).
\]
Increase $C_{\rm NB}$, if necessary, so that $C_{\rm NB}\geq C_0$.
Then $R^2\leq C_0\leq d$, and hence
$\log(4R^2d)\leq2\log(4d)$.
Consequently,
$\ell\log(4R^2d)\leq j_N(d)\log(4d)$.
Since forcing probabilities are nondecreasing in the power, we conclude
\begin{align}
 \varepsilon_{\ell,{\rm B}}
 &\leq\varepsilon_{j_N(d),{\rm B}}
 \leq(4d)^{-j_N(d)}
 \leq(4R^2d)^{-\ell}.
 \label{eqn:bipartite-full-trace-strong-band-control}
\end{align}
This verifies \eqref{eqn:bipartite-forced-band-condition}, so
Lemma~\ref{lem:bipartite-graph-verification}\ref{item:bipartite-graph-general-bands}
gives $\widetilde K_\ell\leq2C_0$,
$\widetilde\chi_\ell\leq4C_0^2$, and $\widetilde\Gamma_\ell=2$.

Taking $C_{\rm tr}\geq\max\{4C_0,4C_0^2,4e\}$ gives all the
stated parameter bounds throughout the required range.
\end{proof}

\subsection{Column rescaling}
\label{subsec:column-rescaling}

The column contraction in \eqref{eqn:rescaled-column-matrix}
introduces dependence between entries in the same column. We verify
the graph moment bounds by combining lower tails for the column
variances with finite differences along singleton chains.

\begin{lemma}[Graph moment verification for the rescaled matrix]
\label{lem:rescaled-column-moment-verification}
Under the hypotheses of
Proposition~\ref{prop:nonbacktracking-radius-bounds}\ref{item:nonbacktracking-rescaled},
let $\mZ$ be the matrix in \eqref{eqn:rescaled-column-matrix} and choose
\[
 R=(\rowvarmax\colvarmin)^{1/4},\qquad s=R^2d.
\]
There are fixed constants $A,c,C>0$, depending only on the fixed
hypothesis constants, such that, for all sufficiently large $N$ and
every integer $3\leq\ell\leq\lfloor c\sqrt d\log N\rfloor$,
$\mZ/R$ satisfies \ref{eqn:bipartite-graph-condition} at scale
$s$, with
\begin{align}
 \widetilde K_\ell&\leq C\exp(C\ell\eta),
 &\widetilde\chi_\ell&\leq C,
 &\widetilde\Gamma_\ell&=1,
 &\eta&=A\sqrt{\log d/d}.
 \label{eqn:rescaled-column-moment-parameters}
\end{align}
\end{lemma}

\begin{proof}
\emph{Comparison weights.}
Set $v_j=\sum_i\E|\emX_{ij}|^2$ and $z_j=\colvarmin/v_j$.
The comparison weights $w_{ij}=z_jp_{ij}(1-p_{ij})/d$ have column
sums $\colvarmin$, row sums at most $\rowvarmax$, and maximum
$C/N$. All of $\rowvarmax,\colvarmin,R$, and $v_j$ are bounded
above and away from zero by fixed profile constants. We prove the
graph bounds uniformly over the stated range of powers.

\smallskip\noindent\emph{Lower tails outside the word support.}
Fix an injectively labelled word, with row set $I$ and column set $J$.
For $i\notin I$ and $j\in J$, let
\[
 C_{ij}=\widetilde{\emA}_{ij}
 \prod_{k\in J\setminus\{j\}}(1-\widetilde{\emA}_{ik})
 \indi{\sum_{k\notin J}\widetilde{\emA}_{ik}\le U_{\rm r}-1},
 \qquad T_j=\frac{1-2p_\star}{d}\sum_{i\notin I}C_{ij}.
\]
Such an edge survives the row cutoff, so
$T_j\le\sum_{i\notin I}U_i|\emX_{ij}|^2$.
Each row vector $(C_{ij})_{j\in J}$ has at most one nonzero entry,
and these vectors are independent across rows. Its success
probabilities $q_{ij}$ satisfy
$q_{ij}\ge p_{ij}(1-C|J|p_\star-Cd^{-10})$:
the row-tail estimate follows by Chernoff optimization at
$\theta=5\sqrt{\log d/(C_0d)}$, since the exponent is at most
$-(25/2+o(1))\log d$. Hence
$\E T_j\ge v_j-C(\ell p_\star+p_\star+d^{-10})$.
The sparsity range gives $\ell p_\star+p_\star=o(\eta)$.
For $t_j\ge0$, the elementary inequality
\[
 \E e^{-\sum_jt_jC_{ij}}
 =1+\sum_jq_{ij}(e^{-t_j}-1)
 \le\prod_j\bigl(1+q_{ij}(e^{-t_j}-1)\bigr)
\]
allows separate lower-tail Chernoff parameters in every column.
Choosing $A$ large enough therefore gives, for every $B\subseteq J$,
\begin{align}
 \P(T_j<v_j-\eta\text{ for all }j\in B)&\le d^{-10|B|},
 &\E\prod_{j\in J}(1+C E_j)&\le(1+Cd^{-10})^{|J|},
 \label{eqn:rescaled-column-joint-lower-tails}
\end{align}
where $E_j=\indi{T_j<v_j-\eta}$. These variables involve only rows
outside $I$, so conditioning or tilting the word edges does not
change these estimates.

\smallskip\noindent\emph{Boundary rows and conditional column integration.}
First condition every Bernoulli outside the word support. A word row
$i$ has support degree $k_i$ and remaining degree $K_i$. Its survival indicator is identically one when
$K_i\le U_{\rm r}-k_i$, identically zero when $K_i>U_{\rm r}$,
and can vary only on
$\mathcal B_i=\{U_{\rm r}-k_i<K_i\le U_{\rm r}\}$.
The variables $K_i$ are independent across word rows and independent
of all $T_j$. The same Chernoff calculation gives
\begin{align}
 \P(\mathcal B_i)&\le d^{-10}e^{\theta k_i}.
 \label{eqn:rescaled-column-row-boundary}
\end{align}
Let $B$ be the set of boundary rows. Fix also their support-edge
variables and discard configurations in which a word row is deleted.
Every word row then survives under every change of the remaining
support variables. Consequently, in each column factor we may write
\[
 Q_j=\sum_{i\notin I}U_i|\emX_{ij}|^2+\sum_{i\in I}|\emX_{ij}|^2.
\]
Each remaining variable affects only its own column, by an increment
$(1-2p_{ij})/d\in[0,1/d]$. Conditional integration thus factors by
columns. This is a conditional factorization of the word, not an
independence assertion about the normalized entries.

\smallskip\noindent\emph{Finite differences and singleton-chain payments.}
Write $g(x)=\min\{1,\colvarmin/x\}$, with $g(0)=1$. Its a.e. derivative
on $[q,\infty)$ takes values in
$[-\colvarmin/\max(\colvarmin,q)^2,0]$.
For $x\ge q$ and $h,k\in[0,1/d]$,
\begin{align}
 |\Delta_hg(x)|&\le\frac{\colvarmin h}{\max(\colvarmin,q)^2},
 &|\Delta_h\Delta_kg(x)|&\le
       \frac{\colvarmin\min(h,k)}{\max(\colvarmin,q)^2}.
 \label{eqn:rescaled-column-sharp-differences}
\end{align}
For the second estimate, integrate the difference of two derivatives;
their range has the displayed length, including across the clipping
point. In particular, no factor two is needed.
An internal degree-two singleton column has factor $g(Q_j)$.
If at least one incident row is not in $B$, apply
$\E[(A-p)F(A)]=p(1-p)(F(1)-F(0))$ on its one or two free edges.
Since $Q_j\ge T_j$, the resulting difference coefficient is at most
\begin{align}
 \frac{z_j^2}{\colvarmin d}(1+C\eta)(1+CE_j).
 \label{eqn:rescaled-column-singleton-payment}
\end{align}
Any frozen singleton edge contributes at most $2p_{ij}(1-p_{ij})$
when subsequently integrated absolutely; this factor two is charged
to its boundary row. For a degree-two plural column, discard excess
powers of $f_j\le1$ to leave $g(Q_j)^2$, which is at most
$z_j^2(1+C\eta)(1+CE_j)$. Its edge moments satisfy
$\E|A-p|^k\le p(1-p)$ for $k\ge2$.

On each singleton chain pair consecutive internal vertices, as in
\eqref{eqn:singleton-certificate-count}. Select the column in each
pair and match it to the row in that pair. There are exactly
$\varsigma$ selected columns, with distinct matched internal rows.
A selected column lacks the payment
\eqref{eqn:rescaled-column-singleton-payment} only if both incident
rows belong to $B$, and hence its matched row belongs to $B$.
There are therefore at most $|B\cap S_{\rm r}|$ missing payments,
where $S_{\rm r}$ denotes the internal degree-two rows.
Restoring them costs at most $(Cd)^{|B\cap S_{\rm r}|}$, including
the factors two from their frozen singleton edges.
An internal boundary row has probability at most $2d^{-10}$ by
\eqref{eqn:rescaled-column-row-boundary}. Core rows have total support
degree $O(g+1)$, where $g$ is the support genus, by
Lemma~\ref{lem:reduced-graph-code} and \eqref{eqn:core-size}.
Their boundary choices and the factors $e^{\theta k_i}$ cost only
$C^{g+1}$. Summing over the boundary sets thus costs
\[
 C^{g+1}(1+Cd^{-9})^{|S_{\rm r}|}.
\]
Here the probability of an exact boundary set is bounded by the
product of the selected row-tail probabilities; its joint estimate
with the column events factors because the corresponding row sets
are disjoint.

The unselected singleton columns number at most one per core chain.
They and the noninternal columns cost $C^{g+1}$: restore a factor
$z_j$ per incident support edge using $z_j\ge c>0$, and bound their
remaining singleton moments absolutely. The core-degree bound again
controls every such constant. Thus, writing $W_Z$ for the word
product before normalization by $R$, the preceding estimates and
\eqref{eqn:rescaled-column-joint-lower-tails} give
\begin{align}
 |\E W_Z|&\le C^{g+1}e^{C\ell\eta}d^{-\ell}
 \prod_{(i,j)\in\gE(\gG_\xi)}p_{ij}(1-p_{ij})z_j
 \begin{cases}
  1,&\xi\text{ plural},\\
  (\colvarmin d)^{-\varsigma},&\xi\text{ singleton}.
 \end{cases}
 \label{eqn:rescaled-column-word-bound}
\end{align}
No factor exponential in the number of internal vertices is hidden
in $C^{g+1}$; their remaining factors are in $e^{C\ell\eta}$.

\smallskip\noindent\emph{graph moment bounds.}
Apply Lemma~\ref{lem:weighted-variance-label-sum} to $w_{ij}/R^2$.
The row and column sums are bounded by
$\sqrt{\rowvarmax/\colvarmin}$ and
$\sqrt{\colvarmin/\rowvarmax}$, respectively. These bounds have
product one, so the lemma gives the plural graph bound at scale $s$.
For singleton words the additional scale factor is
\[
 s^{|\gE(\gG_\xi)|-\ell}(\colvarmin d)^{-\varsigma}
 =s^{|\gE(\gG_\xi)|-\ell-\varsigma}
       (\rowvarmax/\colvarmin)^{\varsigma/2}
 \le s^{|\gE(\gG_\xi)|-\ell-\varsigma}.
\]
Absorb $C^{g+1}$ into $\widetilde K_\ell$ and
$\widetilde\chi_\ell^g$ to obtain
\eqref{eqn:rescaled-column-moment-parameters}.
\end{proof}

\subsection{Hermitian degree cutoffs}
\label{subsec:hermitian-moment-verification}

The independent variables in the original Hermitian model are the
unordered edges $\{u,v\}$, $u<v$; each determines both symmetric
entries. We verify entry and graph moments over the required
power range using the degree estimates in
Section~\ref{subsec:hermitian-cutoff-retention}.

\paragraph{Entry moments.}
\begin{lemma}[Hermitian entry moment bounds]
\label{lem:projected-hermitian-sparse-bounds}
Under \eqref{eqn:uniform-sparsity-bound}, the zero-padded matrix
$\widehat{\mH}$ satisfies Definition~\ref{def:sparse-random-matrix}
at scale $d$ with moment constant $2C_{\rm sp}$.
\end{lemma}

\begin{proof}
The bound $|\widehat{\emH}_{uv}|\leq|\emA_{uv}-p_{uv}|/\sqrt d$
gives the entry maximum $d^{-1/2}$. By
\eqref{eqn:shared-bernoulli-centered-moments} and
$p_\star\leq C_{\rm sp}d/N$, for every $k\geq1$,
\[
 \E|\widehat{\emH}_{uv}|^k
 \leq\frac{2p_{uv}}{d^{k/2}}
 \leq\frac{2C_{\rm sp}}{Nd^{(k-2)/2}}.
\]
\end{proof}

\paragraph{graph moment verification.}
The degree estimates in Section~\ref{subsec:hermitian-cutoff-retention}
supply the probabilities needed for singleton words.  The remaining change from the bipartite calculation
is the sum over vertex labels: the Hermitian variance weights have
row sums at most one.

The graph moment proof will use either the general-band condition
\begin{align}
 \varepsilon_{\ell,{\rm H}}&\leq(4d)^{-\ell},
 \label{eqn:hermitian-verification-forced-condition}
\end{align}
or the upper-only condition
\begin{align}
 L=0,\qquad32d\alpha_{\ell,{\rm H}}\leq1.
 \label{eqn:hermitian-verification-cap-condition}
\end{align}
As in the bipartite model, $\alpha_{\ell,{\rm H}}$ accounts for
shared edges in joint degree events.
Define
\begin{align}
 \chi_\ell\coloneqq
 1\vee\left(\frac Nd\max_{u<v}p_{uv}(1-p_{uv})\right).
 \label{eqn:hermitian-verification-chi}
\end{align}

\begin{lemma}[Hermitian graph moment verification]
\label{lem:hermitian-graph-verification}
Assume \eqref{eqn:uniform-sparsity-bound}, $d\geq1$, and let
$\ell\geq3$ be an integer.  For $\widehat{\mH}$, both lines of
\ref{eqn:hermitian-graph-condition} hold at scale $d$, uniformly
over the canonical classes, under either of the following conditions.
In both cases use $\chi_\ell$ from \eqref{eqn:hermitian-verification-chi}.
\begin{enumerate}[label=\textup{(\roman*)},leftmargin=2.2em,itemsep=0.6em]
\item Under \eqref{eqn:hermitian-verification-forced-condition}, take
$K_\ell=1$ and $\Gamma_\ell=2$.
\item Under \eqref{eqn:hermitian-verification-cap-condition}, take
$K_\ell=2e^{2\ell\alpha_{\ell,{\rm H}}}$ and $\Gamma_\ell=4e^\theta$.
\end{enumerate}
\end{lemma}

\begin{proof}[Proof of Lemma~\ref{lem:hermitian-graph-verification}]
Use the word product, survival indicator, and tilted law from the
proof of Lemma~\ref{lem:bipartite-graph-verification}, with
normalization $d^{-\ell}$ and unordered support edges. A diagonal
edge or an edge of zero variance contributes zero. For every other
word, the absolute moment bound, singleton cancellation, and
upper-cap cover argument apply unchanged: each unordered edge is
independent of the others and has at most two selected endpoints.

The variance label sum is the only model-specific change.
The weights $p_{uv}(1-p_{uv})/d$ have row sums at most one and
maximum at most $\chi_\ell/N$, with $1\leq\chi_\ell\leq C_0$.
The Hermitian part of Lemma~\ref{lem:weighted-variance-label-sum} gives
\begin{align}
 &\sum_{\varphi\text{ injective}}c(\varphi\xi)d^{-\ell}
       \prod_{e\in\gE(\gG_\xi)}p_{\varphi(e)}(1-p_{\varphi(e)})
 \leq N^2
    \left(\frac{\chi_\ell}{N}\right)^{\abs{g(\gG_\xi)}}
    d^{\abs{\gE(\gG_\xi)}-\ell}.
 \label{eqn:hermitian-verification-tree-sum}
\end{align}
This proves the plural bound with $K_\ell=1$.

For general bands, \eqref{eqn:shared-forced-singleton-word-bound}
and \eqref{eqn:shared-singleton-chain-count}, with
\eqref{eqn:hermitian-verification-forced-condition}, give
\begin{align}
 \varepsilon_{\ell,{\rm H}}2^{e_1(\xi)}
 &\leq
 2^{2\left(\abs{\gE(\gG_\xi)}-
       \abs{\gS_{\{\xi_0,\xi_\ell\}}(\gG_\xi)}\right)}
 d^{-\varsigma(\xi)}.
 \label{eqn:hermitian-verification-singleton-payment}
\end{align}
Multiplying the variance label sum by this factor gives
$\Gamma_\ell=2$.

For upper caps, the weighted Chernoff estimate
\eqref{eqn:shared-joint-upper-tail} and the ensuing cover count give
\eqref{eqn:bipartite-upper-cap-word-bound} with
$R^2d$, $\alpha_{\ell,{\rm B}}$ replaced by
$d$, $\alpha_{\ell,{\rm H}}$. Condition
\eqref{eqn:hermitian-verification-cap-condition} supplies
$(32\alpha_{\ell,{\rm H}})^{\varsigma(\xi)}\leq d^{-\varsigma(\xi)}$.
The variance label sum then proves the claimed
$K_\ell=2e^{2\ell\alpha_{\ell,{\rm H}}}$ and $\Gamma_\ell=4e^\theta$,
which also cover the plural words.
\end{proof}

\paragraph{Verification over the full trace range.}
Both forcing probabilities and upper-cap corrections are nondecreasing
in $\ell$.  It therefore suffices to check either condition at the
largest required power, using the same $\theta$ throughout the
upper-cap alternative. We verify these moment criteria under the
cutoff regimes in Section~\ref{sec:regularization-thresholds}.

\begin{lemma}[Hermitian cutoffs verify the full trace range]
\label{lem:hermitian-cutoff-full-trace-verification}
Assume \eqref{eqn:uniform-sparsity-bound},
\eqref{eqn:rectangular-edge-sparsity-range}, and either
condition~\ref{item:cutoff-upper-only} or
condition~\ref{item:cutoff-strong-band}.
There are fixed $b_{\rm tr},C_{\rm tr}>0$ and $N_0$, depending only
on the fixed hypothesis constants, such that for $N\geq N_0$,
$\widehat{\mH}$ satisfies \ref{eqn:hermitian-graph-condition}
at sparsity scale $d$ throughout
$3\leq\ell\leq\lfloor b_{\rm tr}\sqrt d\log N\rfloor$, with
\begin{align}
 \chi_\ell,\Gamma_\ell\leq C_{\rm tr},\qquad
 K_\ell\leq C_{\rm tr}\exp(2\ell d^{-3}).
\end{align}
Under strong bands, $K_\ell\leq C_{\rm tr}$.
\end{lemma}

\begin{proof}[Proof of Lemma~\ref{lem:hermitian-cutoff-full-trace-verification}]
Apply the degree-tail estimates from the proof of
Lemma~\ref{lem:bipartite-cutoff-full-trace-verification} with
$R=1$ and sparsity scale $d$ in \eqref{eqn:hermitian-sparsity-scale}. 
For upper caps, \eqref{eqn:hermitian-upper-only-retention-tail} gives
$q_{\theta,{\rm H}}\leq d^{-4}$ at
$\theta=\sqrt{\log d/(C_0d)}$. The same shared-edge correction yields
$\alpha_{\ell,{\rm H}}\leq d^{-3}$ and
$32d\alpha_{\ell,{\rm H}}\leq1$ throughout the stated trace range.
Lemma~\ref{lem:hermitian-graph-verification} therefore gives
$K_\ell\leq2e^{2\ell d^{-3}}$, $\chi_\ell\leq C_0$, and
$\Gamma_\ell\leq4e$.

For strong bands, \eqref{eqn:hermitian-verification-margin-tail}
at $j_N(d)$ gives
$\varepsilon_{j_N(d),{\rm H}}\leq(4d)^{-j_N(d)}$;
monotonicity then verifies
\eqref{eqn:hermitian-verification-forced-condition} for every
required power, with $K_\ell=1$ and $\Gamma_\ell=2$.
\end{proof}

%%%%%%%%%%%%%%%%%%%%%%%%%%%%%%%%%%%%%%%%%%%%%%%%%%%%%%%%%%%%%%%%%%%%%%
%%%%%%%%%%%%%%%%%%%%%%%  Technical Lemmas  %%%%%%%%%%%%%%%%%%%%%%%%%%%
%%%%%%%%%%%%%%%%%%%%%%%%%%%%%%%%%%%%%%%%%%%%%%%%%%%%%%%%%%%%%%%%%%%%%%

\section{Technical Lemmas}
\label{sec:technical-lemmas}

\begin{lemma}[Chernoff's inequality, {\cite[Section~2.3]{vershynin2018high}}]
\label{lem:Chernoff}
Let $X_i$ be independent Bernoulli variables with means $p_i$, and choose
$S=\sum_iX_i$ and $\mu=\E S$. For $\mu>0$,
\begin{align}
 \P(S\geq t)&\leq e^{-\mu}(e\mu/t)^t,
 &&t\geq\mu,\nonumber\\
 \P(S\leq(1-\delta)\mu)&\leq e^{-\delta^2\mu/2},
 &&0<\delta<1.\nonumber
\end{align}
We also use the exponential form: for deterministic $a_i\geq0$ and
$\theta\in\R$,
\[
 \E\exp\!\left(\theta\sum_i a_iX_i\right)
 \leq\exp\!\left(\sum_i p_i(e^{\theta a_i}-1)\right).
\]
In particular, for $\theta>0$ and $t\in\R$,
\[
 \P\left(\sum_i a_iX_i\geq t\right)
 \leq\exp\!\left[-\theta t+\sum_i p_i(e^{\theta a_i}-1)\right].
\]
If $\mu=0$, then $S=0$ almost surely.
\end{lemma}

\begin{proof}[Proof of the exponential form]
Since $X_i$ equals $1$ with probability $p_i$ and $0$ with probability
$1-p_i$, for every $\theta\in\R$ we have
\[
 \E e^{\theta a_iX_i}
 =(1-p_i)+p_i e^{\theta a_i}
 =1+p_i(e^{\theta a_i}-1)
 \leq\exp\!\left(p_i(e^{\theta a_i}-1)\right).
\]
The inequality follows from $1+x\leq e^x$ for every real $x$, so it
also applies when $\theta<0$. Independence lets us factor the
expectation of the product, giving
\[
 \begin{aligned}
 \E\exp\!\left(\theta\sum_i a_iX_i\right)
 &=\E\prod_i e^{\theta a_iX_i}
 =\prod_i\E e^{\theta a_iX_i}\\
 &\leq\prod_i\exp\!\left(p_i(e^{\theta a_i}-1)\right)
 =\exp\!\left(\sum_i p_i(e^{\theta a_i}-1)\right).
 \end{aligned}
\]
For $\theta>0$, Markov's inequality applied to
$\exp(\theta\sum_i a_iX_i)$ at threshold $e^{\theta t}$ gives the
stated exponential tail bound.
\end{proof}

\begin{lemma}[Bennett's inequality, {\cite[Theorem $2.9.2$]{vershynin2018high} }]\label{lem:Bennett}
    Let $\rX_1,\dots, \rX_n$ be independent random variables. Assume that $|\rX_i - \E \rX_i| \leq K$ almost surely for every $i$. Then for any $t>0$, we have
    \begin{align}
        \P \Bigg( \sum_{i=1}^{n} (\rX_i - \E \rX_i) \geq t \Bigg) \leq \exp \Bigg( - \frac{\sigma^2}{K^2} \cdot h \bigg( \frac{Kt}{\sigma^2} \bigg)\Bigg)\,, \notag
    \end{align}
    where $\sigma^2 = \sum_{i=1}^{n}\Var(\rX_i)$ is the variance of the
    sum and $h(t) \coloneqq(1+t)\log(1+t)-t$.
    Furthermore, define $\rY_i=-\rX_i$ and apply the inequality above;
    for any $t>0$, we then have
    \begin{align}
        \P \Bigg( \sum_{i=1}^{n} (\rX_i - \E \rX_i) \leq -t \Bigg) \leq \exp \Bigg( - \frac{\sigma^2}{K^2} \cdot h \bigg( \frac{Kt}{\sigma^2} \bigg)\Bigg)\,.
    \end{align}
For either sign and any deterministic $V>0$ with $\sigma^2\leq V$,
we use the standard quadratic consequence
\begin{align}
 \P\left(\pm\sum_i(\rX_i-\E \rX_i)\geq t\right)
 &\leq\exp\!\left[-\frac{t^2}{2(V+Kt/3)}\right],
 \qquad t>0,
 \label{eqn:bennett-quadratic-tail}
\end{align}
which uses $h(u)\geq u^2/[2(1+u/3)]$ for $u\geq0$.
\end{lemma}

\begin{lemma}[Weyl's inequality and spectral counting, \cite{Weyl1912DasAV}]
    \label{lem:weyl}
    Let $\rmA, \rmE \in \C^{m \times n}$. Then
    \begin{align}
        |\sigma_i(\rmA+\rmE)-\sigma_i(\rmA)|
        \leq \|\rmE\|,
        \qquad 1\leq i\leq\min\{m,n\}.
    \end{align}
    If $m=n$ and $\rmA,\rmE\in\C^{n\times n}$ are Hermitian, then
    \begin{align}
        |\lambda_i(\rmA+\rmE)-\lambda_i(\rmA)|
        \leq\|\rmE\|,
        \qquad 1\leq i\leq n.
    \end{align}
\end{lemma}

\begin{lemma}[Schur test in matrix version, \cite{schur1911bemerkungen}]\label{lem:schur_test}
    Let $\{a_j\}_{j=1}^{n}$ and $\{b_l\}_{l=1}^{m}$ be two sequences of positive real numbers, and let $\lambda, \mu$ be positive real numbers. For matrix $\rmX \in \C^{n \times m}$, we have $\|\rmX\| \leq \sqrt{\lambda \mu}$ if
    \begin{align}
        \sum_{j=1}^{n} |\ermX_{jl}| \cdot a_{j} \leq \lambda \, b_{l}, \,\, \forall l\in [m], \quad \textnormal{ and }\quad \sum_{l=1}^{m} |\ermX_{jl}| \cdot b_{l} \leq \mu\, a_{j}, \,\, \forall j \in [n].
    \end{align}
    As a consequence, we take $a_{j} = b_{l} = 1$ and choose $\lambda$ (resp.\ $\mu$) as the maximum column (resp.\ row) sum, then
    \begin{align}
        \|\rmX\| \leq \bigg( \max_{j\in [n]} \sum_{l=1}^{m} |\ermX_{jl}| \bigg)^{1/2} \cdot \bigg( \max_{l\in [m]}\sum_{j=1}^{n} |\ermX_{jl}| \bigg)^{1/2}.
    \end{align}
\end{lemma}

\Needspace{14\baselineskip}
\begin{lemma}[Schur-complement determinant and positivity]\label{lem:schur-complement}
Let
\[
 \mM=\begin{bmatrix}\mA&\mB\\\mC&\mD\end{bmatrix},
 \qquad
 \mS=\mD-\mC\mA^{-1}\mB,
\]
where $\mA\in\C^{n\times n}$ is invertible,
$\mD\in\C^{m\times m}$, and the other blocks have compatible
dimensions. Then
\[
 \det\mM=\det\mA\,\det\mS.
\]
If, in addition, $\mM$ is Hermitian and $\mA\succ0$, then
\[
 \mM\succ0\quad\Longleftrightarrow\quad\mS\succ0.
\]
\end{lemma}
\begin{proof}[Proof of Lemma~\ref{lem:schur-complement}]
The factorization
\[
 \mM=
 \begin{bmatrix}\id_n&0\\\mC\mA^{-1}&\id_m\end{bmatrix}
 \begin{bmatrix}\mA&0\\0&\mS\end{bmatrix}
 \begin{bmatrix}\id_n&\mA^{-1}\mB\\0&\id_m\end{bmatrix}
\]
gives the determinant identity. When $\mM$ is Hermitian, the first
and last factors are adjoints of each other, so the same
factorization gives the positivity criterion by congruence.
\end{proof}

\end{document}